\RequirePackage{fix-cm}

\documentclass[
	,12pt%
	,pagesize%
	,headings=big%
	,paper=a4%
	,parskip=false%
	,headsepline=true%
	,abstracton
	,toc=bibliography
]{scrartcl}
\addtokomafont{sectioning}{\normalfont\bfseries}
\setkomafont{title}{\normalfont\LARGE}
\setkomafont{subtitle}{\normalfont\Large}
\addtokomafont{pageheadfoot}{\scshape\small}

\usepackage{csquotes}
\usepackage[english]{babel}
\usepackage{graphicx}
\graphicspath{{Pictures/}{/}}
\usepackage[usenames,dvipsnames]{xcolor}
\usepackage{tikz-cd}
\usetikzlibrary{backgrounds}

\usepackage[final]{fixme}
\usepackage[nointlimits]{amsmath}
\usepackage{amssymb,mathrsfs,mathtools}
\usepackage{txfonts}
\usepackage{upgreek}
\usepackage{braket}
\usepackage{cancel}
\usepackage{siunitx}
\usepackage{aliascnt} 
\usepackage[amsmath,thmmarks,hyperref,thref]{ntheorem}
\usepackage[expansion=true,protrusion=true]{microtype}
\usepackage{xkeyval}
\usepackage{multirow}

\usepackage[
	backend=bibtex%
	,isbn=false%
	,doi=false%
	,url=false%
	]{biblatex}
\usepackage[final,%
	pdftex,%
	bookmarks,%
	bookmarksdepth=3,%
	breaklinks=true,%
	colorlinks=true,%
	urlcolor=RoyalBlue,%
	linkcolor=RoyalBlue,%
	citecolor=ForestGreen,%
]{hyperref}%
\usepackage[all]{hypcap}

\addto\extrasenglish{%
}

\newcommand{\myincludegraphics}[2]{\begin{tikzpicture}
    \node[inner sep=0pt] (fig) at (0,0) {\includegraphics{#1}};
	\node[above right= 1ex] at (fig.south west) {\footnotesize\textbf{(#2)}};    
\end{tikzpicture}
}

\DeclareDocumentCommand{\Hess}{ O{} }{\operatorname{Hess}_{#1}}

\DeclareDocumentCommand{\sobo}{ m m o o }{%
	\expandafter\ifstrequal\expandafter{#1}{0}{%
		L^{#2}\IfValueT{#3}{(#3 \IfValueT{#4}{; #4} )}
	}{%
		W^{#1,#2}\IfValueT{#3}{(#3 \IfValueT{#4}{; #4} )}
	}
}
\DeclareDocumentCommand{\soboo}{ m m o o }{%
\expandafter\ifstrequal\expandafter{#1}{1}{%
	W_{\mathrlap{0}}{\mathstrut}^{1,#2}\IfValueT{#3}{(#3 \IfValueT{#4}{; #4} )}
}%
{%
	\expandafter\ifstrequal\expandafter{#1}{0}{%
		L^{#2}\IfValueT{#3}{(#3 \IfValueT{#4}{; #4} )}
	}{%
		(W^{#1,#2}\!\cap\!W_{\mathrlap{0}}{\mathstrut}^{1,#2})\IfValueT{#3}{(#3 \IfValueT{#4}{; #4} )}
	}
}}
\newcommand{\ext}{\operatorname{ext}}

\usepackage{xparse}
\DeclareDocumentCommand{\RM}{  }{M}

\DeclareDocumentCommand{\RX}{ O{\,} O{} O{} }{X_{\!#2}#3#1}
\DeclareDocumentCommand{\RH}{ O{\,} O{} O{} }{H_{\!#2}#3#1}
\DeclareDocumentCommand{\RY}{ O{\,} O{} O{} }{Y_{\!#2}#3#1}
\DeclareDocumentCommand{\RT}{ O{\,} O{} O{} }{T_{\!#2}#3#1}
\DeclareDocumentCommand{\RZ}{ O{\,} O{} O{} }{Z_{\!#2}#3#1}
\DeclareDocumentCommand{\T}{ O{} O{} }{T_{\!#1}#2}

\DeclareDocumentCommand{\Ri}{ O{} }{i_{#1}}
\DeclareDocumentCommand{\Rj}{ O{} }{j_{#1}}
\DeclareDocumentCommand{\RI}{ O{} }{I_{#1}}
\DeclareDocumentCommand{\RJ}{ O{} }{J_{#1}}
\DeclareDocumentCommand{\RK}{ O{} }{K_{#1}}
\DeclareDocumentCommand{\Rg}{ O{} }{g_{#1}}

\DeclareDocumentCommand{\RP}{ O{} }{P_{#1}}
\DeclareDocumentCommand{\RQ}{ O{} }{Q_{#1}}
\DeclareDocumentCommand{\RPsi}{ O{} }{\varPsi_{#1}}
\DeclareDocumentCommand{\RMor}{ }{L}

\usepackage{xspace}
\newcommand{\UnnamedSpace}{para-Hilbert space\xspace}
\newcommand{\UnnamedSpaces}{para-Hilbert spaces\xspace}
\newcommand{\UnnamedMorphism}{para-Hilbert morphism\xspace}
\newcommand{\UnnamedMorphisms}{para-Hilbert morphisms\xspace}
\newcommand{\UnnamedBundle}{para-Hilbert bundle\xspace}
\newcommand{\UnnamedBundles}{para-Hilbert bundles\xspace}

\newcommand{\UnnamedManifold}{para-Riemannian manifold\xspace}
\newcommand{\UnnamedManifolds}{para-Riemannian manifolds\xspace}

\newcommand{\CUnnamedSpaces}{Para-Hilbert Spaces\xspace}

\newcommand{\CUnnamedManifolds}{Para-Riemannian Manifolds\xspace}

\newcommand{\RieszH}{H}

\newcommand{\ConnComp}{\operatorname{Conn}}

\newcommand{\averageima}{\top}
\newcommand{\averagedom}{\perp}

\newcommand{\lot}{\text{l.o.t.}}
\newcommand{\Willmore}{\cW}

\newcommand{\ElasticEnergy}{\cF}
\newcommand{\ReferenceMetric}{G}

\newcommand{\nospaceperiod}{\makebox[0pt][l]{\,.}}

\newcommand{\lcola}{Black}

\newcommand{\lcold}{Black}
\newcommand{\lcole}{Black}

\newcommand{\vol}{{\on{vol}}}

\newcommand{\res}{{\on{res}}}

\newcommand{\Imm}{\on{Imm}}

\newcommand{\AmbSpace}{{\R^m}}
\newcommand{\AmbMetric}{g_0}

\newcommand{\ConfSpace}{\cC}

\newcommand{\invisible}[1]{}

\let\originalleft\left
\let\originalright\right
\renewcommand{\left}{\mathopen{}\mathclose\bgroup\originalleft}
\renewcommand{\right}{\aftergroup\egroup\originalright}

\newcommand{\qand}{\quad \text{and} \quad}
\newcommand{\qor}{\quad \text{or} \quad}

\newcommand{\Curve}{\gamma}

\newcommand{\loc}{{\on{loc}}}

\newcommand{\pull}{\#}

\DeclareMathOperator{\LandO}{\on{O}}

\newcommand{\transp}{^{\mathsf{T\!}}}
\newcommand{\adj}{^{*\!}}
\newcommand{\dual}{'^{\!}}
\newcommand{\ddual}{''^{\!}}
\newcommand{\pinv}{^{\dagger\!}}

\newcommand{\cA}{{\mathcal{A}}}

\newcommand{\cC}{{\mathcal{C}}}

\newcommand{\cF}{{\mathcal{F}}}

\newcommand{\cT}{{\mathcal{T}}}

\newcommand{\cW}{{\mathcal{W}}}

\DeclareMathOperator{\II}{I\!\;\!I}

\newcommand{\dd}{{\on{d}}}

\newcommand{\at}{|}

\DeclareMathOperator{\Sym}{Sym}

\newcommand{\pd}{\partial}

\newcommand{\ceq}{\coloneqq}

\newcommand{\R}{{\mathbb{R}}}

\newcommand{\N}{\mathbb{N}}

\DeclareMathOperator{\id}{id}

\newcommand{\nabs}[1]{\lvert{#1}\rvert} 
\newcommand{\norm}[1]{\left\lVert#1\right\rVert}
\newcommand{\nnorm}[1]{\lVert{#1}\rVert}

\newcommand{\ninnerprod}[1]{\langle{#1}\rangle}

\newcommand{\paren}[1]{{}\left(#1\right){}}

\newcommand{\bigparen}[1]{{}\big(#1\big){}}

\newcommand{\intervaloo}[1]{\left]#1\right[}
\newcommand{\intervalco}[1]{\left[#1\right[}
\newcommand{\intervalcc}[1]{\left[#1\right]}

\newcommand{\nintervaloo}[1]{{]#1[}}
\newcommand{\nintervalco}[1]{{[#1[}}
\newcommand{\nintervalcc}[1]{{[#1]}}
\newcommand{\nintervaloc}[1]{{]#1]}}

\newcommand{\on}[1]{\operatorname{#1}}

\DeclareMathOperator{\ima}{im}

\DeclareMathOperator{\Hom}{Hom}    
\DeclareMathOperator{\Bil}{Bil}    

\DeclareMathOperator{\tr}{tr}

\DeclareMathOperator{\grad}{grad}  

\DeclareMathOperator{\coker}{coker}

\newcommand{\mynewtheorem}[4] 
{
\newaliascnt{#1}{#2}
\newtheorem{#1}[#1]{#3}
\aliascntresetthe{#1}
\expandafter\def\csname #1autorefname\endcsname{%
#4%
}%
}

\newtheorem{theorem}{Theorem}[section]
\mynewtheorem{lemma}{theorem}{Lemma}{Lemma} 
\mynewtheorem{proposition}{theorem}{Proposition}{Proposition} 
\mynewtheorem{corollary}{theorem}{Corollary}{Corollary}
\mynewtheorem{quest}{theorem}{Question}{Question}

\mynewtheorem{problem}{theorem}{Problem}{Problem}

\theoremstyle{break}
\mynewtheorem{btheo}{theorem}{Theorem}{Theorem} 
\mynewtheorem{blem}{theorem}{Lemma} {Lemma}
\mynewtheorem{bcor}{theorem}{Corollary}{Corollary}
\mynewtheorem{bproblem}{theorem}{Problem}{Problem}

\theoremstyle{plain}
\theorembodyfont{\normalfont}
\mynewtheorem{definition}{theorem}{Definition}{Definition}
\mynewtheorem{example}{theorem}{Example}{Example}
\mynewtheorem{remark}{theorem}{Remark}{Remark}

\theoremstyle{break}
\mynewtheorem{bdfn}{theorem}{Definition}{Definition}
\mynewtheorem{bex}{theorem}{Example}{Example}
\mynewtheorem{brem}{theorem}{Remark}{Remark}
\theoremheaderfont{\itshape}
\theorembodyfont{\upshape}
\theoremstyle{nonumberplain}
\theoremseparator{.}
\theoremsymbol{\ensuremath{\Box}}
\newtheorem{proof}{\textsc{Proof}}

\begin{document}
\title{On $H^2$-gradient Flows for the Willmore Energy}
\author{Henrik Schumacher\thanks{
\href{mailto:henrik.schumacher@uni-hamburg.de}{henrik.schumacher@uni-hamburg.de}
}}

\maketitle

\begin{abstract}
We show that the concept of $H^2$-gradient flow for the Willmore energy and other functionals that depend at most quadratically on the second fundamental form is well-defined in the space of immersions of Sobolev class $W^{2,p}$ from a compact, $n$-dimensional manifold into Euclidean space, provided that $p \geq 2$ and $p>n$. We also discuss why this is \emph{not} true for Sobolev class $H^2=W^{2,2}$. In the case of equality constraints, we provide sufficient conditions for the existence of the projected $H^2$-gradient flow and demonstrate its usability for optimization with several numerical examples.
\end{abstract}

\section{Introduction}\label{sec:intro}

Riemannian geometry provides a vast toolbox for the numerical treatment of nonlinear optimization problems. 
Following Hilbert spaces, Riemannian manifolds may be considered as the second nicest kind of spaces to perform (smooth) optimization on (see, e.g., \cite{MR2968868}).
It takes no wonder that there have been several attempts to introduce Riemannian geometry to infinite-dimensional spaces of immersions (see \cite{Eckstein:2007:GSF:1281991.1282017} for inner products based on Sobolev space $H^1$ and their applications in geometry processing) and to shape spaces, the quotient spaces of immersions modulo reparametrization 
(%
see \cite{MR2333829}, 
\cite{MR2201275},
and \cite{MR2888014}).

Many of these attempts have been detailed only for variational problems of one-dimensional shapes, exploiting the Morrey embedding $H^1 \hookrightarrow C^0$.
Indeed, some infinite dimensional problems related to curvature energies of higher-dimensional immersed submanifolds (such as surfaces in $\R^3$), can hardly be put into an economic, strongly Riemannian context.\footnote{Here and in the following, we use the terms \emph{Riemannian manifold} and \emph{strong Riemannian manifold} synonymously.} This is unfortunate as such energies occur frequently in practical applications, e.g., in mechanics as bending energy in the Kirchhoff-Love model for thin plates (see \cite{Kirchhoff1850}, \cite{Love491}); 
in biology as Canham-Helfrich energy of cell membranes (see \cite{CANHAM197061}, \cite{citeulike:6130970}); and in computer graphics as regularizers for various geometry processing tasks (see \cite{MR3285020} and references therein).

A classical and very instructive example is provided by the Willmore energy of an immersion $f \colon \varSigma \to \AmbSpace$ of a compact smooth manifold $\varSigma$ into Euclidean space. Up to some constants, it is given by
\begin{align*}
	\Willmore(f) \ceq \dim(\varSigma)^{-2} \int_\varSigma \nabs{\Delta_f \, f}^2_{\AmbSpace}\, \vol_f.
\end{align*}
Here, $\Delta_f$ denotes the Laplace-Beltrami operator with respect to Riemannian metric $g_f \ceq f^\pull \ninnerprod{\cdot, \cdot}_{\AmbSpace} \ceq \ninnerprod{\dd f\, \cdot, \dd f\, \cdot}_{\AmbSpace}$ induced by $f$ and $\vol_f$ denotes the associated Riemannian volume density. 
This representation of the Willmore energy suggests to use 
\begin{align}
	b_1(u,w) \ceq \int_\varSigma  \ninnerprod{\Delta_f u,\Delta_f w}\, \vol_f
	\qor
	b_2(u,w) \ceq \int_\varSigma  \ninnerprod{(1-\Delta_f) u,(1-\Delta_f) w}\, \vol_f	
	\label{eq:bilinearformsonImm}
\end{align}
as Riemannian metrics on the space of immersions.\footnote{The bilinear form $b_1$ will be positive definite only if suitable additional constraints are imposed on $u$ and $v$ as the locally constant functions form the kernel of $\Delta_f$.} Then, a gradient of $\Willmore$ can be defined by
\begin{align}
	b_i(\grad(\Willmore)\at_f,w) = \ninnerprod{\dd \Willmore\at_f,w}
	\qquad \text{for all variations $w$ of $f$.}
	\label{eq:weakgradienteq}
\end{align}
Obviously, $-\grad(\Willmore)$ is a decending direction, i.e., $ \ninnerprod{\dd \Willmore\at_f,-\grad(\Willmore)\at_f} \leq 0$ with equality if and only if $f$ is a critical point of $\Willmore$.

In local coordinates $(x_1,\dotsc,x_n) \colon U \subset \varSigma \to \R^n$, the Laplace-Beltrami operator reads as follows (up to lower order terms):
\begin{align*}
	\Delta_f \, u = \sum_{i,j=1}^n (\textbf{G}_f^{-1})_{ij} \pd_i\pd_j u + \lot,
\end{align*}
where $(\textbf{G}_f)_{ij}\ceq g_f\bigparen{\frac{\pd}{\pd x_i},\frac{\pd}{\pd x_j}} = \ninnerprod{\pd_i f,\pd_j f}_{\AmbSpace}$ is the Gram matrix of the Riemannian metric $g_f$ with respect to the coordinate vector fields $\frac{\pd}{\pd x_1},\dots,\frac{\pd}{\pd x_n}$.
Moreover, $\vol_f$ can be expressed as $\sqrt{\det{\textbf{G}_f}} \, \dd x$, where $\dd x$ denotes the Euclidean density. Hence, the contribution of $f|_U$ to $\cW(f)$ is given by
\begin{align*}
	\int_U \nabs{(\textbf{G}^{-1}_f)_{ij} \,\pd_i\pd_j f}^2_{\AmbSpace} \textstyle\sqrt{\det{\textbf{G}_f}} \, \dd x + \lot
\end{align*}
For $W^{2,2}(\varSigma;\AmbSpace)$, the Gram matrix $(\textbf{G}_f)_{ij}$ is only of Sobolev class $W^{1,1}$. This shows that the Willmore energy $\Willmore(f)$ is \emph{not} well-defined on the Sobolev space $W^{2,2}(\varSigma;\AmbSpace)$, even although the Willmore energy depends only quadratically on second derivatives of $f$. Moreover, letting $f \in W^{2,2}(\varSigma;\AmbSpace)$  would make it hard to make sense of \eqref{eq:bilinearformsonImm}, let alone discussing the solvability of \eqref{eq:weakgradienteq}.

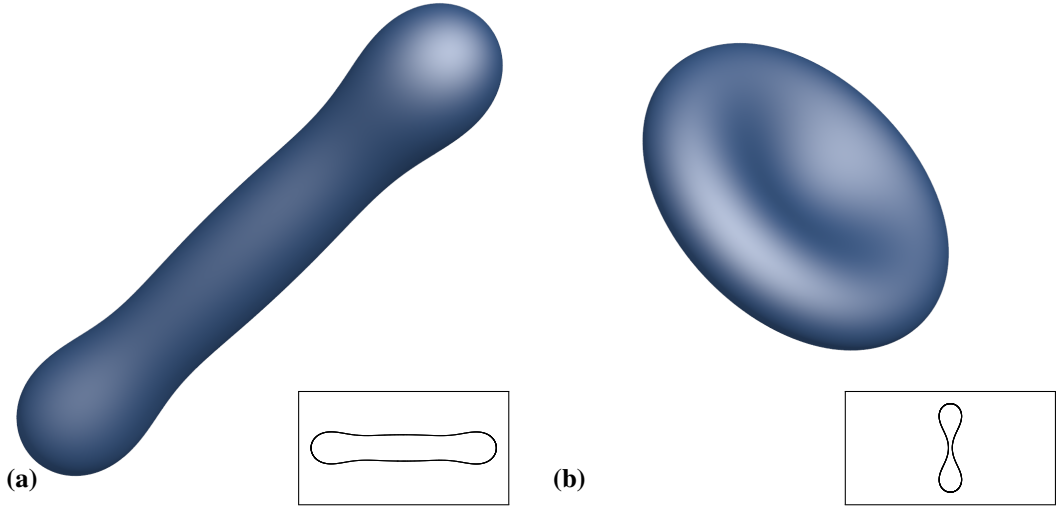
\begin{figure}
\capstart
\begin{center}
\begin{minipage}{0.45\textwidth}
\begin{center}
\begin{tikzpicture}
    \node[rectangle,
           path picture={
               \node[
               		xshift = - 0.0\textwidth,
               		yshift = 0.0\textwidth
               	] (fig) at (path picture bounding box.center){
                   \includegraphics[width=\textwidth]{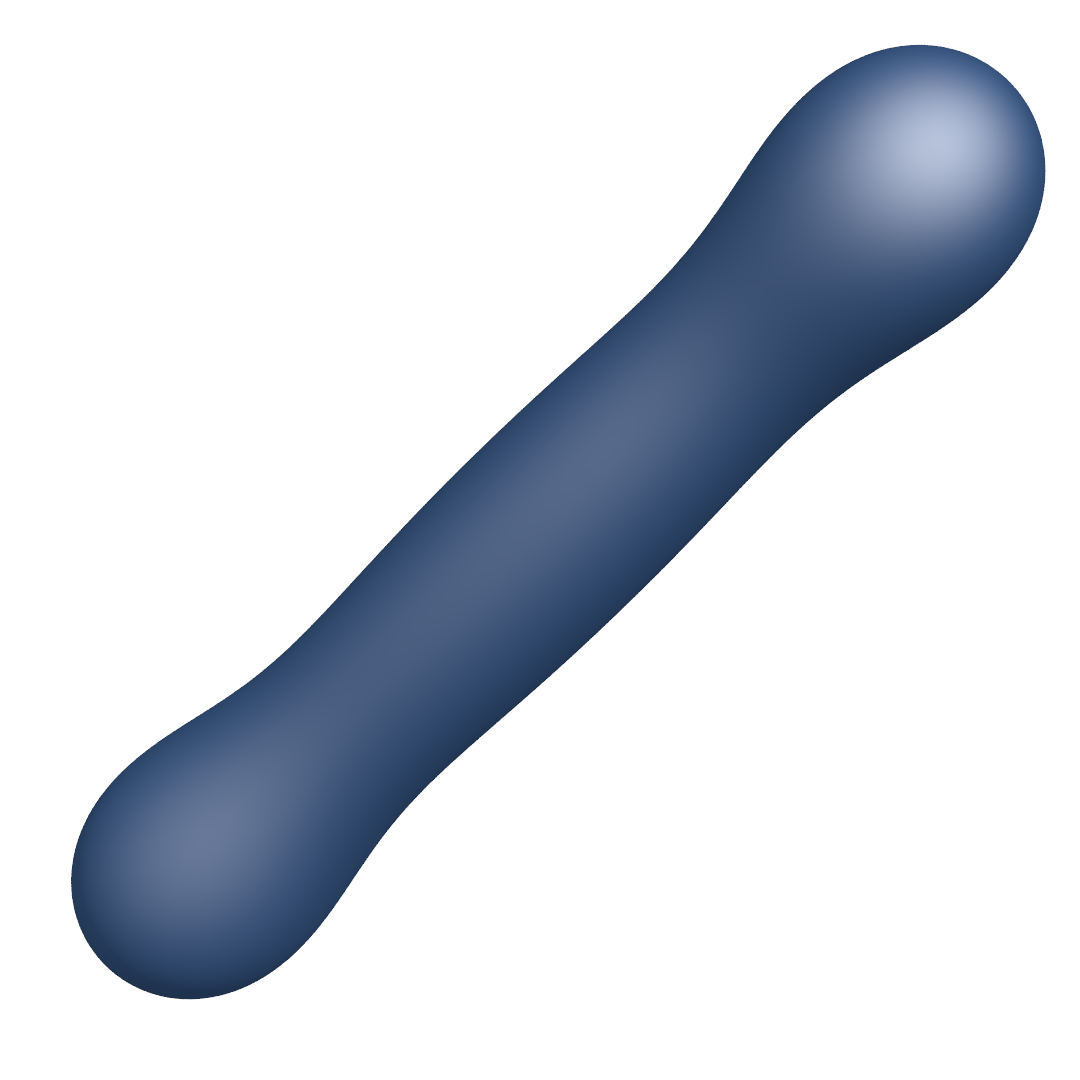}
               };
           },
           minimum width=\textwidth,minimum height=\textwidth
           ] (box) at (0,0)  {          };  	   	
	\node[above right= 1ex] at (box.south west) {\footnotesize\textbf{(a)}};
	\node[above left] (inset) at (box.south east) {
    		\framebox{\includegraphics[scale=0.2]{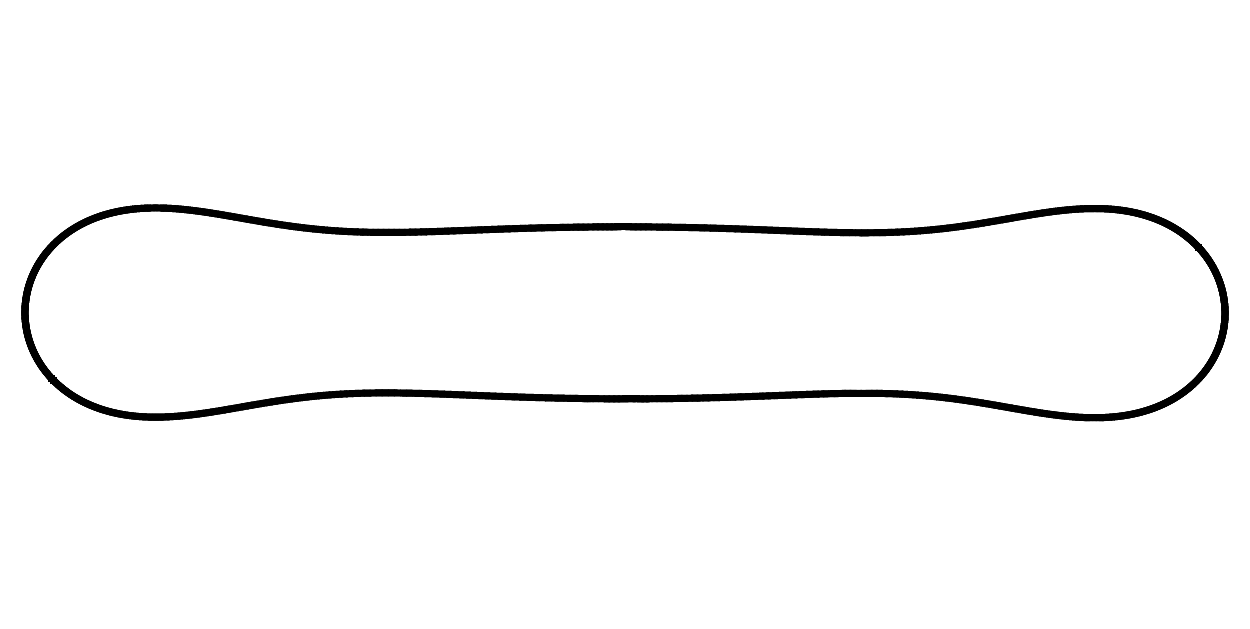}}
    	};	
\end{tikzpicture}
\end{center}
\end{minipage}
\begin{minipage}{0.45\textwidth}
\begin{center}
\begin{tikzpicture}
    \node[rectangle,
           path picture={
               \node[
               		xshift = - 0.0\textwidth,
               		yshift = 0.1\textwidth
               	] (fig) at (path picture bounding box.center){
                   \includegraphics[width=\textwidth]{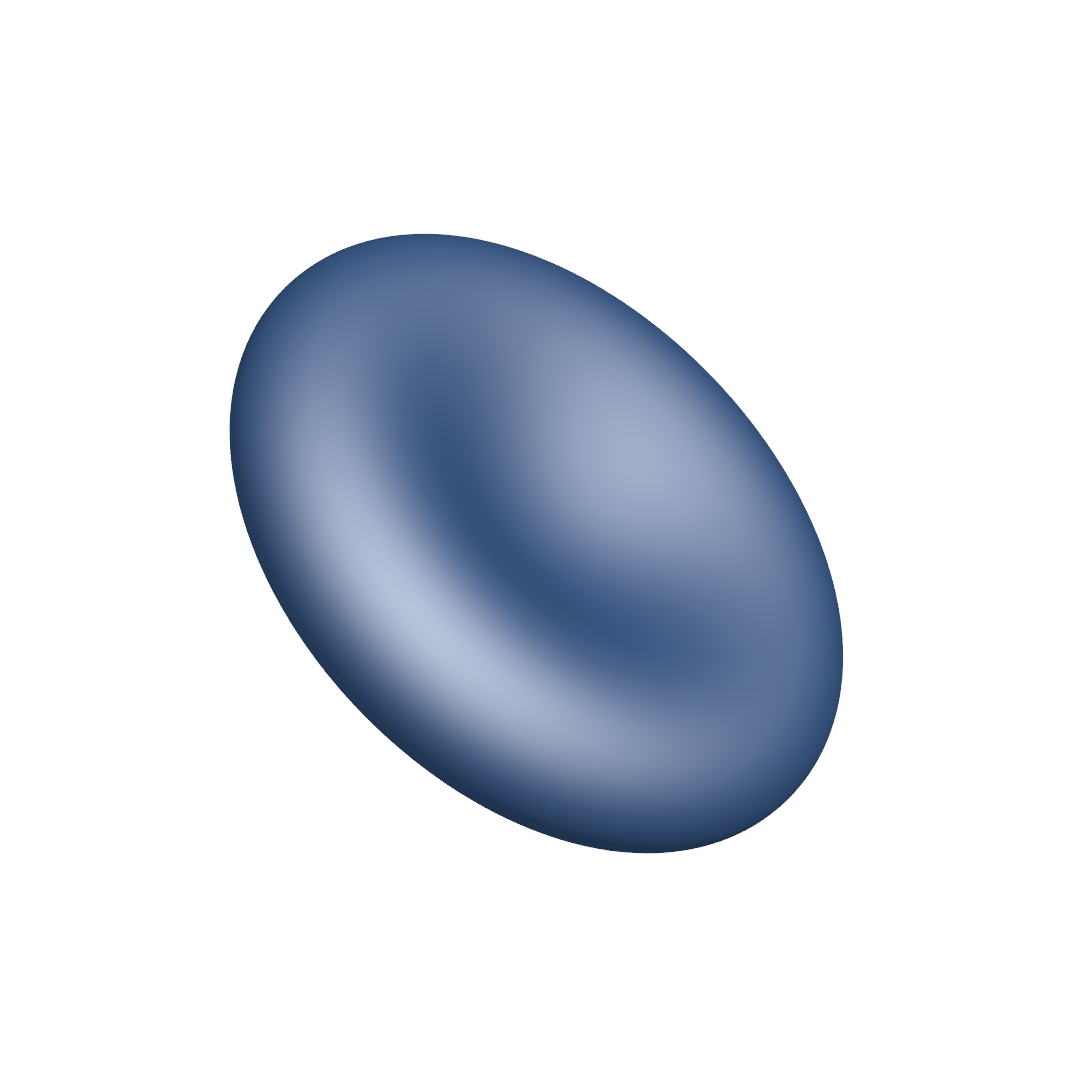}
               };
           },
           minimum width=\textwidth,minimum height=\textwidth
           ] (box) at (0,0)  {          };  	   	
	\node[above right= 1ex] at (box.south west) {\footnotesize\textbf{(b)}};
	\node[above left] (inset) at (box.south east) {
    		\framebox{\includegraphics[scale=0.2]{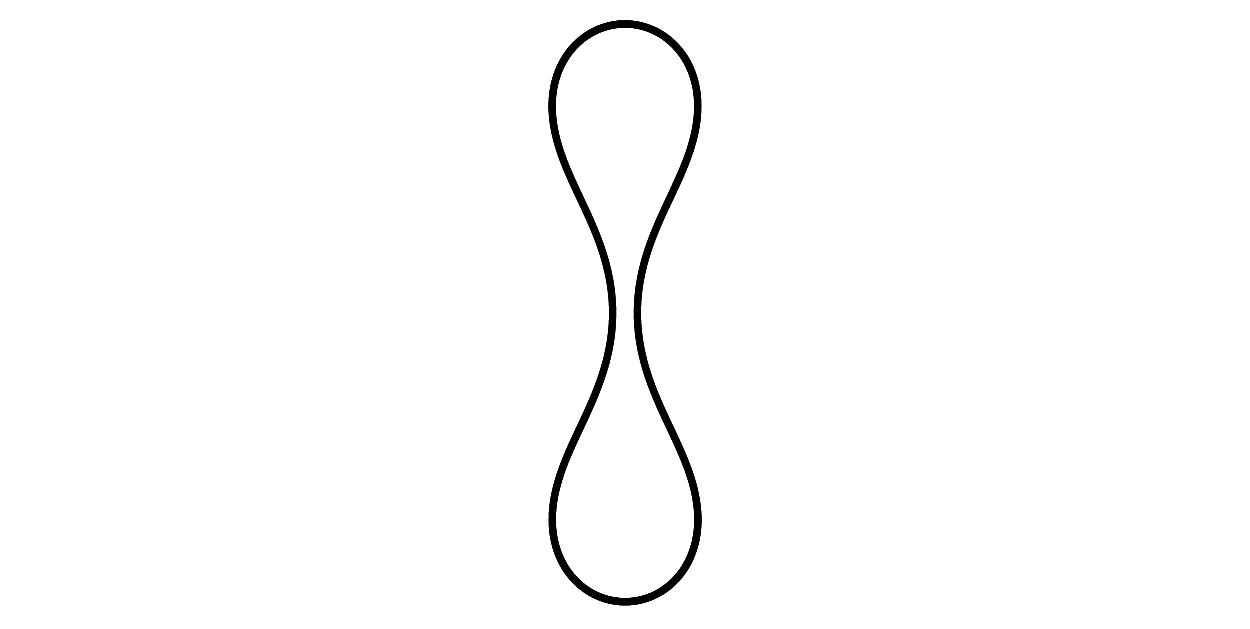}}
    	};	
\end{tikzpicture}
\end{center}
\end{minipage}
\end{center}
\caption{Two discrete local minima  ($328$k faces; insets show cross sections along symmetry axes) of the Willmore energy subject to the same constraints on surface area and enclosed volume ($\text{area}=7.24$, $\text{volume}=1.00$).
The discrete Willmore energy of surface (a) is $30.10$, the one of surface (b) is $26.42$.
}
\label{fig:Canham}
\end{figure}

We could repair this by considering $f \in W^{s,2}(\varSigma;\AmbSpace)$ with $s - \frac{n}{2}>1$ so that $f \in W^{1,\infty}(\varSigma;\AmbSpace)\cap W^{2,2}(\varSigma;\AmbSpace)$. However, even in dimension $n=2$, we would need $s > 2$, which might seem excessively large.
Moreover, neither of the two bilinear forms in \eqref{eq:bilinearformsonImm} is an inner product on the Hilbert space $W^{s,2}(\varSigma;\AmbSpace)$ as both of them induce strictly weaker norms.
One \emph{can} employ bilinear forms such as
\begin{align}
	b_3(u,w) \ceq \int_\varSigma  \ninnerprod{(1-\Delta_f)^\frac{s}{2} u,(1-\Delta_f)^\frac{s}{2} w}\, \vol_f
	\label{eq:bilinearformsonImm2}
\end{align}
(this is the approach considered in \cite{MR2888014}), but we deem it a bit dissatisfactory to adjust the differentiability to the dimension of the manifolds to immerse. Concerning the numerical treatment, solving the gradient equation with conforming Ritz-Galerkin methods would require finite element spaces of class $C^{\lfloor s\rfloor}(\varSigma;\AmbSpace) \subsetneq C^{1}(\varSigma;\AmbSpace)$. Although not impossible, even treating finite element spaces of class $C^{1}(\varSigma;\AmbSpace)$ can already be very complicated in terms of implementation. Note also that solving the gradient equation \eqref{eq:weakgradienteq} numerically becomes increasingly expensive with increasing $s$, as the system matrix rapidly loses sparsity. Hence, we strongly desire to use $s$ as small as possible.

\bigskip

This is why we follow a different direction. Instead of cranking up the differentiability, we may increase the integrability: As we will see in \autoref{sec:CurvatureEnergies}, the space 
\begin{align*}
	\Imm^{2,p}(\varSigma;\AmbSpace)  \ceq \set{
		f \in W^{2,p}(\varSigma;\AmbSpace) 
		| 
		\text{for all $x \in \varSigma$: $\dd f\at_x$ is injective}
		}
\end{align*}
is an open subset of $W^{2,p}(\varSigma;\AmbSpace)$ and the
Willmore energy is a smooth function on $\Imm^{2,p}(\varSigma;\AmbSpace)$,
provided that $p\geq 2$ and $p > \dim(\varSigma)$. As many of the finite element spaces of class $W^{2,2}$ happen to be also of class $W^{2,\infty}$, increasing the integrability of trial functions need not induce any additional costs. 
However, this forces us to take a closer look at the solvability of \eqref{eq:weakgradienteq}. Note that the pairings in \eqref{eq:bilinearformsonImm} are still well-defined for a tangent vector $u \in T_f \Imm^{2,p}(\varSigma;\AmbSpace) =W^{2,p}(\varSigma;\AmbSpace)$ and for $w \in W^{2,q}(\varSigma;\AmbSpace)$, where $q$ is the H\"{o}lder conjugate of $p$. As we have to show later, they induce linear isomorphisms $J_f \colon X_f \to Y_f\dual$ for suitably chosen subspaces $X_f \subset W^{2,p}(\varSigma;\AmbSpace)$ and $Y_f \subset W^{2,q}(\varSigma;\AmbSpace)$. Hence, \eqref{eq:weakgradienteq} is solvable whenever we are able to extend the differential $\dd \Willmore \at_f \in X_i'$ continuously to an element in $Y_f\dual$.
Once put on solid ground, this allows us to make sense of both the $H^2$-gradient flow and of the $H^2$-gradient descent. Note that \emph{discretized} variants of $H^1$ and $H^2$-gradient flows are already known for some time and that they are frequently applied with success in computer graphics applications (see, e.g., \cite{Eckstein:2007:GSF:1281991.1282017}). Hence we regard the theoretical background as our principal contribution.

\bigskip

Although we primarily treat certain curvature energies of immersed manifolds, we aim also at developing a theory that can be applied to other infinite-dimensional problems.
To this end, we introduce the notion of \UnnamedSpaces in \autoref{sec:RieszStructures}.
In order to be able to apply our theory to constrained problems, we have to introduce the notion of the pseudoinverse of linear operators between \UnnamedSpaces. Moreover, we need practicable criteria	 to decide wether a given operator admits a pseudoinverse. Indeed, this will occupy the major part of \autoref{sec:RieszStructures}.

Once we have settled the linear functional analysis of \UnnamedSpaces, we can 
introduce the category of \UnnamedManifolds and apply our knowledge to define gradients and gradient flows of scalar functions (see \autoref{sec:UnnamedManifolds}). Moreover, we provide a constraint qualification for equality constraints that guarantees that the feasible set is again a \UnnamedManifold and show how the projected gradient can be obtained by solving linear saddle point system.
Only then we are able to apply the nonlinear theory to curvature dependent energies (see
\autoref{sec:CurvatureEnergies}).
Finally, we provide some numerical examples to demonstrate the implementability, the robustness, and the efficiency of the introduced concepts in \autoref{sec:NumericalExamples}.


\clearpage

\section{\CUnnamedSpaces}\label{sec:RieszStructures}

We introduce the category of \UnnamedSpaces and develop some part of its functional analysis.
The objects of this category are commutative diagrams. This requires a certain management of nomenclature in order to be able to handle several objects at once. The way we choose may seem quite natural from the perspective of category theory (see \autoref{rem:cattheory} below). Still, we aim at presenting the theory such that it is also accessible for non-experts in category theory.

\subsection{Basic Definitions}\label{sec:RieszStructures_BasisDefinitions}

\begin{definition}\label{dfn:Rieszstructures}
A \UnnamedSpace $E$ consists of
\begin{enumerate}
\item two Banach spaces $\RX[E]$, $\RY[E]$ and a Hilbert space $\RH[E]$;
\item a chain of continuous, linear, and dense injections
$
\begin{tikzcd}[]
	\RX[E]
		\ar[r,hook, two heads,"{\Ri[E]}"] 
	&\RH[E]
		\ar[r,hook, two heads,"{\Rj[E]}"] 
	&
	\RY[E];
\end{tikzcd}
$
\item an inner product $\ninnerprod{\cdot,\cdot}_{\RH[E]}$ on $\RH[E]$ and its Riesz isomorphism $\RI[E] \colon \RH[E] \to \RH[E]'$;
\item  and of an isomorphism of Banach spaces $\RJ[E] \colon \RX[E] \to \RY[E]'$ 
\end{enumerate}
such that the following diagram commutes:
\begin{equation}
\begin{tikzcd}[]
	\RX[E]
		\ar[d,hook, two heads,"{\Ri[E]}"'] 
		\ar[r,"{\RJ[E]}","\cong"'] 
	&\RY[E]'
		\ar[d,hook, two heads,"{\Rj[E]'}"] 
	\\
	\RH[E]
		\ar[r,"{\RI[E]}"',"\cong"]	
	&\RH[E]'\nospaceperiod
\end{tikzcd}\label{eq:Rieszsquarediagram}
\end{equation}
We also write $\RX[E][][']$, $\RH[E][][']$, and $\RY[E][][']$ instead of $\RX[E]'$, $\RH[E]'$,  and $\RY[E]'$.
Moreover, every \UnnamedSpace $E$ exhibits a pre-Hilbert metric $\Rg[E]$ which can be expressed in each of the following ways:
\begin{align*}
	\Rg[E](x_1,x_2) 
	\ceq \ninnerprod{\Ri[E] \, x_1, \Ri[E]\, x_2}_{\RH[E]}
	= \ninnerprod{\RI[E] \, \Ri[E] \, x_1, \Ri[E]\, x_2}	
	= \ninnerprod{\RJ[E] \, x_1, \Rj[E]\,\Ri[E]\, x_2},
	\quad \text{for $x_1$, $x_2 \in \RX[E]$.}
\end{align*}
A \emph{morphism} $A \colon E_1 \to E_2$ between \UnnamedSpaces $E_1$ and $E_2$
is by definition a linear and continuous chain map, i.e., 
a triple $(\RX[A],\RH[A],\RY[A])$ with $\RX[A] \in L(\RX[E_1] ;\RX[E_2])$, $\RH[A] \in L(\RH[E_1] ;\RH[E_2])$, and $\RX[A] \in L(\RY[E_1] ;\RY[E_2])$ such that following diagram commutes:
\begin{equation*}
\begin{tikzcd}[]
	\RX[E_1]
		\ar[d,hook, two heads,"{\Ri[E_1]}"]
		\ar[rr, "{\RX[A]}"]
	&&\RX[E_2]
		\ar[d,hook, two heads,"{\Ri[E_2]}"]	
	\\
	\RH[E_1]
		\ar[d,hook, two heads,"{\Rj[E_1]}"]	
		\ar[rr, "{\RH[A]}"]	
	&&\RH[E_2]
		\ar[d,hook, two heads,"{\Rj[E_2]}"]		
	\\
	\RY[E_1]
		\ar[rr, "{\RY[A]}"]	
	&&\RY[E_2]\nospaceperiod
\end{tikzcd}
\end{equation*}
The space of all \UnnamedMorphisms between $E_1$ and $E_2$ is denoted by $\RMor(E_1;E_2)$.
For morphisms $A \in \RMor(E_1;E_2)$ and $B \in \RMor(E_2;E_3)$ between \UnnamedSpaces $E_1$, $E_2$ and $E_3$, we define their product $A\, B \in \RMor(E_1;E_3)$ by setting $\RZ[(A\,B)] \ceq \RZ[A] \; \RZ[B]$ for all $\RZ \in \{\RX,\RH,\RY\}$.
\end{definition}

\begin{remark}\label{rem:cattheory}
The notation was chosen such that it reflects the fact that
$\RX$, $\RX[][][']$, $\RH$, $\RH[][][']$, $\RY$, and $\RY[][][']$ are functors from the category of \UnnamedSpaces into the category of Banach spaces and that $\Ri$, $\Rj$, $\RI$, and $\RJ$ are natural transformations between them.
Note that the functors $\RX$, $\RH$, and $\RY$ are covariant while their duals $\RX[][][\dual]$, $\RH[][][\dual]$, and $\RY[][][\dual]$ are contravariant.
\end{remark}

\begin{remark}
At first glance, the notion of a \UnnamedSpace seems to boil down to the notion of a \emph{Gelfand triple} or \emph{rigged Hilbert space}, i.e., a topological vector space $\RX$, a Hilbert space $\RH$ together with linear, dense embeddings $\RX \hookrightarrow \RH \hookrightarrow \RX'$. 
However, this is not true, since \UnnamedSpaces involve a \emph{third} Banach space $\RY$ which need not coincide with $\RX'$.
Moreover, we require as additional data that the Riesz isomorphism $\RI \colon \RH \to \RH[][][\dual]$ induces an isormorphism $\RJ \colon \RX \to \RY[][][\dual]$. If $\RY$ were identical to $\RX[][][\dual]$, this would imply $\RX= \RH$.
\end{remark}

A prototypical \UnnamedSpace is given by the following example. 

\begin{example}\label{ex:LpspacesasUnnamedSpace}
Let $(\varOmega,\cA,\mu)$ be a finite measure space, let $p \in \intervalcc{2,\infty}$, and let $q \in \intervalcc{1,2}$ be the H\"older conjugate of $p$. Put $\RX[E] \ceq L^p(\varOmega,\mu)$, $\RH[E] \ceq L^2(\varOmega,\mu)$, $\RY[E] \ceq L^q(\varSigma;\mu)$, and denote by $\Ri[E] \colon L^p(\varOmega,\mu) \hookrightarrow L^2(\varOmega,\mu)$, $\Rj[E] \colon L^2(\varOmega,\mu) \to L^q(\varOmega,\mu)$ the canonical embeddings.
The Riesz isomorphism $\RI[E] \colon L^2(\varOmega,\mu) \to (L^2(\varOmega,\mu))'$ is given by $\ninnerprod{\RI[E]\,v_1,v_2} \ceq \int_\varOmega v_1\,v_2\, \dd \mu$ for $v_1$, $v_2 \in \RH[E]$.
Analogously, one may consider the operator $\RJ[E] \colon L^p(\varOmega,\mu) \to (L^q(\varOmega,\mu))'$ defined by $\ninnerprod{\RJ[E] \,u,w} \ceq \int_\varOmega u\,w\, \dd \mu$ for $u \in \RX[E]$ and $w \in \RY[E]$. Observe that $\RI[E] \circ \Ri[E] = \Rj[E]' \circ \RJ[E]$.
By the Radon-Nikodym theorem, $\RJ[E]$ is an isomorphism of Banach spaces, hence $E$ is a \UnnamedSpace.
Observe that for $p= \infty$, we obtain a \UnnamedSpace which is not reflexive.
\end{example}

Axioms 1--3 of \autoref{dfn:Rieszstructures} imply that a morhism between \UnnamedSpaces is uniquely defined by its values on each of the three scales $\RX$, $\RH$, and $\RY$:
\begin{lemma}[Identity lemma]\label{lem:nulllemma}
Let $E_1$, $E_2$ be \UnnamedSpaces and let $A \in \RMor(E_1;E_2)$ be a morphism. Then the following statements are equivalent:
\begin{gather*}
	\mathrm{1.} \quad A=0.
	\qquad
	\mathrm{2.} \quad \RY[A]=0.	
	\qquad
	\mathrm{3.} \quad \RH[A]=0.	
	\qquad
	\mathrm{4.} \quad \RX[A]=0.
\end{gather*}
\end{lemma}

In the category of Hilbert spaces, each closed subspace is split, i.e., it has a closed complement. As it is common knowledge, this does not hold true in the category of Banach spaces: A closed subspace in a Banach space is split if and only if there is a continuous projector onto it. This motivates the following definition.

\begin{definition}\label{dfn:splitsubspace}
Let $E_2$ be a \UnnamedSpace and let
$\RX[E_1] \subset \RX[E_2]$,
$\RH[E_1] \subset \RH[E_2]$,
and
$\RY[E_1] \subset \RY[E_2]$
be split Banach subspaces with $\Ri[E_2](\RX[E_1]) \subset \RH[E_1]$ and $\Rj[E_2](\RH[E_1])  \subset \RY[E_1]$.
We call the triple $E_1 =(\RX[E_1],\RH[E_1],\RY[E_1])$ a \emph{split subspace} of the \UnnamedSpace $E_2$ if there is a \UnnamedMorphism $\RP \in \RMor(E_2;E_2)$ such that $\RP$ is a projector onto $E_1$ (i.e., $PP=P$ and $\RZ[P](\RZ[E_2]) = \RZ[E_1]$ for $\RZ\in \{\RX,\RH,\RY\}$).
\invisible{\autoref{dfn:Rieszstructures}}
\end{definition}

The following lemma shows that split subspaces inherit the denseness properties of an ambient \UnnamedSpace.
\begin{lemma}\label{lem:splitsubspacedense}
Let $E_2$ be a \UnnamedSpace and let $E_1 \subset E_2$ be a split subspace.
Then $\Ri[E_2](\RX[E_1]) \subset \RH[E_1]$ and $\Rj[E_2](\RH[E_1])  \subset \RY[E_1]$ are dense.
\end{lemma}
\begin{proof}
\invisible{\autoref{dfn:splitsubspace}}
Let $v \in \RH[E_1]$ and $w \in \RY[E_1]$. Since $\Ri[E_2] ( \RX[E_2])$ and $\Rj[E_2] ( \RH[E_2])$ are dense in $\RH[E_2]$ and $\RY[E_2]$, respectively, we may write
$
	v = \lim_{n \to \infty} \Ri[E_2] \; u_n
$ and $
	w = \lim_{n \to \infty} \Rj[E_2] \; v_n
$
with suitable $u_n \in \RX[E_2]$, $v_n \in \RH[E_2]$, $n \in \N$. Let $P \in \RMor(E_2;E_2)$ be a projector onto $E_1$. Then we have
\begin{align*}
	v 
	&= (\id_{\RH[E_2]} - \RH[P]) \; v
	= \lim_{n \to \infty} (\id_{\RH[E_2]} - \RH[P]) \; \Ri[E_1] \; u_n 
	= \lim_{n \to \infty} \Ri[E_1] \; (\id_{\RX[E_2]} - \RX[P]) \; u_n.
\end{align*}
Analogously, we obtain
$w = \lim_{n \to \infty} \Rj[E_1] \; (\id_{\RH[E_2]} - \RH[P]) \; v_n$.
\end{proof}

\begin{definition}\label{dfn:generalizedinverse}
Let $A \in \RMor(E_1;E_2)$ be a morphism between \UnnamedSpaces $E_1$ and $E_2$.
We say that $B \in \RMor(E_2;E_1)$ is a \emph{generalized inverse of $A$ (in the category of \UnnamedSpaces)} if
\begin{align*}
	A\,B\,A = A \qand B\,A\, B = B.
\end{align*}
\invisible{\autoref{dfn:Rieszstructures}}
\end{definition}

If $A$ has generalized inverse $B$, then $P = \id_{\RX[E_1]} - B \,A \in L(E_1;E_1)$ is a projector onto $\ker(A)$ and $Q = A \, B \in L(E_2;E_2)$ is a projector onto $\ima(A)$.
Hence both 
\begin{align*}
	\ker(A) \ceq (\ker(\RX[A]),\ker(\RH[A]),\ker(\RY[A]))
	\qand
	\ima(A) \ceq (\ima(\RX[A]),\ima(\RH[A]),\ima(\RY[A]))
\end{align*}
are split subspaces and \autoref{lem:splitsubspacedense} implies:
\begin{corollary}\label{cor:kerimadense}
Let $A \in \RMor(E_1;E_2)$ be a morphism between \UnnamedSpaces.
Suppose that $A$ has a generalized inverse $B \in \RMor(E_2;E_1)$. 
Then each of the following inclusion is dense:
\begin{align*}
	\Ri[E_1] \, ( \ker(\RX[A])) &\subset \ker(\RH[A]),
	&
	\Rj[E_1] \, ( \ker(\RH[A])) &\subset \ker(\RY[A]),
	\\
	\Ri[E_2] \, ( \ima(\RX[A])) &\subset \ima(\RH[A]),
	&
	\Rj[E_2] \, ( \ima(\RH[A])) &\subset \ima(\RY[A]).
\end{align*}
\end{corollary}

\subsection{Reflexive \CUnnamedSpaces}

If $A$ is a morphism of \UnnamedSpaces, we may introduce an adjoint of $\RX[A]$ as in the following proposition. In order to define an adjoint morphism $A^*$, we will have to be content with the smaller class of reflexive \UnnamedSpaces (see below).
\begin{proposition}\label{prop:adjXA}
Let $A \in \RMor(E_1;E_2)$ be a morphism of \UnnamedSpaces. Define the \emph{adjoint $\RX[A][][\adj]$ of $\RX[A]$} by $\RX[A][][\adj] \ceq \RJ[E_1]^{-1} \; \RY[A][][\dual] \; \RJ[E_2]$. 
The adjoint $\RX[A][][\adj]$ is a continuous linear operator from $\RX[E_2]$ to $\RX[E_1]$ and satisfies
\begin{align*}
	\Rg[E_1](\RX[A][][\adj] \; u_2, u_1) 
	= \Rg[E_2](u_2, \RX[A] \; u_1) 
	\quad
	\text{for all $u_1 \in \RX[E_1]$ and $u_2 \in \RX[E_2]$.}
\end{align*}
\end{proposition}
\begin{proof}
\invisible{\autoref{dfn:Rieszstructures}}
\begin{align*}
	\Rg[E_1](\RX[A][][\adj] \; u_2, u_1) 
	&= \ninnerprod{ \RJ[E_1] \; \RX[A][][\adj] u_2, \Rj[E_1] \; \Ri[E_1] \; u_1 }
	= \ninnerprod{ \RY[A][][\dual ] \; \RJ[E_2] \; u_2, \Rj[E_1] \; \Ri[E_1] \; u_1 }	
	\\	
	&= \ninnerprod{ \RJ[E_2] \; u_2, \RY[A] \; \Rj[E_1] \; \Ri[E_1] \; u_1 }		
	= \ninnerprod{ \RJ[E_2] \; u_2,  \Rj[E_2] \; \Ri[E_2] \; \RX[A] \; u_1 }			
	= \Rg[E_2](u_2, \RX[A] \; u_1 ).
\end{align*}
\end{proof}

\begin{definition}\label{dfn:reflexive}
We say that a \UnnamedSpace $E$ is \emph{reflexive} if $\RY[E]$ is reflexive. (In this case, $\RX[E]$ is also reflexive since it is isomorphic to the reflexive Banach space $\RY[E][][\dual]$.)
\invisible{\autoref{dfn:Rieszstructures}}
\end{definition}

\begin{proposition}\label{prop:Koperator}
Let $E$ be a reflexive \UnnamedSpace.\invisible{\autoref{dfn:reflexive}}
Denote the canonical embeddings by $\RPsi[\RX[E]] \colon \RX[E] \to \RX[E][][\ddual]$ and $\RPsi[\RY[E]] \colon \RY[E] \to \RY[E][][\ddual]$.
Then $\RK[E] \ceq \RJ[E]' \,\RPsi[\RY[E]] \colon \RY[E] \to \RX[E][][\dual]$ is the unique isomorphism of Banach spaces which makes the following diagram commutative
\begin{equation}
\begin{tikzcd}[]
	\RX[E]
		\ar[r,"{\RJ[E]}"]
		\ar[d,hook, two heads,"{\Ri[E]}"']	
	& \RY[E][][']
		\ar[d,hook, two heads,"{\Rj[E]\dual}"]	
	\\
	\RH[E]
		\ar[r,"{\RI[E]}"]
		\ar[d,hook, two heads,"{\Rj[E]}"']		
	&\RH[E][][\dual]
		\ar[d,hook, two heads,"{\Ri[E]\dual}"]		
	\\
	\RY[E]
		\ar[r,dashed,"{\RK[E]}"]
	&\RX[E][][\dual]	\nospaceperiod			
\end{tikzcd}\label{eq:RK}
\end{equation}
Moreover, one has the identity $\RK[E]\dual \; \RPsi[\RX[E]] = \RJ[E]$.
\end{proposition}
\begin{proof}
For $u_1$, $u_2 \in \RX[E]$, one computes
\begin{align*}
	\ninnerprod{\RK[E]\; \Rj[E] \; \Ri[E] \; u_1, u_2} 
	&= \ninnerprod{\RJ[E]' \,\RPsi[\RY[E]] \; \Rj[E] \; \Ri[E] \; u_1, u_2} 
	 = \ninnerprod{\RPsi[\RY[E]] \; \Rj[E] \; \Ri[E] \; u_1, \RJ[E] \; u_2} 	
	\\
	&= \ninnerprod{\RJ[E] \; u_2 , \Rj[E] \; \Ri[E] \; u_1} 	
  	 = \ninnerprod{\Rj[E]\dual \;\RJ[E] \; u_2 , \Ri[E] \; u_1} 	
	\\
	&= \ninnerprod{\RI[E] \; \Ri[E] \; u_2 , \Ri[E] \; u_1} 	
	 = \ninnerprod{\RI[E] \; \Ri[E] \; u_1 , \Ri[E] \; u_2} 	
	 = \ninnerprod{ \Ri[E]\dual \; \Rj[E]\dual \;\RJ[E] \; u_1 , u_2}.
\end{align*}
This shows that the diagram is commutative. Uniqueness follows from the denseness of the injections $\Ri[E]$ and $\Rj[E]$. Now let $u \in \RX[E]$ and $w \in \RY[E]$. With $\RK[E] \ceq \RJ[E]\dual \; \RPsi[\RY[E]]$, we compute
\begin{align*}
	\ninnerprod{\RK[E]\dual \; \RPsi[\RX[E]]\; u, w}
	&=\ninnerprod{\RPsi[\RY[E]]\dual \; \RJ[E]\ddual \; \RPsi[\RX[E]] \; u, w}
	= \ninnerprod{\RPsi[\RX[E]] \; u, \RJ[E]\dual \; \RPsi[\RY[E]] \; w}	
	\\
	&= \ninnerprod{\RJ[E]\dual \; \RPsi[\RY[E]] \; w , u}	
	= \ninnerprod{\RPsi[\RY[E]] \; w , \RJ[E] \; u}	
	= \ninnerprod{\RJ[E] \; u , w},
\end{align*}
showing that $\RK[E]\dual \; \RPsi[\RX[E]] = \RJ[E]$.
\end{proof}

\begin{definition}\label{dfn:adjointmorphism}
Let $A \in \RMor(E_1;E_2)$ be a morphism between reflexive \UnnamedSpaces.
With \autoref{prop:Koperator}, we may also define the \emph{adjoint} of $\RY[A]$ by $\RY[A][][\adj]  \ceq  \RK[E_1]^{-1}\, \RX[A][][\dual]\, \RK[E_2]$.
Observe that the triple $( \RX[A][][\adj],\RH[A][][\adj],\RY[A][][\adj])$ is again a \UnnamedMorphism. Thus, we say that $A\adj \ceq ( \RX[A][][\adj],\RH[A][][\adj],\RY[A][][\adj]) \in \RMor(E_2;E_1)$ is the \emph{adjoint \UnnamedMorphism of $A$}. 
This whole setting is summarized in the commutative diagram below.
Note its mirror symmetry with respect to the virtual line through $\RI[E_2]$ and $\RI[E_1]$.
\begin{equation}
\begin{tikzcd}[column sep=4em, row sep=.75em]
	&\RY[E_2][][']
		\ar[rrr,"{\RY[A][][\dual]}", pos=0.3]
		\ar[ddd,hook, two heads,"{\Rj[E_2]'}"', pos=0.7]
	&
	&
	&\RY[E_1][][\dual]
		\ar[ddd,hook, two heads,"{\Rj[E_1]'}", pos=0.5]
	\\
	\RX[E_2]
		\ar[ur,"{\RJ[E_2]}"]
		\ar[rrr,"{\RX[A][][\adj]}", pos=0.7, crossing over]		
		\ar[ddd,hook, two heads,"{\Ri[E_2]}"']	
	&
	&
	&\RX[E_1]
		\ar[ur,"{\RJ[E_1]}"']	
	\\
	\\	
	&\RH[E_2][][\dual]
		\ar[rrr,"{\RH[A][][\dual]}", pos=0.3]
		\ar[ddd,hook, two heads,"{\Ri[E_2]'}"', pos=0.7]		
	&
	&
	&\RH[E_1][][\dual]
		\ar[ddd,hook, two heads,"{\Ri[E_1]'}"]			
	\\
	\RH[E_2]
		\ar[ur,"{\RI[E_2]}"]
		\ar[rrr,"{\RH[A][][\adj]}", pos=0.7, crossing over]	
		\ar[ddd,hook, two heads,"{\Rj[E_2]}"']		
	&
	&
	&\RH[E_1]
		\ar[ur,"{\RI[E_1]}"']	
		\ar[from=uuu,hook, two heads,"{\Ri[E_1]}", pos=0.3, crossing over]									
	\\
	\\	
	&\RX[E_2][][\dual]
		\ar[rrr,"{\RX[A][][\dual]}", pos=0.3]
	&
	&
	&\RX[E_1][][\dual]
	\\
	\RY[E_2]
		\ar[ur,"{\RK[E_2]}"]
		\ar[rrr,dashed,"{\RY[A][][\adj]}", pos=0.7]		
	&
	&
	&\RY[E_1]
		\ar[ur,"{\RK[E_1]}"']					
		\ar[from=uuu,hook, two heads,"{\Rj[E_1]}", pos=0.3, crossing over]				
\end{tikzcd}\label{eq:RYstar}
\end{equation}
\end{definition}

\subsection{Pseudoinverses}

A straight-forward application of the identity lemma (see \autoref{lem:nulllemma}) leads us to the following:
\begin{proposition}\label{prop:pseudoinverse}
Let $E_1$ and $E_2$ be \UnnamedSpaces, let $A \in \RMor(E_1;E_2)$ be a morphism and let $B \in \RMor(E_2;E_1)$ be a generalized inverse.
We say that $B$ is a \emph{pseudoinverse} of $A$ if and only if $\RH[(A \,B)][][\adj] =\RH[(A \, B)]$ and $\RH[(B \,A)][][\adj] =\RH[(B \, A)]$ hold true, i.e, both $\RH[(A \,B)]$ and $\RH[(B \,A)]$ are orthoprojectors.
The pseudoinverse, if existent, is unique and we denote it by $A\pinv$.
Moreover, $\RH[A][][\pinv]$ coincides with the Moore-Penrose pseudoinverse of $\RH[A]$.

If $E_1$ and $E_2$ are reflexive then a generalized inverse $B$ of $A$ is the pseudoinverse if and only if $(A\,B)\adj = A\,B$ and $(B\,A)\adj = B\,A$ hold true.
\invisible{\autoref{dfn:generalizedinverse}}
\end{proposition}

We have primarily operators in mind which originate from the linearization of constraints. Thus, surjective operators are of particular interest to us.

\begin{lemma}\label{lem:pseudoinverseofsurjection}
Let $E_1$ and $E_2$ be reflexive \UnnamedSpaces, let $A \in \RMor(E_1;E_2)$ be a morphism with right inverse $B \in \RMor(E_2;E_1)$.
Then the following statements are equivalent:
\begin{enumerate}
	\item $A\,A\adj$ is continuously invertible.\label{item:AAstarinvertible}	
	\item $A$ admits a pseudoinverse.\label{item:Apinvexists}	
	\item $A\,A\adj$ has closed range, i.e., $\RZ[(A\,A\adj)]$ has closed range for each $\RZ \in \{\RX,\RH,\RY\}$.\label{item:AAstarclosedrange}
\end{enumerate}
Moreover, if these conditiond are satisfied, we have the identities
\begin{align}
	A\pinv = A\adj\,(A\,A\adj)^{-1}
	\qand
	(A\,A\adj)^{-1} = (A\pinv)\adj \, A\pinv.
	\label{eq:pseudoinverseformulasurjective}
\end{align}
\end{lemma}
\begin{proof}
``\ref{item:AAstarinvertible} $\Rightarrow$ \ref{item:Apinvexists}'': 
If $A\,A\adj$ is invertible, the first formula in \eqref{eq:pseudoinverseformulasurjective} defines a continuous linear operator and it is readily checked that this operator is a pseudoinverse of $A$.
\newline
``\ref{item:Apinvexists} $\Rightarrow$ \ref{item:AAstarclosedrange}'':
As operators with generalized inverse have closed range, it suffices to verify that $(A\pinv)\adj A\pinv$ provides a generalized inverse of $A\,A\adj$. This is done in the following computations:
\begin{align*}
	(A\,A\adj) \, ((A\pinv)\adj \, A\pinv) \,(A\,A\adj)
	&= 
	A\,(A\pinv\,A)\adj \,(A\pinv A) \,A\adj
	= 
	A\,(A\pinv\,A) \,(A\pinv A) \,A\adj
	= A\,A\adj,
	\\
	((A\pinv)\adj\, A\pinv) \, (A\,A\adj)	\, ((A\pinv)\adj\, A\pinv)
	&=
	(A\pinv)\adj\, (A\pinv \, A) \, (A\pinv\, A)\adj \,A\pinv
	=
	(A\pinv)\adj\, (A\pinv \, A) \, (A\pinv\, A) \,A\pinv
	=
	(A\pinv)\adj \, A\pinv.
\end{align*}
``\ref{item:AAstarclosedrange} $\Rightarrow$ \ref{item:AAstarinvertible}'': 
As $\RH[A]$ is a surjective, $\RH[A]\,\RH[A][][\adj]$ is continuously invertible. Now, \autoref{lem:splitsubspacedense} implies that both $\RX[A]\,\RX[A][][\adj]$ and $\RY[A]\,\RY[A][][\adj]$ have trivial kernels and dense images.
By the open mapping theorem, $\RX[A]\,\RX[A][][\adj]$ and $\RY[A]\,\RY[A][][\adj]$ are invertible if and only of their ranges are closed.
\end{proof}

We have an analogous result for injective operators.
As the computations for its verification are quite similar, we skip its proof.

\begin{lemma}\label{lem:pseudoinverseofinjection}
Let $E_1$ and $E_2$ be reflexive 
\UnnamedSpaces,\invisible{\autoref{prop:pseudoinverse}} 
let $B \in \RMor(E_2;E_1)$ be a morphism with left inverse $A \in \RMor(E_1;E_2)$.
Then the following statements are equivalent:
\begin{enumerate}
	\item $B\adj\,B$ is continuously invertible.\label{item:BstarBinvertible}	
	\item $B$ admits a pseudoinverse.\label{item:Bpinvexists}	
	\item $B\adj\,B$ has closed range.\label{item:BstarBclosedrange}
\end{enumerate}
Moreover, if these conditiond are satisfied, we have the identities
\begin{align}
	B\pinv = (B\adj\,B)^{-1} B\adj,
	\qand	
	(B\adj\,B)^{-1} = B\pinv \, (B\pinv)\adj .	
	\label{eq:pseudoinverseformulainjective}
\end{align}
\end{lemma}

\subsection{Subspace Theorem}

For the discussion of equality constraints, it is important to decide whether the kernel
of a \UnnamedMorphism $A$ is again a \UnnamedSpace. The following theorem assures us that this is the case whenever $A$ is surjective (which is a standard constraint qualification) and admits a pseudoinverse.

\begin{theorem}\label{theo:SubspaceTheorem}
Let $E_1$ and $E_2$ be reflexive \UnnamedSpaces and let $A \in \RMor(E_1;E_2)$ be a surjective morphism. Define $E_0\ceq (\ker(\RX[A]),\ker(\RH[A]),\ker(\RY[A]))$ and denote the canonical injections by $\Ri[E_0] \colon \RX[E_0] \hookrightarrow \RH[E_0]$ and $\Rj[E_0] \colon \RH[E_0] \hookrightarrow \RY[E_0]$. Moreover, define
\begin{align}
	\arraycolsep=1.4pt
	\begin{array}{rclcrcll}
		\RI[E_0] \colon \RH[E_0] &\to &\RH[E_0][][\dual],
		& \quad &\ninnerprod{\RI[E_0]\, v_1, v_2} &\ceq &\ninnerprod{\RI[E_1]\, v_1, v_2}
		& \quad \text{for $v_1$, $v_2 \in \RH[E_0]$},
		\\
		\RJ[E_0] \colon \RX[E_0] &\to &\RY[E_0][][\dual],
		& \quad &\ninnerprod{\RJ[E_0]\, u, w} &\ceq &\ninnerprod{\RJ[E_1]\, u, w}
		& \quad \text{for $u \in \RX[E_0]$, $w \in \RY[E_0]$}.
	\end{array}
	\label{eq:SubspaceMetrics}
\end{align}
Then the following statements are equivalent:
\begin{enumerate}
	\item $A$ admits a pseudoinverse.\label{item:subspace_pinvexist}
	\item The mapping
	\label{item:subspace_saddlepointsystem}
	$
		\begin{pmatrix}
			\RJ[E_1] & \RY[A][][\dual]\\
			\RX[A] & 0
		\end{pmatrix}
		\colon \RX[E_1] \oplus \RY[E_2][][\dual]
		\to \RY[E_1][][\dual] \oplus \RX[E_2]		
	$ is continuously invertible.
	\item The mapping
	\label{item:subspace_saddlepointsystem2}
	$
		\begin{pmatrix}
			\RK[E_1] & \RX[A][][\dual]\\
			\RY[A][] & 0
		\end{pmatrix}
		\colon \RY[E_1] \oplus \RX[E_2][][\dual]
		\to \RX[E_1][][\dual] \oplus \RY[E_2]
	$ is continuously invertible.	

\end{enumerate}
In any of these cases, $(E_0,\RI[E_0],\RJ[E_0])$ is a \UnnamedSpace.
\end{theorem} 
\begin{proof}
``\ref{item:subspace_pinvexist} $\Leftrightarrow$ (\ref{item:subspace_saddlepointsystem} and \ref{item:subspace_saddlepointsystem2})'': Note that $\RJ[E_1]$ and $\RK[E_1]$ are continuously invertible. Hence, the saddle point matrices are invertible if and only if their Schur complements $R = - \RX[A] \, \RJ[E_1]^{-1} \, \RY[A][][\dual]$ and $S = - \RY[A] \, \RK[E_1]^{-1} \, \RX[A][][\dual]$ are invertible. Now observe that 
$R	 \, \RJ[E_2] 
= -\RX[(A\,A\adj)]$ 
and 
$S \, \RK[E_2] 
= -\RY[(A\,A\adj)]$. By \autoref{lem:pseudoinverseofsurjection}, these Schur complements are continuously invertible if and only if $A$ admits a pseudoinverse.
\newline
``\ref{item:subspace_saddlepointsystem} $\Leftrightarrow$ \ref{item:subspace_saddlepointsystem2}'': Observe with \autoref{prop:Koperator} that 
\begin{align*}
	\begin{pmatrix}
		\RJ[E_1] & \RY[A][][\dual]\\
		\RX[A][] & 0
	\end{pmatrix}\dual
	=
	\begin{pmatrix}
		\id_{} & 0\\
		0 & \RPsi[\RY[E_2]]
	\end{pmatrix}	
	\begin{pmatrix}
		\RK[E_1] & \RX[A][][\dual]\\
		\RY[A] & 0
	\end{pmatrix}
	\begin{pmatrix}
		\RPsi[\RY[E_1]]^{-1} & 0\\
		0 & \id_{}
	\end{pmatrix}.
\end{align*}
Hence either of the saddle point matrices is continuously invertible if and only if the other one is.
\newline
So far, we have shown that statements \ref{item:subspace_pinvexist}, \ref{item:subspace_saddlepointsystem}, and \ref{item:subspace_saddlepointsystem2} are pairwise equivalent. Now, suppose that any one of them (hence all of them) hold true.
The injections $\Ri[E_0]$ and $\Rj[E_0]$ are continuous and by \autoref{cor:kerimadense}, they are also dense. The identity $\RI[E_0]\,\Ri[E_0] = \Rj[E_0]\dual\,\RJ[E_0]$ holds by construction. Thus it suffices to show that $\RJ[E_0]$ is an isomorphism.
Since $\RJ[E_0]$ is injective, we only have to show that it is also surjective. Fix an arbitrary $\eta \in \RY[E_0][][\dual]$.
Denote by $\RZ[C] \colon \RZ[E_0] \hookrightarrow \RZ[E_1]$, $Z \in \{\RX,\RH,\RY\}$ the canonical injections and by 
$P \ceq \id_{E_1} - A\pinv A \in \RMor(E_1;E_1)$ the orthoprojector onto $E_0$
so that we have $\RY[P] \, \RY[C] = \RY[C]$.
For given $\eta \in \RY[E_0][][\dual]$ choose an extension $\tilde \eta \in \RY[E_1][][\dual]$ of $\eta$ (i.e., we have $\RY[C][][\dual] \, \tilde \eta = \eta$; the  Hahn-Banach theorem guarantees existence) and solve the
saddle point system
\begin{align*}
	\begin{pmatrix}
		\RJ[E_1] & \RY[A][][\dual]\\
		\RX[A] & 0
	\end{pmatrix}
	\begin{pmatrix}
		\tilde u\\
		\lambda
	\end{pmatrix}		
	=
	\begin{pmatrix}
		\tilde \eta \\
		0
	\end{pmatrix}
	\quad
	\text{with $\tilde u \in \RX[E_1]$ and $\lambda \in \RY[E_2][][\dual]$.}
\end{align*}
Note that we have $\RX[A]\,\tilde u = 0$, hence there is a $u \in \RX[E_0]$ with $\RX[C]\, u = \tilde u$. Now, we obtain
\begin{align*}
	\eta 
	&= \RY[C][][\dual] \, \tilde \eta 
	= \RY[C][][\dual] \, \bigparen{\RJ[E_1] \, \tilde u + \RY[A][][\dual] \, \lambda} 	
	= \RY[C][][\dual] \, \RJ[E_1] \, \RX[C] \, u 
	   + \RY[C][][\dual] \, \RY[P][][\dual]\, \RY[A][][\dual] \,  \lambda
	\\
	&= \RY[C][][\dual] \; \RJ[E_1] \, \RX[C] \,u
	= \RJ[E_0] \, u.
\end{align*}
Hence $\RJ[E_0]$ is an isomorphism.
\end{proof}

\begin{remark}\label{rem:saddlepointmatrixisuseful}
The saddle point matrix $\begin{pmatrix}
		\RJ[E_1] & \RY[A][][\dual]\\
		\RX[A] & 0
	\end{pmatrix}$
is of significant practical importance: Let $u_1 \in \RX[E_2]$ and $u_2 \in \RX[E_2]$. Then the orthogonal projection $u$ of $u_1$ onto $\ker(\RX[A])$ and the pseudoinverse $\RX[A][][\pinv] \, u_2$ of $u_2$ satisfy
\begin{align}
	\begin{pmatrix}
		\RJ[E_1] & \RY[A][][\dual]\\
		\RX[A] & 0
	\end{pmatrix} \, 
	\begin{pmatrix}
		u \\ \lambda
	\end{pmatrix}
	= 
	\begin{pmatrix}
		\RJ[E_2] \, u_1 \\ 0
	\end{pmatrix}
	\qand
	\begin{pmatrix}
		\RJ[E_1] & \RY[A][][\dual]\\
		\RX[A] & 0
	\end{pmatrix} \, 
	\begin{pmatrix}
		\RX[A][][\pinv] \, u_2 \\ \mu
	\end{pmatrix}
	= 
	\begin{pmatrix}
		0 \\ u_2
	\end{pmatrix}
	\label{eq:PSaddlepointSystem}
\end{align}
with suitable Lagrange multipliers $\lambda$, $\mu \in \RY[E_2][][\dual]$.
\end{remark}

\subsection{Characterization of Pseudoinvertible Operators}

Having realized the importance of surjective morphisms with pseudoinverse, we need a criterion which enables us to verify that a given morphism admits a pseudoinverse. Sometimes, a right-inverse can be directly constructed and can be easier analyzed than the morphism itself.

\begin{theorem}\label{lem:ApinvBpinv}
Let $E_1$ and $E_2$ be reflexive \UnnamedSpaces, 
let $A \in \RMor(E_1;E_2)$ be a morphism and let $B \in \RMor(E_2;E_2)$ be a right inverse of $A$.
Then $A$ admits a pseudoinverse if and only if $B$ does.
\end{theorem}
\begin{proof}
``$\Rightarrow$'': Suppose that $A\pinv$ exists. We are going to show that $B\adj B$ is invertible so that we have $B\pinv = (B\adj B)^{-1} B\adj$ (see \autoref{lem:pseudoinverseofinjection}).
As $\RH[B]$ is injective, $\RH[B][][\adj]\,\RH[B]$ must have trivial kernel and \autoref{cor:kerimadense} shows that $\RX[B][][\adj]\,\RX[B]$ and $\RY[B][][\adj]\,\RY[B]$ are also injective.
In order to show surjectivity, we define the projector $P \ceq A\pinv \, A$. By \eqref{eq:pseudoinverseformulasurjective}, we may write $P = A\adj\,(A\,A\adj)^{-1} A$.
With
$\ima(B) = \ima(P\,B) \oplus \ima((1-P)\,B)$
and $B^* A^* = (AB)^* = \id_{E_2}$, we obtain
\begin{align*}
	\ima(B\adj B) \supset B\adj \,(\ima(P\,B)) = \ima(B\adj \, P \, B) 
	= \ima (B^* A^* (A\,A^*)^{-1} A \, B)	
	= \ima((A\,A\adj )^{-1}) = E_2.
\end{align*}
Hence $B\adj B$ is surjective. By the open mapping theorem, $B\adj B$ is invertible.
\newline
``$\Leftarrow$'': Suppose that $B\pinv$ exists.
Observe that $\RH[(AA^*)] = \RH[A]\, \RH[A][][\adj]$ is invertible since $\RH[A]$ is surjective.
Hence \autoref{cor:kerimadense} implies that $AA\adj$ is injective.
Since $B\pinv$ exists, we may define the projector $Q \ceq B \, B\pinv$ and \eqref{eq:pseudoinverseformulainjective} tells us  $Q = B\,(B\adj B)^{-1} B\adj$.
Utilizing 
$\ima(A\adj) = \ima(Q\,A\adj) \oplus \ima((1-Q)\,A\adj)$
leads to
\begin{align*}
	\ima(A \, A\adj) 
	\supset A \,(\ima(Q\,A\adj)) 
	= \ima (A \, Q \, A\adj ) 
	= \ima( A \, B\,(B\adj B)^{-1} B\adj \, A\adj )
	= \ima((B\adj B)^{-1}) = E_2,
\end{align*}
which shows that $A \, A\adj$ is surjective. The open mapping theorem shows that $A \, A\adj$ is invertible and we obtain $A\pinv = A\adj \, (A \, A\adj)^{-1}$
\end{proof}

\begin{corollary}\label{cor:ApinvBstarBFredholm}
Let $E_1$ and $E_2$ be reflexive \UnnamedSpaces, 
let $A \in \RMor(E_1;E_2)$ be a surjective morphism.
Then $A$ admits a pseudoinverse if and only if it has a right inverse $B \in \RMor(E_2;E_2)$ such that $BB\adj$ is a Fredholm operator.
In particular, this is fulfilled whenever $E_2$ is finite-dimensional. 
\end{corollary}

\clearpage

\section{\CUnnamedManifolds}\label{sec:UnnamedManifolds}

Generalizing \UnnamedSpaces to a nonlinear setting necessitates the use of vector bundles\footnote{We use the terms \emph{vector bundle} and \emph{continuous, locally trivial vector bundle} synonymously.}, in particular of Banach bundles, i.e., vector bundles whose fibers are Banach spaces. 
In a nutshell, vector bundles are continuous families of topological vector spaces parameterized over a further topological space, the so-called base space. 
For a brief introduction to vector bundles, we refer the reader to \cite{MR1335233}.

\subsection{Basic Definitions}

\begin{definition}\label{dfn:Rieszbundlestructures}
Let $R \in\set{C^{k,\alpha}_\loc| 
\text{$k\in \N\cup\{0,\infty\}$, $\alpha \in \intervalcc{0,1}$}}$.
Let $\RM$ be a Banach manifold of class $R$.
A \emph{\UnnamedBundle} $\pi \colon E \to \RM$ over $\RM$ consists of
\begin{enumerate}
\item two Banach bundles $\pi_{\RX[E]} \colon \RX[E] \to \RM$, $\pi_{\RY[E]} \colon \RY[E] \to M$ and of a Hilbert bundle $\pi_{\RH[E]} \colon \RH[E] \to M$;
\item a chain of continuous, linear, and fiberwise dense bundle injections
\begin{equation*}
\begin{tikzcd}[]
	\RX[E]
		\ar[r,hook, two heads,"{\Ri[E]}"] 
	&\RH[E]
		\ar[r,hook, two heads,"{\Rj[E]}"] 
	&
	\RY[E];
\end{tikzcd}
\end{equation*}
\item a fixed bundle inner product $\ninnerprod{\cdot,\cdot}_{\RH[E]}$ and its Riesz isomorphism $\RI[E] \colon \RH[E] \to \RH[E][][']$;
	\item and an isomorphism $\RJ[E] \colon \RX[E] \to \RY[E][][']$ of Banach bundles
\end{enumerate}
such that the following diagram commutes:
\begin{equation}
\begin{tikzcd}[]
	\RX[E]
		\ar[dr,"\pi_{\RX[E]}"']
		\ar[dd,hook, two heads,"{\Ri[E]}"'] 
		\ar[rr,"{\RJ[E]}","\cong"'] 
	&&\RY[E][][']
		\ar[dl,"\pi_{\RY[E][][\dual]}"]	
		\ar[dd,hook, two heads,"{\Rj[E]\dual}"] 
	\\
	&\RM
	\\
	\RH[E]
		\ar[ur,"{\pi_{\RH[E]}}"]		
		\ar[rr,"{\RI[E]}"',"\cong"]	
	&&\RieszH'
		\ar[ul,"\pi_{\RH[E][][\dual]}"']\nospaceperiod
\end{tikzcd}\label{eq:Rieszbundlesquarediagram}
\end{equation}
For a point $a \in M$, we write $\RX[E][a]$, $\RH[E][a]$, $\RY[E][a]$, $\RX[E][a][']$, \dots etc. for the fibers over $a$ of the according Banach bundles.
We say that $E$ is a \UnnamedBundle of class $R$, if 
\begin{enumerate}
	\item $\RX[E]$, $\RH[E]$ and $\RY[E]$ are Banach bundles of class $R$ and
	\item $\Ri[E]$, $\Rj[E]$, $\RI[E]$, and $\RJ[E]$ are Banach bundle morphisms of class $R$.	
\end{enumerate}
\invisible{\autoref{dfn:Rieszstructures}.}
Moreover, we refer to linear and continuous bundle chain maps (of class $R$) between \UnnamedBundles (of class $R$) as \UnnamedMorphisms (of class $R$) and to the space of all \UnnamedMorphisms between \UnnamedBundles $E_1$ and $E_2$ as $\RMor(E_1;E_2)$.
\end{definition}

\begin{definition}\label{dfn:Rieszmanifold}
Let $k \in \N \cup\{\infty\}$ and $\alpha \in \intervalcc{0,1}$.
A \emph{\UnnamedManifold of class $C^{k,\alpha}_{\on{loc}}$} is a Banach manifold $\RM$ of class $C^{k,\alpha}_{\on{loc}}$ together with a \UnnamedBundle $\pi \colon E \to \RM$ of class $C^{k-1,\alpha}_{\on{loc}}$ with $\T \RM = \RX[E]$. In this case we also write
$\RX[M]$, $\RH[M]$, $\RY[M]$, $\Ri[M]$, $\Rj[M]$, $\RI[M]$, $\RJ[M]$, \dots{} etc. for 
$\RX[E]$, $\RH[E]$, $\RY[E]$, $\Ri[E]$, $\Rj[E]$, $\RI[E]$, $\RJ[E]$, \dots{} etc., respectively.

Let $(\RM_1,E_1)$ and $(\RM_2, E_2)$ be \UnnamedManifolds and let $F \colon \RM_1 \to \RM_2$ be a map. We say that $F$ is a \emph{morphism of \UnnamedManifolds of class $C^{k,\alpha}_{\on{loc}}$} if $F$ is of class $C^{k,\alpha}_{\on{loc}}$ and if there is a morphism $A \colon E_1 \to E_2$ of \UnnamedBundles of class $C^{k-1,\alpha}_{\on{loc}}$ over $F$ with $\RX[A] = \T F$.\footnote{Because of the identity lemma (see \autoref{lem:nulllemma}), the morphism $A$ is unique if it exists.} In this case, we also write $\RH[F]\ceq \RH[A]$ and $\RY[F]\ceq \RY[A]$ such that we obtain the following commutative diagram:
\begin{equation*}
\begin{tikzcd}[column sep=2em, row sep=1.2em]
	\RT[\RM_1]
		\ar[rrr,"{\RT[F]}"]
		\ar[rrd,hook, two heads, "{\Ri[M_1]}"]
		\ar[rrddd]		
	&
	&
	&
	\RT[\RM_2]
		\ar[rrd,hook, two heads, "{\Ri[M_2]}"]			
		\ar[rrddd]				
	\\
	&
	&\RH[\RM_1]
		\ar[rrr,"{\RH[F]}",  pos=0.3, crossing over]	
		\ar[dd]	
		\ar[rrd,hook, two heads, "{\Rj[M_1]}"]				
	&
	&
	&\RH[\RM_2]
		\ar[dd]		
		\ar[rrd,hook, two heads, "{\Rj[M_2]}"]				
	\\
	&
	&
	&
	&\RY[\RM_1]
		\ar[rrr,"{\RY[F]}",  pos=0.7, crossing over]	
		\ar[lld]						
	&
	&
	&
	\RY[\RM_2]
		\ar[lld]							
	\\
	&
	&\RM_1
		\ar[rrr,"{F}"']	
	&
	&
	&\RM_2
	\nospaceperiod
\end{tikzcd}
\invisible{\autoref{dfn:Rieszbundlestructures}}
\end{equation*}
\end{definition}

Every Riemannian manifold is a \UnnamedManifold in a natural way. Moreover, every finite-dimensional submanifold of a \UnnamedManifold is indeed a Riemannian manifold. Hence, one has to watch out for infinite dimensional manifolds in order to find a \UnnamedManifold that is not a Riemannian manifold. We will meet some examples in \autoref{prop:W2pisparaRiemannian} and \autoref{prop:W2pisparaRiemannian2}.

\subsection{Gradient Flow}

\begin{proposition}\label{prop:Gradient}
Let $k \in \N \cup\{\infty\}$ and $\alpha \in \intervalcc{0,1}$.
Let $(M,E)$ be a \UnnamedManifold of class $C^{k,\alpha}_{\on{loc}}$ and let $F \colon M \to \R$ be a morphism of class $C^{k,\alpha}_{\on{loc}}$.
Then the \emph{gradient} $\grad_E(F)$ defined by
\begin{align*}
	\grad_E(F)\at_x \ceq \RY[F][x][\adj] \cdot 1 = \RJ[E]^{-1}\, \RY[F][x]
	\quad
	\text{for $x \in M$}
\end{align*}
is a vector field of class $C^{k-1,\alpha}_{\on{loc}}$ on the Banach manifold $M$.
\invisible{\autoref{dfn:Rieszmanifold}}
It is always ascending with respect to $F$, i.\,e., one has $\ninnerprod{\dd F\at_x , \grad_E(F)\at_x} \geq 0$ and with equality if and only if $x$ is a critical point of $F$. 
If $F$ is of class $C^{1,1}_{\on{loc}}$ then both its forward and backward flow exist for short times.
\end{proposition}
\begin{proof}
With the abbreviation $u = \grad_E(F)\at_x$, we obtain
\begin{align*}
	\ninnerprod{\dd F\at_x , u}
	&= \ninnerprod{\RJ[E]\at_x \; u, \Rj[E] \; \Ri[E] \; u}
	 = \ninnerprod{\RI[E]\at_x \; \Ri[E] \; u, \Ri[E] \; u}
	 = \nabs{\Ri[E] \; u}_{\RH[E][x]}^2 \geq 0.
\end{align*}
Since $\Ri[E]$ is injective, this can only vanish if $u=0$ which is equivalent to $\dd F\at_x = 0$.
If $F$ is of class $C^{1,1}_{\on{loc}}$, then $\grad(F)$ is locally Lipschitz continuous and
the Picard-Lindel\"off theorem implies short time existence of the flows induced by 
$\grad(F)$ and $-\grad(F)$.
\end{proof}

\subsection{Submanifolds}

The following theorem takes the role of the implicit function theorem in the category of \UnnamedManifolds.

\begin{theorem}\label{theo:SubmersionSubmanifold}
Let $k \in \N \cup\{\infty\}$ and $\alpha \in \intervalcc{0,1}$.
Let $\varPhi \colon M_1 \to M_2$ be a morphism of class $C^{k,\alpha}_{\on{loc}}$ between \UnnamedManifolds $(M_1,E_1)$  and $(M_2,E_2)$ of class $C^{k,\alpha}_{\on{loc}}$.
Denote the induced bundle morphism by $A = (\RT[\varPhi][x],\RH[\varPhi][x],\RY[\varPhi][x]) \in \RMor(E_1 ; E_2)$.
For a given point $y_0 \in M_2$, consider the set $M_0\ceq \varPhi^{-1}(\{y_0\}) \subset M_1$ and suppose that for each $x \in M_0$, the morphism $A_x$ is surjective.
Define the Banach bundle $E_0 \ceq \ker(A)\at_{M_0}$ over $M_0$ and bundle morphisms $\RI[E_0] \colon \RH[E_0][] \to \RH[E_0][][\dual]$ and $\RJ[E_0] \colon \RX[E_0][] \to \RY[E_0][][\dual]$ similarly to \eqref{eq:SubspaceMetrics} by
\begin{align}
	\arraycolsep=1.4pt
	\begin{array}{rclrcll}
		\RI[E_0]\at_x \colon \RH[E_0][x] &\to \RH[E_0][x][\dual],
		& \quad\ninnerprod{\RI[E_0]\at_x\, v_1, v_2} &\ceq &\ninnerprod{\RI[E_1]\at_x\, v_1, v_2}
		& \quad \text{for $v_1$, $v_2 \in \RH[E_0][x]$},
		\\
		\RJ[E_0]\at_x \colon \RX[E_0][x] &\to \RY[E_0][x][\dual],
		& \quad\ninnerprod{\RJ[E_0]\at_x\, u, w} &\ceq &\ninnerprod{\RJ[E_1]\at_x\, u, w}
		& \quad \text{for $u \in \RX[E_0][x]$, $w \in \RY[E_0][x]$}.
	\end{array}
	\label{eq:SubmanifoldMetrics}
\end{align}
Then $(M_0,E_0)$ is a \UnnamedManifold of class $C^{k,\alpha}_{\on{loc}}$.
\end{theorem}
\begin{proof}
By the implicit function theorem, $M_0$ is a Banach submanifold of class $C^{k,\alpha}_{\on{loc}}$.  Moreover, we have $T_xM_0 = \ker(\RT[\varPhi][x]) = \RX[E_0][x]$  for each $x \in M_0$.
We infer that $E_0\at_x$ is a \UnnamedSpace from \autoref{theo:SubspaceTheorem}. Since the family of projectors $\RZ[P][x] \ceq \RZ[(A\pinv A)][x]$ is continuous and since we have $\RZ[E_0][x]= \ker(\RZ[A][x]) = \ker(\RZ[P][x])$, we may deduce that $\RZ[E_0]$ is a vector bundle of class $C^{k-1,\alpha}_{\on{loc}}$ over $M_0$ for each $\RZ \in \{\RX, \RH,\RY\}$. 
\end{proof}

Note that \autoref{cor:ApinvBstarBFredholm} implies that $M_0$ is a \UnnamedManifold whenever $\varPhi \colon M_1 \to M_2$ is a morphism with surjective linearization into a \emph{finite-dimensional} \UnnamedManifold $M_2$.

\subsection{Projected Gradient Flow}

For a submanifold $(M_0,E_0) \subset (M_1,E_1)$ as above, the \emph{restricted} or \emph{projected gradient} flow is not only existent but also computable from $\RJ[E_1]$ and the from linearization of the constraint mapping $\varPhi$. This is crucial for our numerical applications in \autoref{sec:NumericalExamples} below.

\begin{proposition}\label{prop:ProjectedGradient}
Suppose the setting of \autoref{theo:SubmersionSubmanifold}. Let $F \colon M_1 \to \R$ be a morphism of class $R=C^{1,1}_{\on{loc}}\cap C^{k,\alpha}_{\on{loc}}$. Then also $F_0 \ceq F|_{M_0} \colon M_0 \to \R$ is a morphism of class $R$. By \autoref{prop:Gradient}, we have short time existence of the gradient flow of $F_0$. Note that for $x \in M_0$, the gradient $\grad_{E_0}(F_0)\at_x$ is identical to the orthoprojection of $\grad_{E_1}(F)\at_x$ onto $\RX[E_0][x] = \ker(\RX[\varPhi][x])$. By \autoref{rem:saddlepointmatrixisuseful}, there is a Lagrange multiplier $\lambda(x) \in \RY[E_2][\varPhi(x)][\dual]$ with
\begin{align*}
	\begin{pmatrix}
		\RJ[E_1]\at_x 	& \RY[\varPhi][x][\dual]
		\\
		\RX[\varPhi][x]			& 0
	\end{pmatrix}
	\begin{pmatrix}
		\grad_{E_0}(F)\at_x
		\\
		\lambda(x)
	\end{pmatrix}
	=
	\begin{pmatrix}
		\RJ[E_1] \grad_{E_1}(F)\at_x
		\\
		0
	\end{pmatrix}
	=
	\begin{pmatrix}
		\RY[F][x]
		\\
		0
	\end{pmatrix}.
\end{align*}
\end{proposition}

\begin{figure}[ht]
\capstart
\newcommand{\settrimming}{%
    \setkeys{Gin}{%
        trim = 20 55 0 105 , clip=true, 
        width=0.32\textwidth
    }
    \presetkeys{Gin}{clip}{}
}
\settrimming
\begin{center}
\myincludegraphics{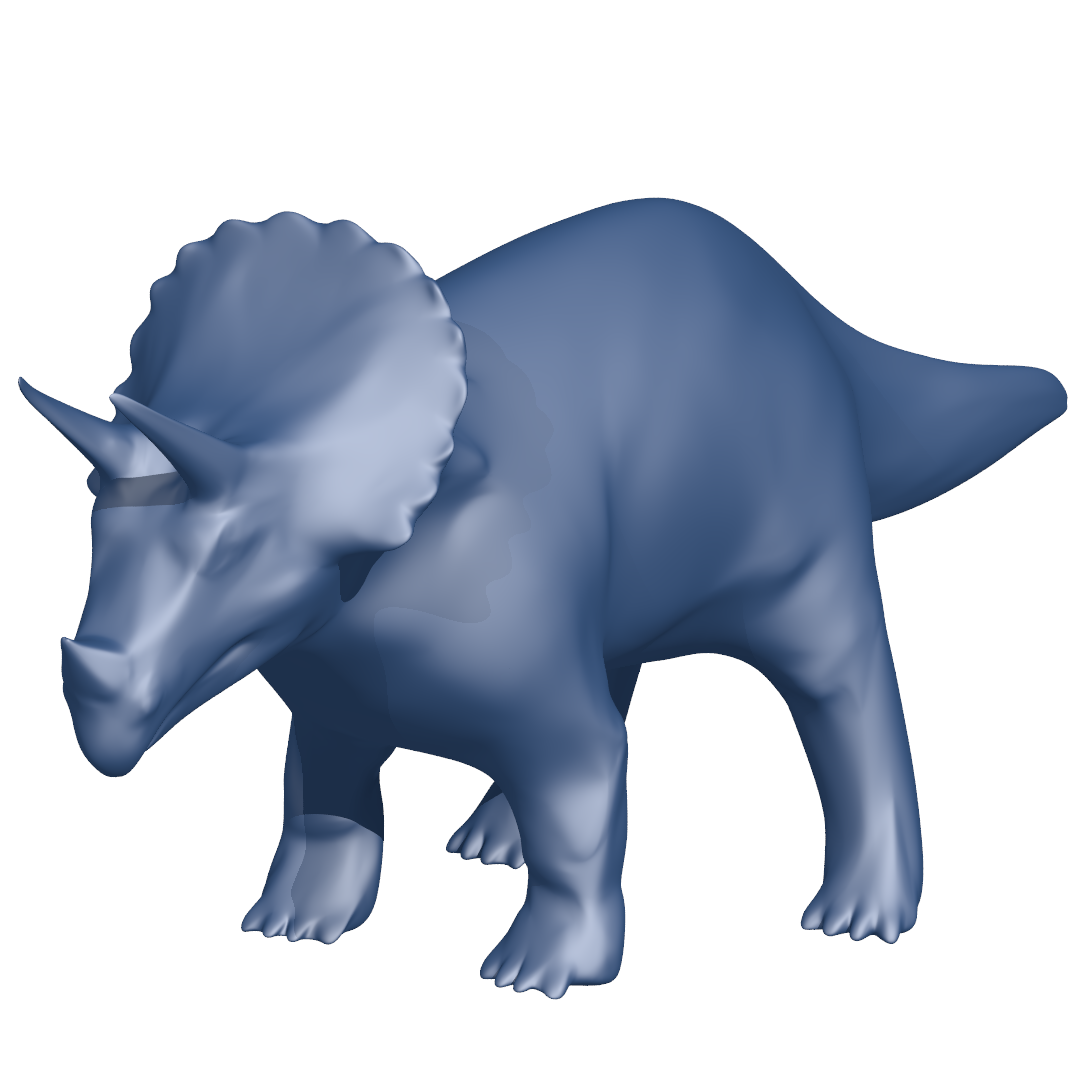}{a}
\myincludegraphics{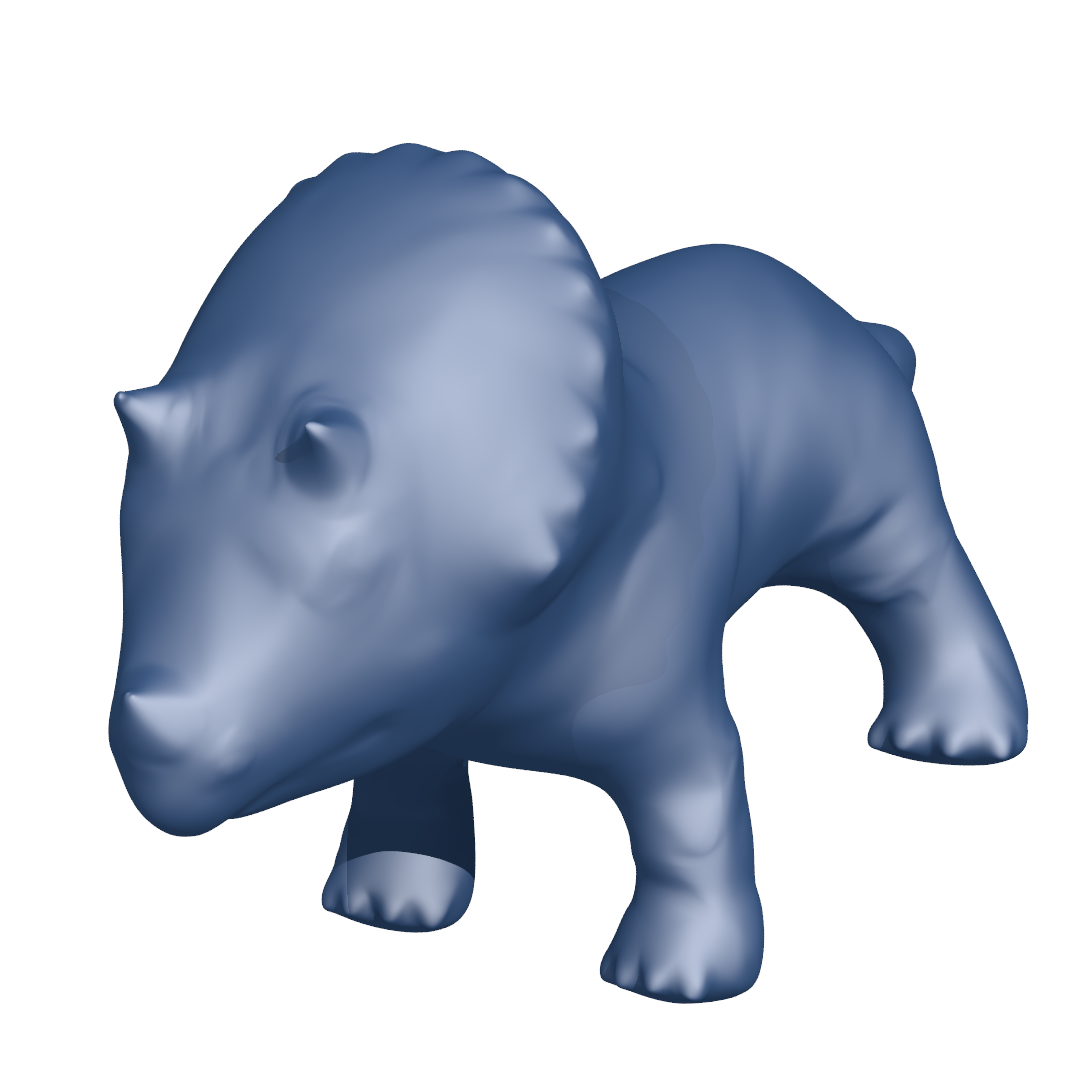}{b}
\myincludegraphics{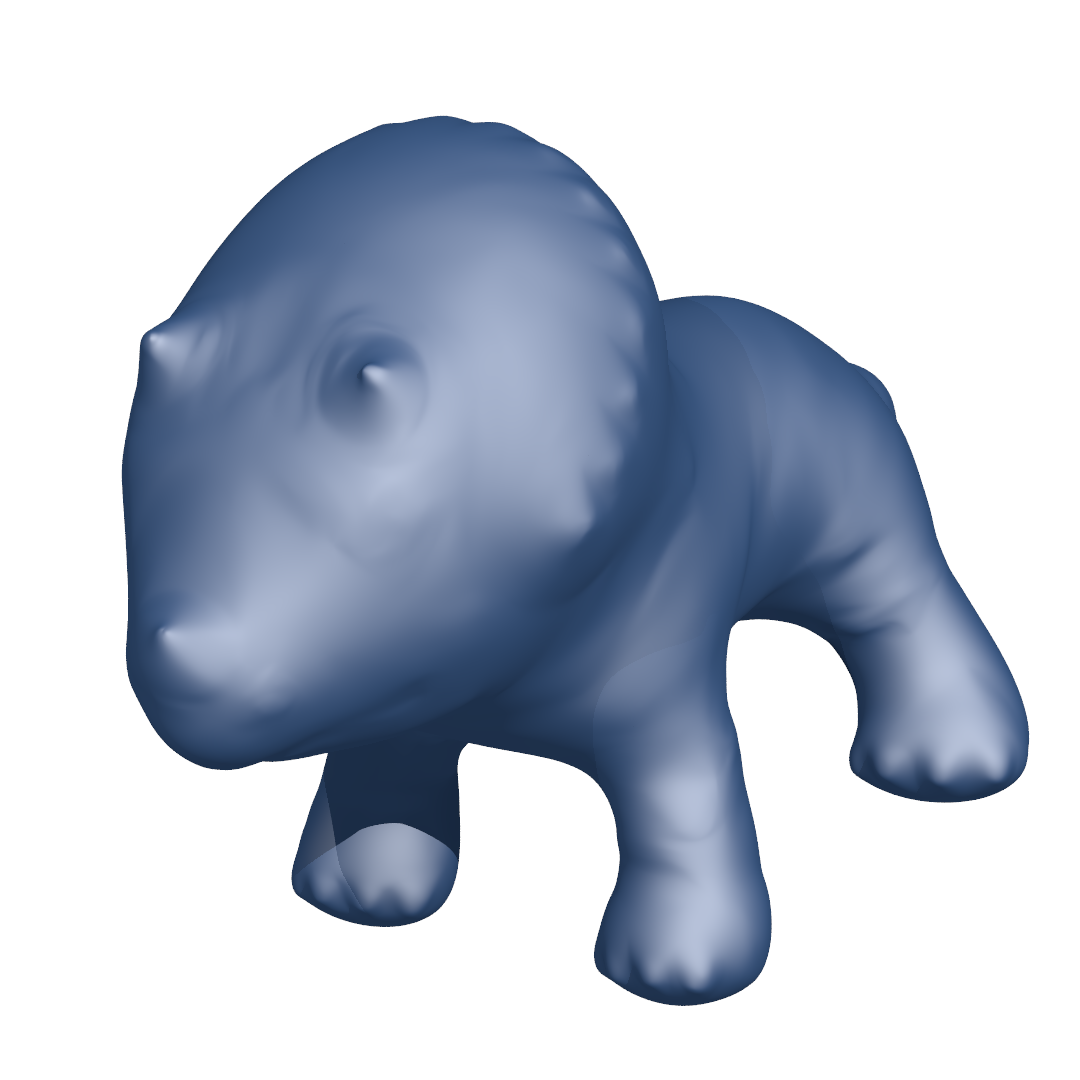}{c}
\\
\myincludegraphics{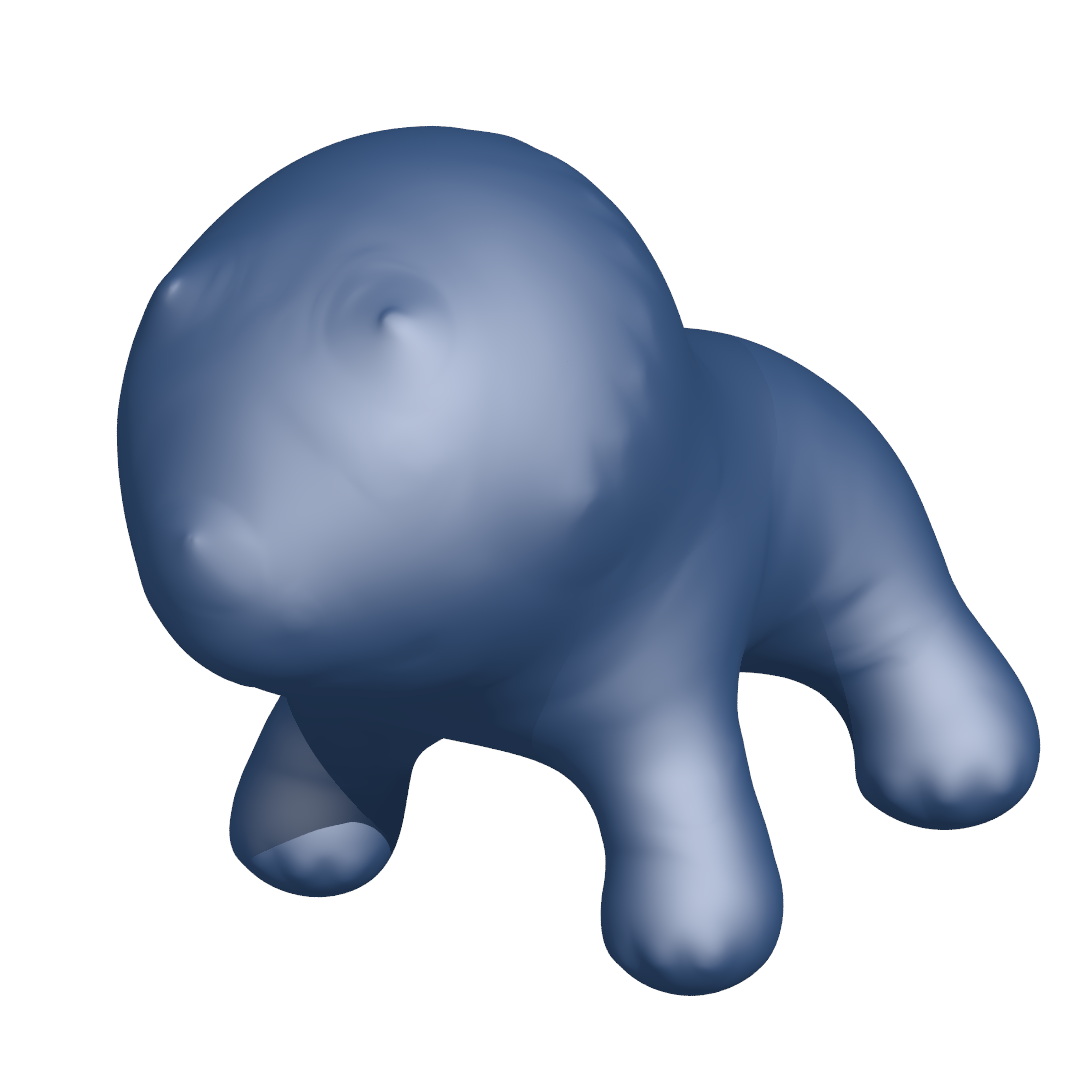}{d}
\myincludegraphics{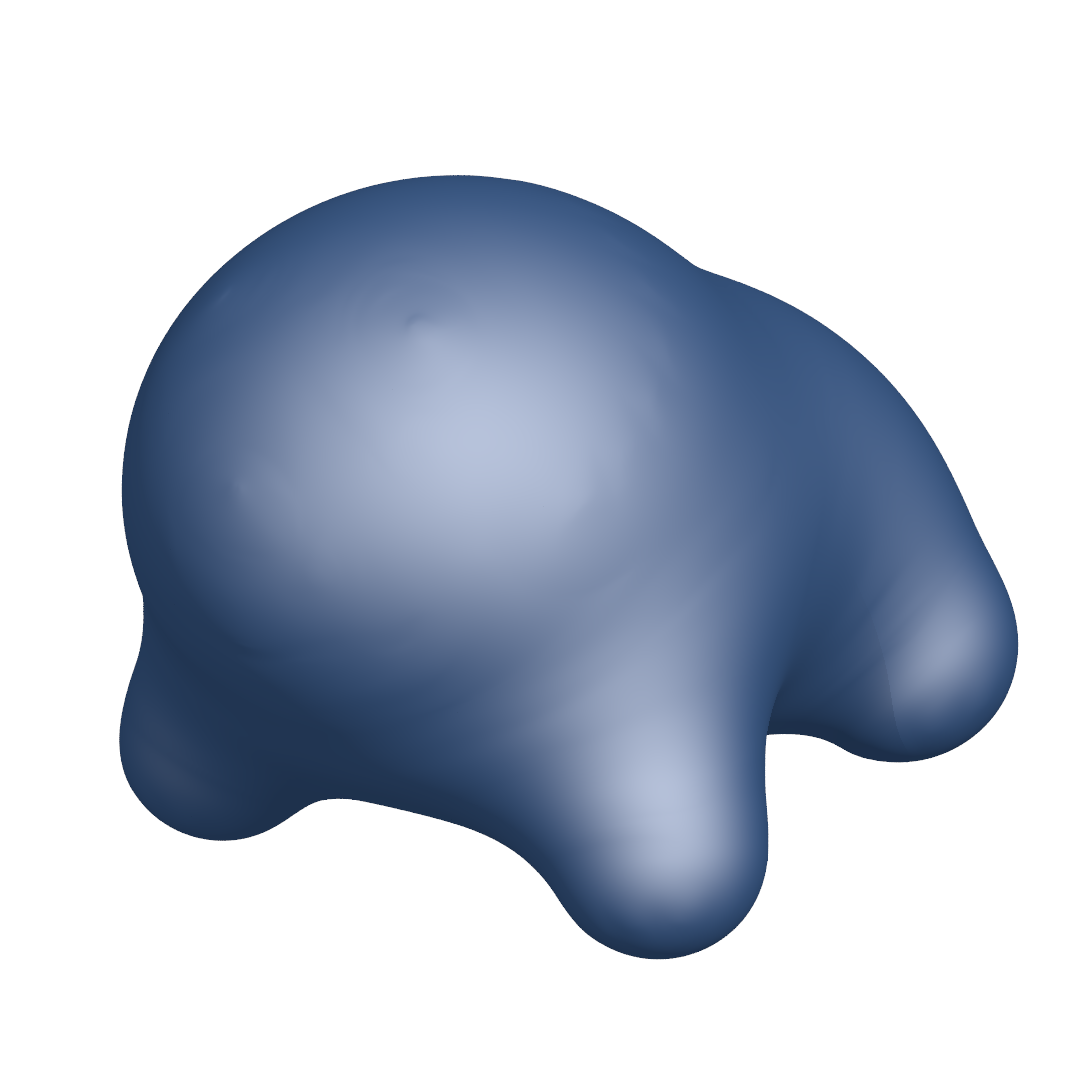}{e}
\myincludegraphics{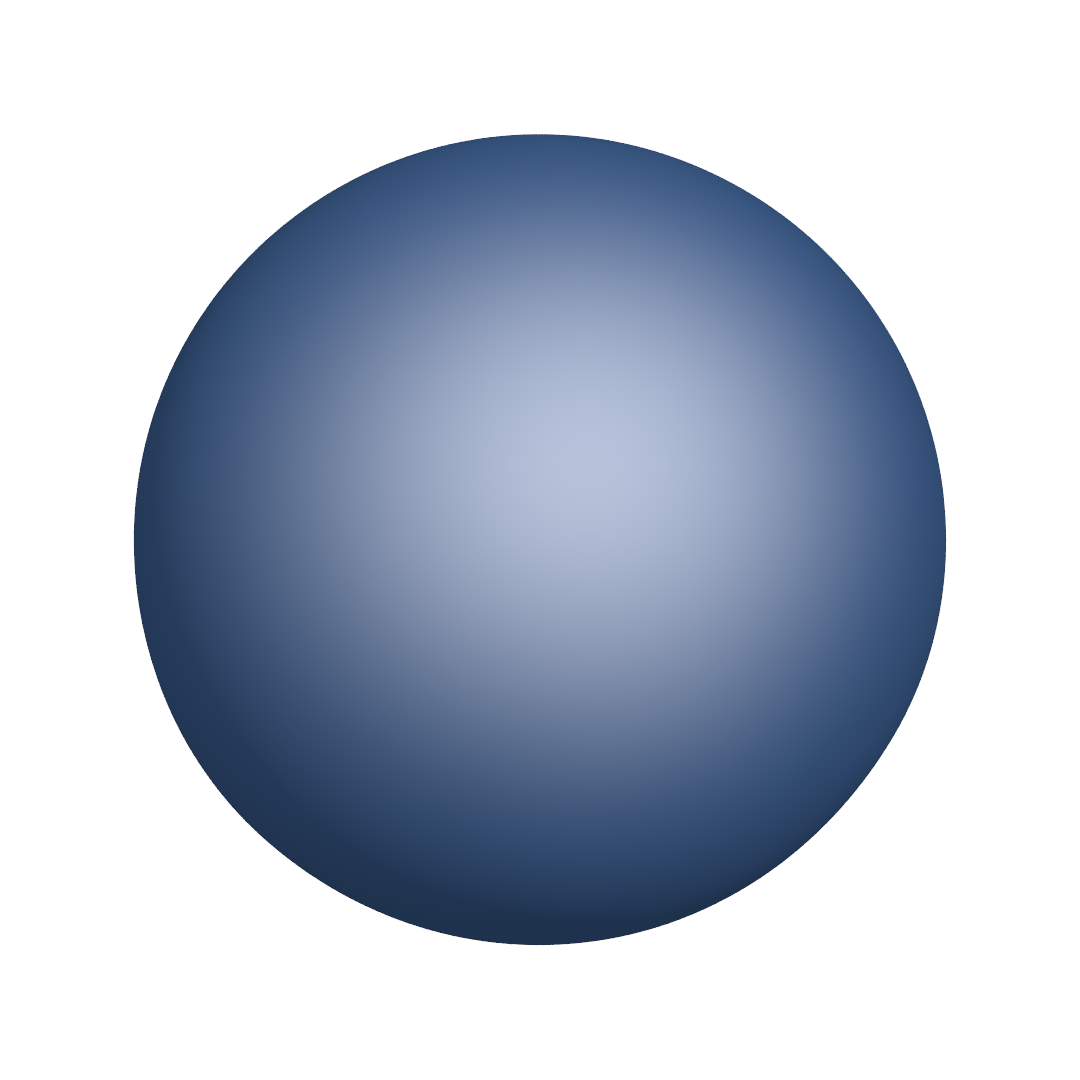}{f}
\caption{Discrete surface ($1.4$ million faces) relaxing under $H^2$-gradient descent for the Willmore energy subject to equality constraints on barycenter and total area:
(a) initial condition; (b)--(e) iterations 6, 12, 18, 24, and 30, respectively.
}
\label{fig:Triceratops}
\end{center}
\invisible{\autoref{theo:elasticafunctionalisRiesz}\autoref{theo:submersionsubmanifold}}
\end{figure}

\section{Curvature Energies}\label{sec:CurvatureEnergies}

In the following, we fix a compact, $n$-dimensional, smooth manifold $\varSigma$.
We are going to formulate curvature dependent energies such as the elastica energy or the Willmore energy on the space of immersions
\begin{align*}
 	\Imm^{2,p}(\varSigma;\AmbSpace) \ceq \set{f \in \sobo{2}{p}[\varSigma][\AmbSpace] | \text{for all $x \in \varSigma$: $\dd f\at_x$ is injective}},
 	\quad
 	\text{for $p>n$.}
\end{align*} 
Let $f \in \sobo{2}{p}[\varSigma][\AmbSpace]$, fix a smooth Riemannian metric $G$ on $\varSigma$ as reference, and denote the Euclidean metric on $\AmbSpace$ by $\AmbMetric$.
By the Morrey embedding $\sobo{2}{p}[\varSigma][\AmbSpace] \hookrightarrow C^{1,\alpha}(\varSigma;\AmbSpace)$ with $\alpha = 1-n/p>0$, the pointwise derivative of $f$ exists and it is $\alpha$-H\"{o}lder continuous. This also shows that the set $\Imm^{2,p}(\varSigma;\AmbSpace) $ is an open subset of $\sobo{2}{p}[\varSigma][\AmbSpace]$.
For $f \in \Imm^{2,p}(\varSigma;\AmbSpace)$, the \emph{pullback metric} $f^\pull \AmbMetric \ceq \AmbMetric(\dd f \, \cdot, \dd f \, \cdot)$ is a Riemannian metric of class $\sobo{1}{p}$ and there are constants $0<\lambda\leq \varLambda<\infty$ so that
$\lambda \, G \leq f^\pull \AmbMetric \leq \varLambda \, G$.
Moreover, $f^\pull \AmbMetric$ induces a \emph{Riemannian density} $\vol_f \ceq \vol_{f^\pull \AmbMetric}$ of class $\sobo{1}{p}$. 

\bigskip

Before we consider curvature functionals, we need the notion of the \emph{second fundamental form} of an immersion. There are various ways to introduce it. We decided to introduce it by utilizing the Hessian of a vector-valued function with respect to $f^\pull \AmbMetric$.

\begin{proposition}\label{prop:Hessian}
Let $f \in \Imm^{2,p}(\varSigma;\AmbSpace)$ with $p>\dim(\varSigma)$ and let $r \in \intervalcc{1,\infty}$. For $u \in \sobo{1}{r}[\varSigma][\AmbSpace]$, we define the \emph{Hessian of $u$ with respect to $f^\pull \AmbMetric$} by
\begin{align*}
	\Hess[f](u)(X,Y) = (\dd(\dd u \, \dd f^\dagger)\,X) \cdot (\dd f \,Y),
	\quad
	\text{for all smooth vector fields $X$ and $Y$  on $\varSigma$.}
\end{align*}
Then
$
 	\Hess[f] \colon \sobo{2}{r}[\varSigma][\AmbSpace] \to  L^{\min(p,r)}(\varSigma;\Sym^2(T\varSigma;\AmbSpace))
$
is a well-defined and continuous operator which depends smoothly on $f$.\footnote{We point out for geometers that $\Hess[f](u)$ coincides with the Hessian $\nabla^{f^\pull \AmbMetric} \dd u$, provided that $f$ and $u$ are sufficiently smooth. Of course, $\nabla^{f^\pull \AmbMetric}$ denotes the Levi-Civita connection of the Riemannian metric $f^\pull \AmbMetric$}
\end{proposition}
\begin{proof}
Let $u \in \sobo{2}{r}[\varSigma][\AmbSpace]$ so that $\dd u$ is an element of $\sobo{1}{r}[\varSigma][\Hom(T\varSigma;\AmbSpace)]$. 
Since $\dd f$ has always maximal rank, we may write $\dd f^\dagger = (\dd f^{*}\dd f)^{-1}\dd f^{*}$,
where the adjoint can taken with respect to any \emph{arbitrary} smooth Riemannian metric $G$ on $\varSigma$.
Hence we deduce $\dd f^\dagger \in \sobo{1}{p}[\varSigma][\Hom(\AmbSpace;T\varSigma)]$ and \autoref{lem:regularityofproducts} below shows that $\dd u \, \dd f^\dagger \in  \sobo{1}{\min(p,r)}[\varSigma][\Hom(\AmbSpace;\AmbSpace)]$ so that $\Hess[f](u)$ is well-defined and continuous in $u$.

In order to show that $\Hess[f]$ depends smoothly on $f$, we first observe that the Moore-Penrose pseudoinverse restricted to linear maps of fixed rank is a smooth transformation. 
A~concise formula for its derivative can be found in \cite{MR0336980}. It allows us to deduce for each $v \in \sobo{2}{p}[\varSigma][\AmbSpace]$:
\begin{align}
		D(f\mapsto \dd f^\dagger)\,v 
	&= -\dd f^\dagger \, \dd v \, \dd f^\dagger + \dd f^\dagger(\dd v \, \dd f^\dagger)^*(\id_\AmbSpace - \dd f \, \dd f^\dagger).
	\label{eq:derivativeofMoorePenrose}
\end{align}
Successive applications of this formula and of \autoref{lem:regularityofproducts} show that $f\mapsto \dd f^\dagger$ and $f \mapsto \Hess[f]$
are smooth, provided that $f \in \Imm^{2,p}(\varSigma;\AmbSpace)$.
\end{proof}

\begin{definition}\label{dfn:elasticafunctional}
Let $p \in \Imm^{2,p}(\varSigma;\AmbSpace)$ with $p\in \intervaloo{\dim(\varSigma),\infty} \cap \intervalco{2,\infty}$.
The \emph{second fundamental form $\II(f)$ of $f$} can be written as $\II(f) \ceq \Hess[f](f)$ so that one has $\II(f) \in  L^p(\varSigma; \Sym^2(T\varSigma;\AmbSpace))$. 
We define the \emph{elastica functional} $\ElasticEnergy \colon \ConfSpace \to \R$ and the \emph{Willmore energy} $\Willmore \colon \ConfSpace \to \R$ by
\begin{align*}
	\ElasticEnergy(f) = \int_\varSigma \nabs{\II(f)}_{f^\pull\AmbMetric}^2 \, \vol_{f}
	\qand
	\Willmore(f) \ceq \int_\varSigma \nabs{H(f)}^2 \, \vol_{f},
\end{align*}
where $H(f) = \frac{1}{\dim(\varSigma)} \tr_{f^\pull \AmbMetric} \II(f) = \frac{1}{\dim(\varSigma)} \, \Delta_f \, f$ is the \emph{mean curvature vector} of $f$.
\end{definition}

Next, we equip certain subspaces of $\Imm^{2,p}(\varSigma;\AmbSpace)$ with a \UnnamedManifold structure that fits well with the energies $\ElasticEnergy$ and $\Willmore$.
As already mentioned in the introduction, we would like to use Riesz isomorphisms induced by the bilinear forms $b_1$ and $b_2$ from \eqref{eq:bilinearformsonImm}.
Contrary to $b_2$, $b_1$ is not positive definite, hence we have to impose some constraints on the space of immersions in order to obtain a \UnnamedManifold.
This additional effort is justified by our experimental observations: They indicate that $b_1$ tends to generate sharper gradient search directions than $b_2$.

In order to streamline the exposition, we only consider connected manifolds.
While we impose Dirichlet boundary conditions in the case of nontrivial boundary, we fix the barycenter of the immersed surface in order to eliminate the kernel of $\Delta_f$.

\begin{figure}
\capstart
\newcommand{\settrimming}{%
    \setkeys{Gin}{%
        trim = 0 0 0 0 , clip=true, 
        width=0.32\textwidth
    }
    \presetkeys{Gin}{clip}{}
}
\settrimming
\begin{center}
\includegraphics{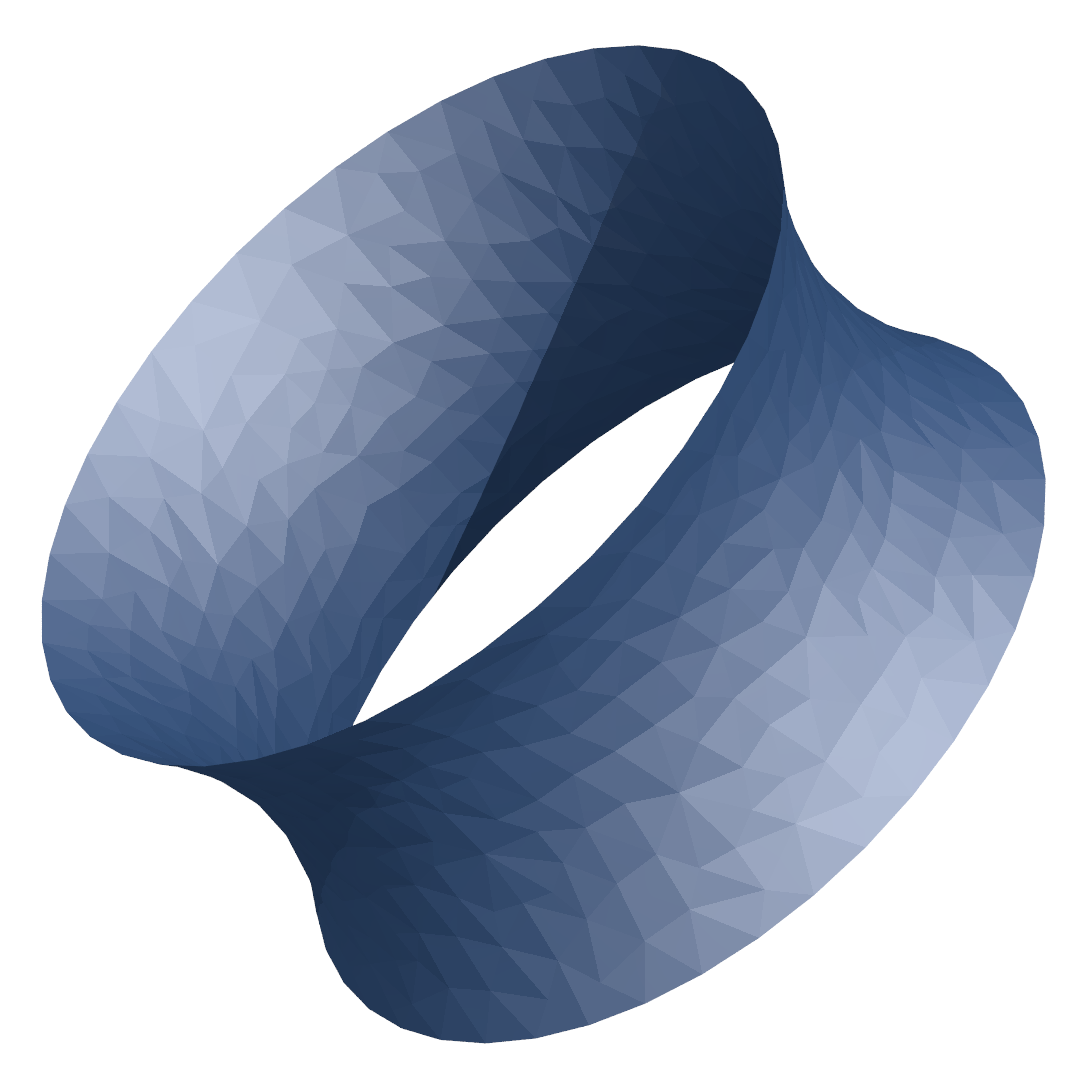}
\includegraphics{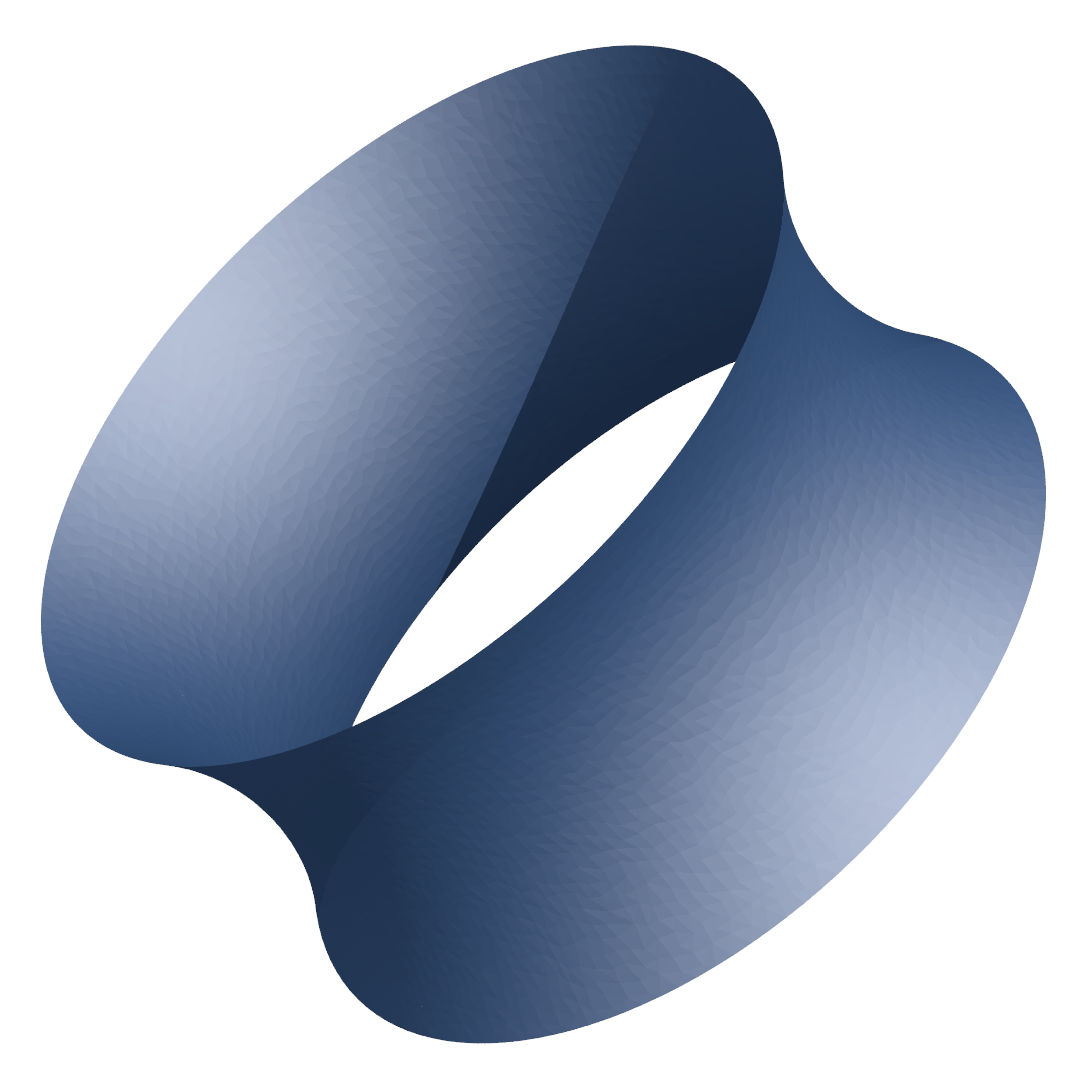}
\includegraphics{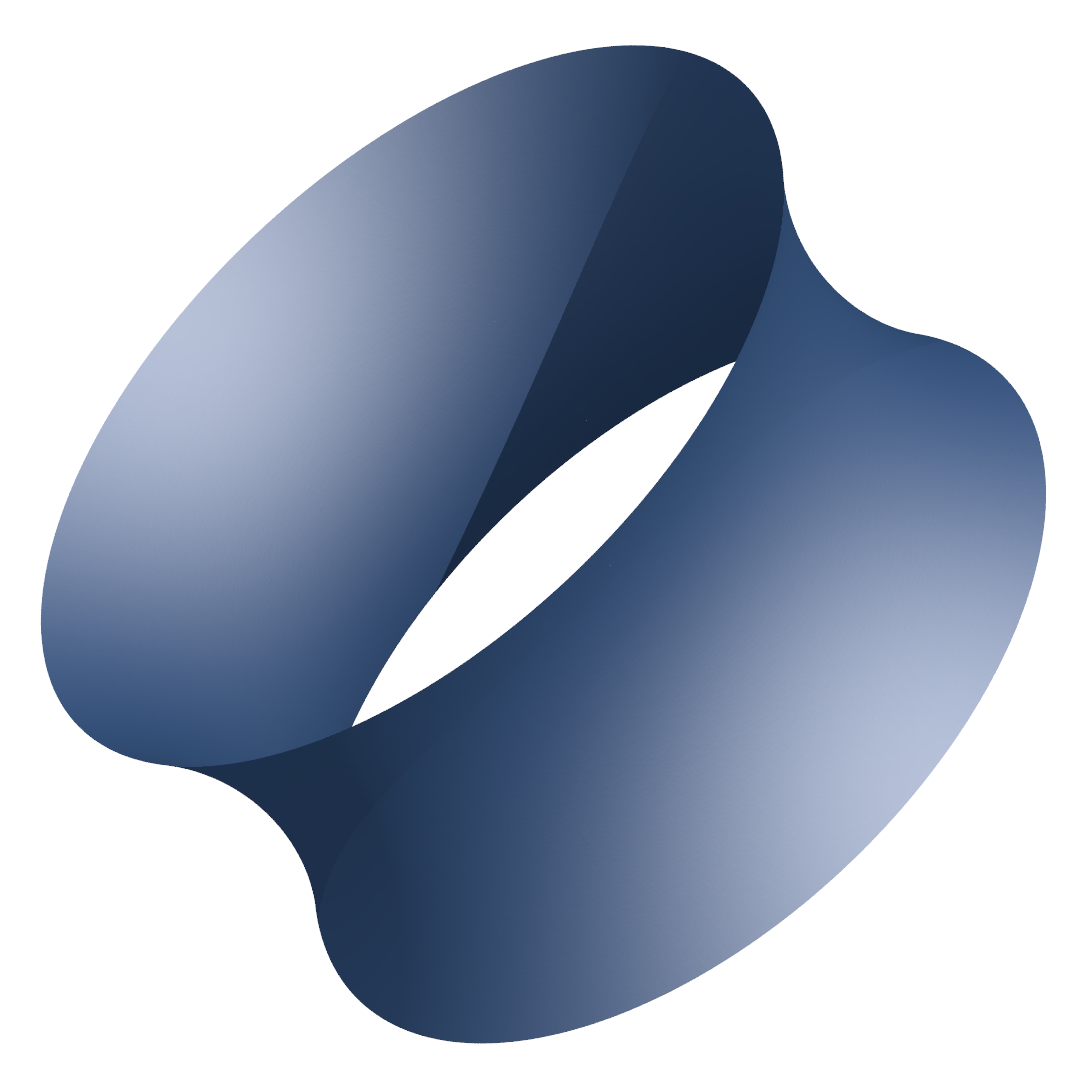}
\caption{Minimizers of the discrete Willmore energy subject to Dirichlet boundary conditions at three different mesh resolutions. They happen to be discrete minimal surfaces (the discrete mean curvature vanishes).
}
\label{fig:Cylinders1}
\end{center}
\end{figure}

\begin{proposition}\label{prop:W2pisparaRiemannian}
Let $\varSigma$ be a connected, $n$-dimensional, compact manifold with nontrivial boundary of class $C^{1,1}$, $p \in \nintervaloo{n,\infty}\cap\nintervalco{2,\infty}$, and
$\gamma \in \Imm^{s,p}(\partial \varSigma;\AmbSpace)$ with $s\ceq 2-1/p$.
Denote the H\"{o}lder conjugate of $p$ by $q = (1-1/p)^{-1}$.
Define $\ConfSpace \ceq \set{f \in \Imm^{2,p}(\varSigma;\AmbSpace) | f\at_{\partial \varSigma} = \gamma}$ and for $f \in \cC$ define the Banach spaces	
\begin{gather*}
	\RX[E][f] \ceq \soboo{2}{p}[\varSigma][\AmbSpace],
	\qquad
	\RH[E][f] \ceq \soboo{2}{2}[\varSigma][\AmbSpace],
	\\
	\qquad \text{and} \qquad 
	\RY[E][f] \ceq \soboo{2}{1}[\varSigma][\AmbSpace].
\end{gather*}
Denote the canonical bundle injections by $\Ri[E] \colon \RX[E] \hookrightarrow \RH[E]$ and $\Rj[E] \colon \RH[E] \hookrightarrow \RY[E]$ and define the mappings
\begin{align}
	\RI[E]\at_f \colon \RH[E][f] &\to \RH[E][f][\dual],
	&
	\ninnerprod{ \RI[E]\at_f \; v_1 , \; v_2}
	&\ceq \textstyle\int_\varSigma \ninnerprod{ \Delta_f \, v_1,\Delta_f \, v_2} \, \vol_f
	\label{eq:RieszI_Imm}
	\\
	\RJ[E]\at_f \colon \RX[E][f] &\to \RY[E][f][\dual],
	&
	\ninnerprod{ \RJ[E]\at_f \; u , \; w}
	&\ceq \textstyle \int_\varSigma \ninnerprod{ \Delta_f \, u,\Delta_f \, w} \, \vol_f.
	\label{eq:RieszJ_Imm}
\end{align}
Then $(\cC,E)$ is a smooth \UnnamedManifold.
\end{proposition}
\begin{proof}
First observe that $\cC$ is a Banach manifold as the boundary trace mapping has a contiuous right inverse.
Fix $f \in \cC$. Observe that $\RT[\ConfSpace][f] = \RX[E][f]$ and $\Rj[E]\dual \, \RJ[E] = \RI[E] \, \Ri[E]$.
Moreover, note that the induced Riemannian metric $g \ceq f^\pull \AmbMetric$ is of class $\sobo{1}{p}$ and that $\Delta_f = \Delta_{g}$.
Thus, \autoref{lem:Laplaceinvertible} (for $k=0$) shows that $\Delta_f$ induces isomorphisms $\RX[\Delta][f] \colon \RX[E][f] \to L^p(\varSigma;\AmbSpace)$, $\RH[\Delta][f] \colon \RH[E][f] \to L^2(\varSigma;\AmbSpace)$, and $\RY[\Delta][f] \colon  \RY[E][f] \to L^q(\varSigma;\AmbSpace)$.
Similarly as in \autoref{ex:LpspacesasUnnamedSpace}, the triple $E_2\at_f \ceq (L^p(\varSigma;\AmbSpace),L^2(\varSigma;\AmbSpace),L^q(\varSigma;\AmbSpace))$ together with canonical inclusions and with the pairings 
\begin{align*}
	\ninnerprod{\RI[E_2]\at_f \, v_1 , v_2} \ceq \int_\varSigma \ninnerprod{v_1,v_2}_\AmbSpace \, \vol_f
	\qand
	\ninnerprod{\RJ[E_2]\at_f \, u , w} \ceq \int_\varSigma \ninnerprod{u,w}_\AmbSpace \, \vol_f	
\end{align*}
forms a \UnnamedSpace. Thus, we may deduce that
$
 	\RI[E]\at_f = \RH[\Delta][f][\dual] \; \RI[E_2]\at_f \; \RH[\Delta][f]
$ and $
	\RJ[E]\at_f = \RY[\Delta][f][\dual] \; \RJ[E_2]\at_f \; \RY[\Delta][f] 	
$
are isomorphisms of Banach spaces. 
Finally, we note that all these structures depend smoothly on $\dd f$, $\dd f^\dagger$, $f^\pull \AmbMetric$, and $\Hess[f]$ which themselves depend smoothly on $f$ (see the proof of \autoref{prop:Hessian})
\end{proof}

\begin{proposition}\label{prop:W2pisparaRiemannian2}
Let $\varSigma$ be a connected, $n$-dimensional, compact manifold without boundary, let $p \in \nintervaloo{n,\infty}\cap\nintervalco{2,\infty}$, and let 
$q$ be the H\"{o}lder conjugate of $p$. Consider the \emph{barycenter mapping} 
\begin{align*}
		\varPsi \colon \Imm^{2,p}(\varSigma;\AmbSpace) \to \AmbSpace, 
	\qquad
	\varPsi(f) \ceq \textstyle \Big( \int_\varSigma \vol_f \Big)^{-1}\Big(\int_\varSigma f\, \vol_f\Big) .
\end{align*}
Fix a $y_0 \in \AmbSpace$ and put $\ConfSpace \ceq \varPsi^{-1}(y_0)$.
For each $f \in \Imm^{2,p}(\varSigma;\AmbSpace)$, the tangent map $\RT[A][f]\ceq\RT[\varPsi][f]$ induces a continuous linear chain map 
\begin{equation*}
\begin{tikzcd}[]
	\sobo{2}{p}[\varSigma][\AmbSpace]
		\ar[d,hook, two heads]
		\ar[rr, "{\RX[A][f]}"]
	&&\AmbSpace
		\ar[d,hook, two heads]	
	\\
	\sobo{2}{2}[\varSigma][\AmbSpace]
		\ar[d,hook, two heads]	
		\ar[rr, "{\RH[A][f]}"]	
	&&\AmbSpace
		\ar[d,hook, two heads]		
	\\
	\sobo{2}{q}[\varSigma][\AmbSpace]
		\ar[rr, "{\RY[A][f]}"]	
	&&\AmbSpace\nospaceperiod
\end{tikzcd}
\end{equation*}
Then $E \ceq \ker(A)\at_\ConfSpace$ together with the Riesz isomorphisms as in \eqref{eq:RieszI_Imm} and \eqref{eq:RieszJ_Imm} forms a \UnnamedBundle and $(\ConfSpace,E)$
is a smooth \UnnamedManifold.
\end{proposition}
\begin{proof}
For $u \in \sobo{2}{p}[\varSigma][\AmbSpace]$, the differential of $\varPsi$ in direction $u$ is given by
\begin{align*}
	\RT[\varPsi][f] \, u = \Big( \int_\varSigma \vol_f \Big)^{-1} \Big(\int_\varSigma u \, \vol_f \Big)
	- \Big( \int_\varSigma \vol_f \Big)^{-2} \Big(\int_\varSigma f \, \ninnerprod{\dd f , \dd u}_f \, \vol_f \Big).
\end{align*}
With the help of \autoref{lem:regularityofproducts}, one can show that $\RT[\varPsi][f]$ can be continuously extended to $\sobo{2}{2}[\varSigma][\AmbSpace]$ and $\sobo{2}{q}[\varSigma][\AmbSpace]$. Moreover, a continuous right-inverse  $B$ of $A$ is readily constructed by
$
	\RZ[B][f]\, V = \paren{x \mapsto V}
$
for each $V \in \AmbSpace$ and each $\RZ \in \set{\RX,\RH,\RY}$. Thus, $\cC$ is a smooth Banach manifold, $\ker(A)$ is a smooth vector bundle, and so is $E = \ker(A)\at_\ConfSpace$. That \eqref{eq:RieszI_Imm} and \eqref{eq:RieszJ_Imm} are isomorphism follows from \autoref{theo:Fredholm}, but we leave the details to the reader.
\end{proof}

\begin{figure}
\capstart
\newcommand{\settrimming}{%
    \setkeys{Gin}{%
        trim = 0 0 0 0 , clip=true, 
        width=0.32\textwidth
    }
    \presetkeys{Gin}{clip}{}
}
\settrimming
\begin{center}
\includegraphics{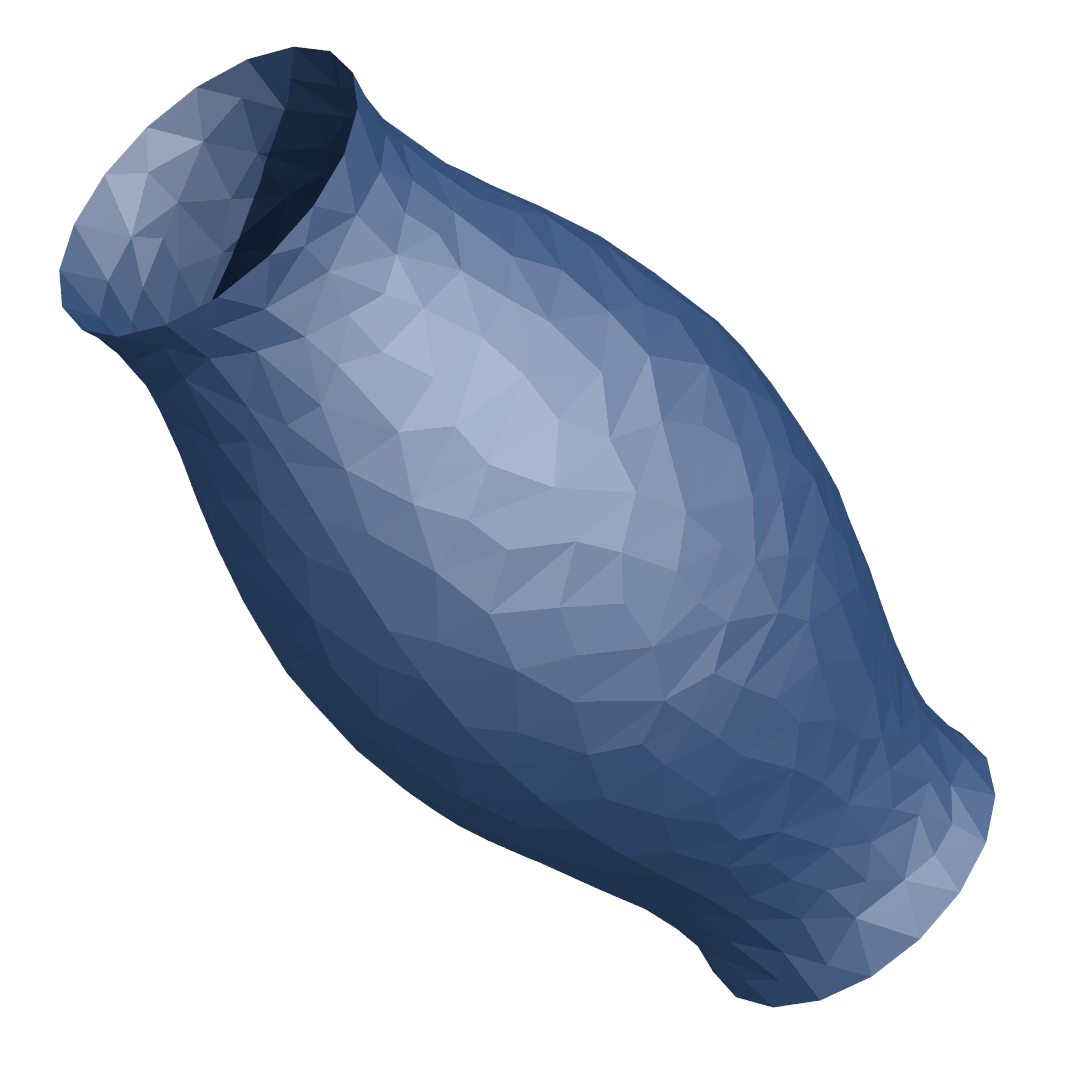}
\includegraphics{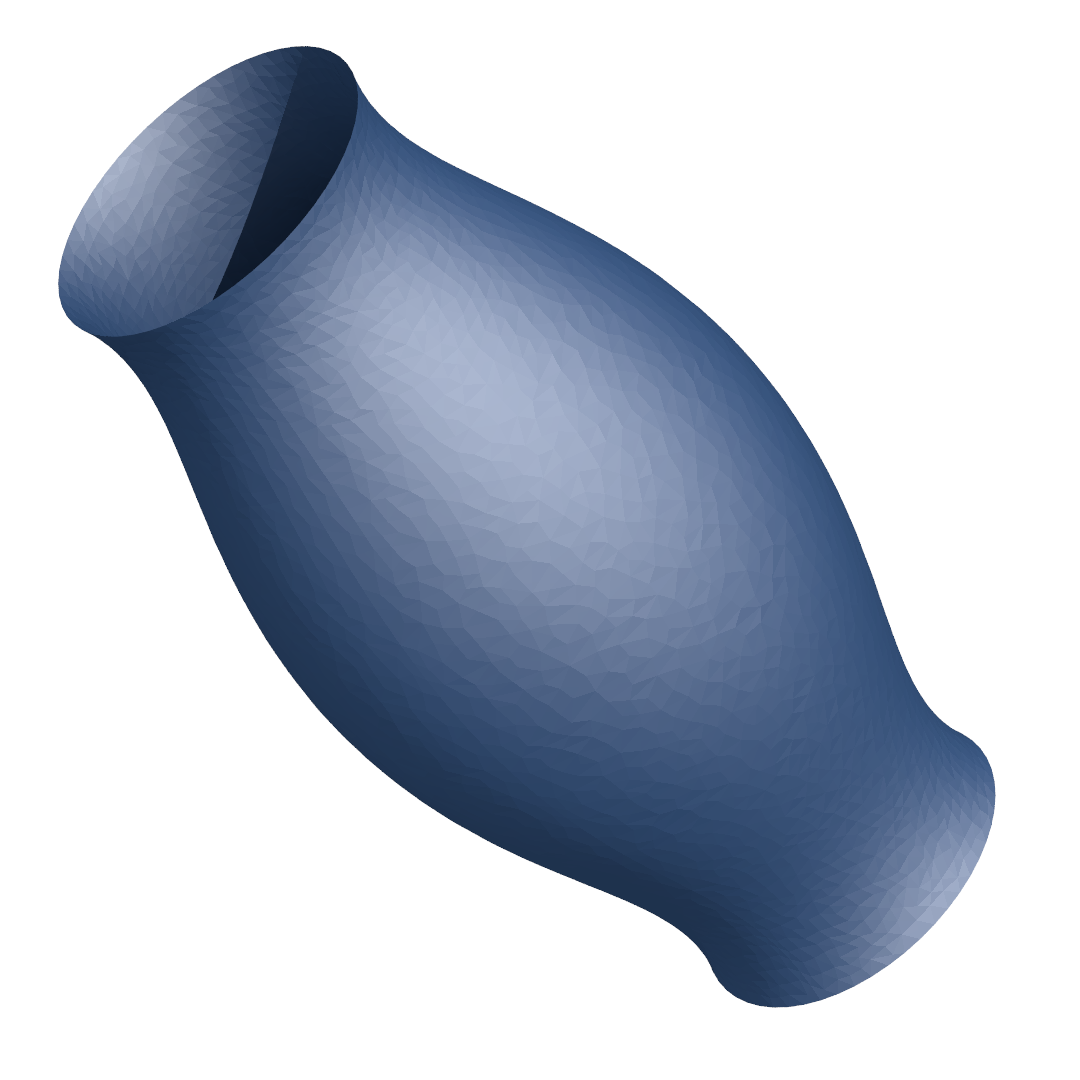}
\includegraphics{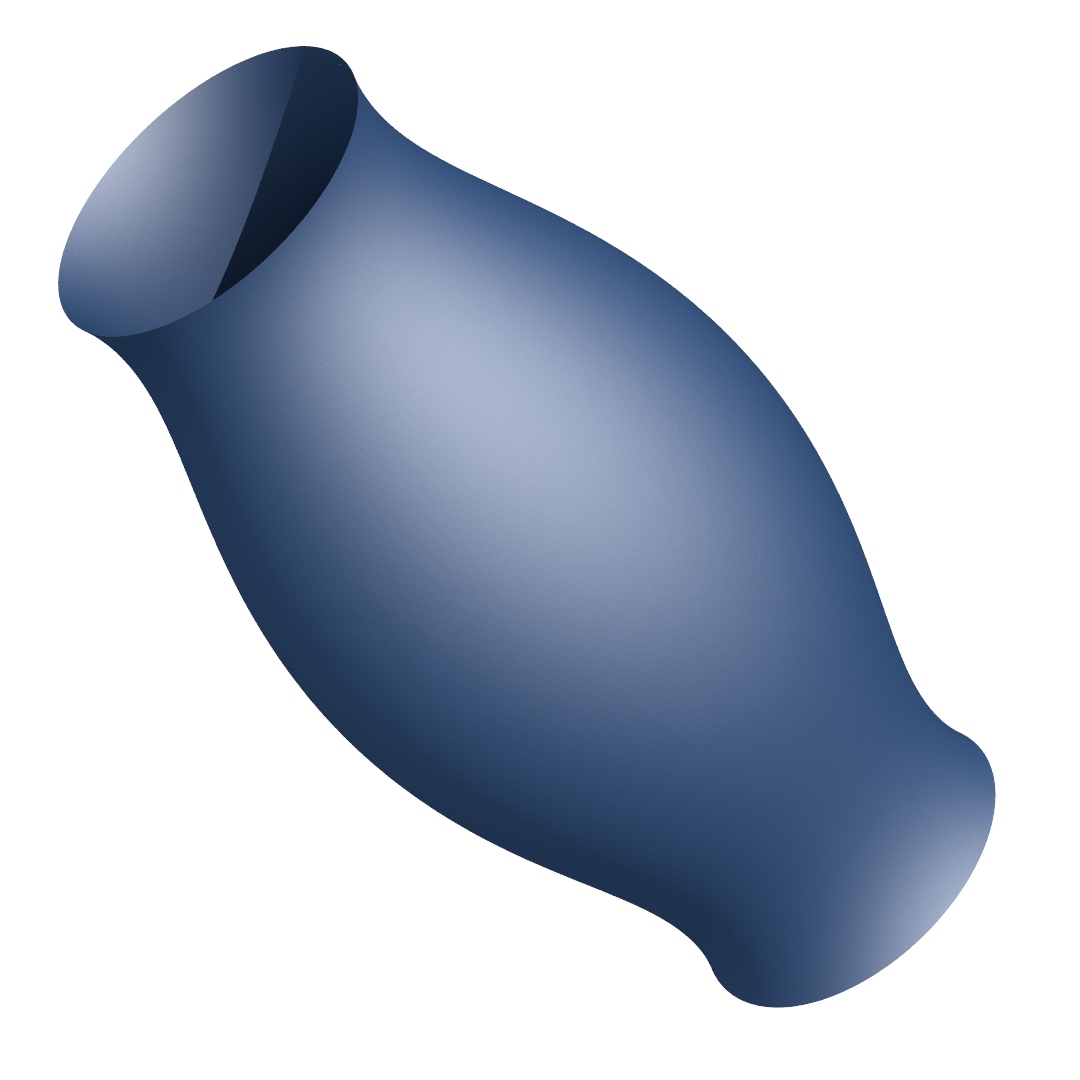}
\caption{The same as in \autoref{fig:Cylinders1} but with Dirichlet boundary conditions that prohibit the existence of minimal surfaces.
Observe the counterintuitive outward bending.
}
\label{fig:Cylinders2}
\end{center}
\end{figure}

\begin{theorem}\label{theo:elasticafunctionalisRiesz}
Let $\ConfSpace$ be one of the  \UnnamedManifolds defined in \autoref{prop:W2pisparaRiemannian} and \autoref{prop:W2pisparaRiemannian2}.
Then the elastica functional $\ElasticEnergy$ and the Willmore energy $\Willmore$ are smooth morphisms on $\ConfSpace$.
\end{theorem}
\begin{proof}
\invisible{\autoref{dfn:elasticafunctional},\autoref{prop:Hessian}}
Fix $f \in \Imm^{2,p}(\varSigma;\AmbSpace)$ and let $u \in \RX[\ConfSpace][f] \subset \sobo{2}{p}[\varSigma][\AmbSpace]$. 
With Equation \eqref{eq:derivativeofMoorePenrose}, the abbreviation $p \ceq \dd f \, \dd f^\dagger$ for the tangent projector, and the product rule, we obtain for any two smooth vector fields $X$ and $Y$ on $\varSigma$:
\begin{align*}
	\MoveEqLeft
	D\bigparen{f\mapsto \Hess[f](f)(X,Y)} \, u 
	\\
	&= (1-p) \cdot \bigparen{\dd(\dd u \, \dd f^\dagger)\,X} \cdot (\dd f \,Y)
	+ (\dd p \,X) \cdot \dd u  \cdot (\id_{T\varSigma} - \dd f^\dagger \, \dd f)\,Y
	\\
	&\qquad	
	+ \bigparen{\dd (p \, (\dd u \, \dd f^\dagger)^*)\,X} \cdot (\id_\AmbSpace - p)	) \cdot (\dd f \,Y)
	- p \cdot (\dd u \, \dd f^\dagger)^* \cdot \bigparen{\dd p \,X} \cdot (\dd f \,Y)	.
\end{align*}
By definition, we have $\bigparen{\dd(\dd u \, \dd f^\dagger)\,X} \cdot (\dd f \,Y) = \Hess[f](u)(X,Y)$ and $(\dd p \, X) \, (\dd f \, Y) = \II(f)(X,Y)$.
Moreover, we utilizing the identities $(\id_\AmbSpace - p) \cdot \dd f=0$,
$\id_{T\varSigma} - \dd f^\dagger \, \dd f=0$,
and
$p \cdot (\dd u \, \dd f^\dagger)^* = (\dd u \, \dd f^\dagger)^*$, we obtain
%
\begin{align*}
	D\II (f) \, u
	&= \bigparen{\id_{\AmbSpace} - \dd f \, \dd f^\dagger} \Hess[f](u) - (\dd u \, \dd f^\dagger)^* \II(f).
\end{align*}
The derivative of $f \mapsto \vol_{f^\pull \AmbMetric}$ in direction $u$ is given by 
$u= \ninnerprod{\dd f, \dd u}_{f^\pull \AmbMetric} \, \vol_{f^\pull \AmbMetric}$ and we have 
$\nabs{S}_{f^\pull \AmbMetric}^2
	= \nabs{S(\dd f^\dagger \,\cdot, \dd f^\dagger \, \cdot)}_{ \AmbMetric}^2$
for $S \in L^r(\varSigma;\Sym(T\varSigma;\AmbSpace))$.
This would allow us to compute a precise expression for $\ninnerprod{\dd \ElasticEnergy\at_f,u}$, but it already suffices for our considerations to observe that $\dd \ElasticEnergy$ is of the form
\begin{align*}
	\ninnerprod{\dd \ElasticEnergy\at_f,u}
	= 2 \, \int_\varSigma 
	\paren{
		\ninnerprod{\II(f), \Hess[f](u)}_{f^\pull \AmbMetric}
		+		
		\mu(\II(f),\dd f, \dd f^\dagger, \dd u)
	}\, \vol_{f},
\end{align*}
where $\mu(\II(f),\dd f, \dd f^\dagger, \dd u)$ is a polynomial expression in $\II(f)$, $\dd f$, $\dd f^\dagger$, and  $\dd u$ with constant coefficients. Moreover, $\II(f)$ occurs with order two and $\dd u$ occurs with order one in this expression.

Now let $w \in \RY[\ConfSpace][f] = \sobo{2}{q}[\varSigma][\AmbSpace]$.
Note that we have $\II(f) \in L^p$, $\dd f \in L^\infty$, and $\dd f^\dagger \in L^\infty$, hence $\mu(\II(f),\dd f, \dd f^\dagger, \cdot) \in L^{p/2}$.
In the case $n=1$, we have $\dd w \in \sobo{1}{q} \hookrightarrow L^{\infty}$.
For $n\geq 2$, we have $\dd w \in \sobo{1}{q} \hookrightarrow L^{r}$ with $r \geq \frac{p}{p-1-\frac{p}{n}}$.
Because of $p>n$, we obtain
\begin{align*}
	r \geq \frac{p}{p-1-\frac{p}{n}} > \frac{p}{p-2} = (p/2)',
\end{align*}
thus $\dd w \in L^{(p/2)'}$, where $(p/2)'$ denotes the H\"older conjugate of $p/2$.
In any case, we obtain $\mu(\II(f),\dd f, \dd f^\dagger, \dd w) \in L^1$. This shows that $ \RX[\ElasticEnergy][f] = \dd \ElasticEnergy\at_f$ can be continuously extended to $\RY[\ElasticEnergy][f] \colon \RY[\ConfSpace][f]\to \R$ so that $\Rj[\R] \, \Ri[\R] \, \RX[\ElasticEnergy][f] = \RY[\ElasticEnergy][f]\,\Rj[\ConfSpace]\,\Ri[\ConfSpace]$ holds. Hence, $\ElasticEnergy$ is a morphism of \UnnamedManifolds. The statement for $\Willmore$ follows from the identity $H(f) = \dim(\varSigma)^{-1}\sum_{i=0}^m \II(f)( \dd f^\dagger e_i ,\dd f^\dagger e_i)$ for any $\AmbMetric$-orthonormal basis $e_1,\dotsc,e_m$ of $\AmbSpace$ and from the above discussion.
\end{proof}
%

\section{Numerical Examples}\label{sec:NumericalExamples}

Of course, we are not the first to minimize the Willmore energy numerically. Usually, semi-implicit discretizations of Willmore-flow (the downward $L^2$-gradient flow) are applied; see. e.g.
\cite{hsu1992}, \cite{Bobenko05discretewillmore}, or \cite{MR2448203} for the Willmore flow applied to triangle meshes and \cite{MR2095338} for level set formulations.
A notable exception is \cite{Crane:2013:RFC}, where 
beautiful ideas from the conformal geometry of surfaces are applied to formulate an $L^2$-gradient-like descent in \emph{mean curvature space}.

All these approaches have in common that immersed surfaces are discretized by finite, immersed, polyhedral surfaces. One commits a severe variational crime this way: Since the Willmore energy is not defined for simplicial surfaces, one has to \emph{design} a discrete Willmore energy which \emph{hopefully} will produce minimizers that approach the minimizers of the (smooth) Willmore energy in a suitable topology as the mesh is suitably refined. 
Even so-called mixed formulations cannot cure this flaw, at least in the form they are usually applied to bi-harmonic problems.\footnote{
Instead of minimizing $\int_\varSigma \nabs{\dim(\varSigma)^{-1}\,\Delta_g f}^2 \, \vol_g$ among $f \in \sobo{2}{p}[\varSigma][\R^3]$, mixed formulations introduce an auxilliary variable $H \in \sobo{1}{p}[\varSigma][\R^3]$ and
minimize $\int_\varSigma \nabs{H}^2 \, \vol_g$ subject to the constraint $\dim(\varSigma) \, H - \Delta_g f =0$ in a suitable weak form. 
While this works out well for elliptic problems involving a \emph{fixed} Riemannian metric $g$ of class $\sobo{1}{p}$, $\dim(\varSigma)< p < \infty$ (see \cite{MR0483557} for an analysis of a domain $\varSigma \subset \R^2$ and with the Euclidean metric $g$), it breaks down when $g$ is induced by $f\in \sobo{1}{p}[\varSigma][\R^3] \setminus \sobo{1}{\infty}[\varSigma][\R^3]$ for the very same reasons as we outlined in the introduction.}

\begin{figure}
\capstart
\newcommand{\settrimming}{%
    \setkeys{Gin}{%
        trim = 10 80 10 100 , clip=true, 
        width=0.20\textwidth
    }
    \presetkeys{Gin}{clip}{}
}
\settrimming
\begin{center}
\begin{equation*}
\begin{tikzcd}[]
	\vcenter{\hbox{\includegraphics{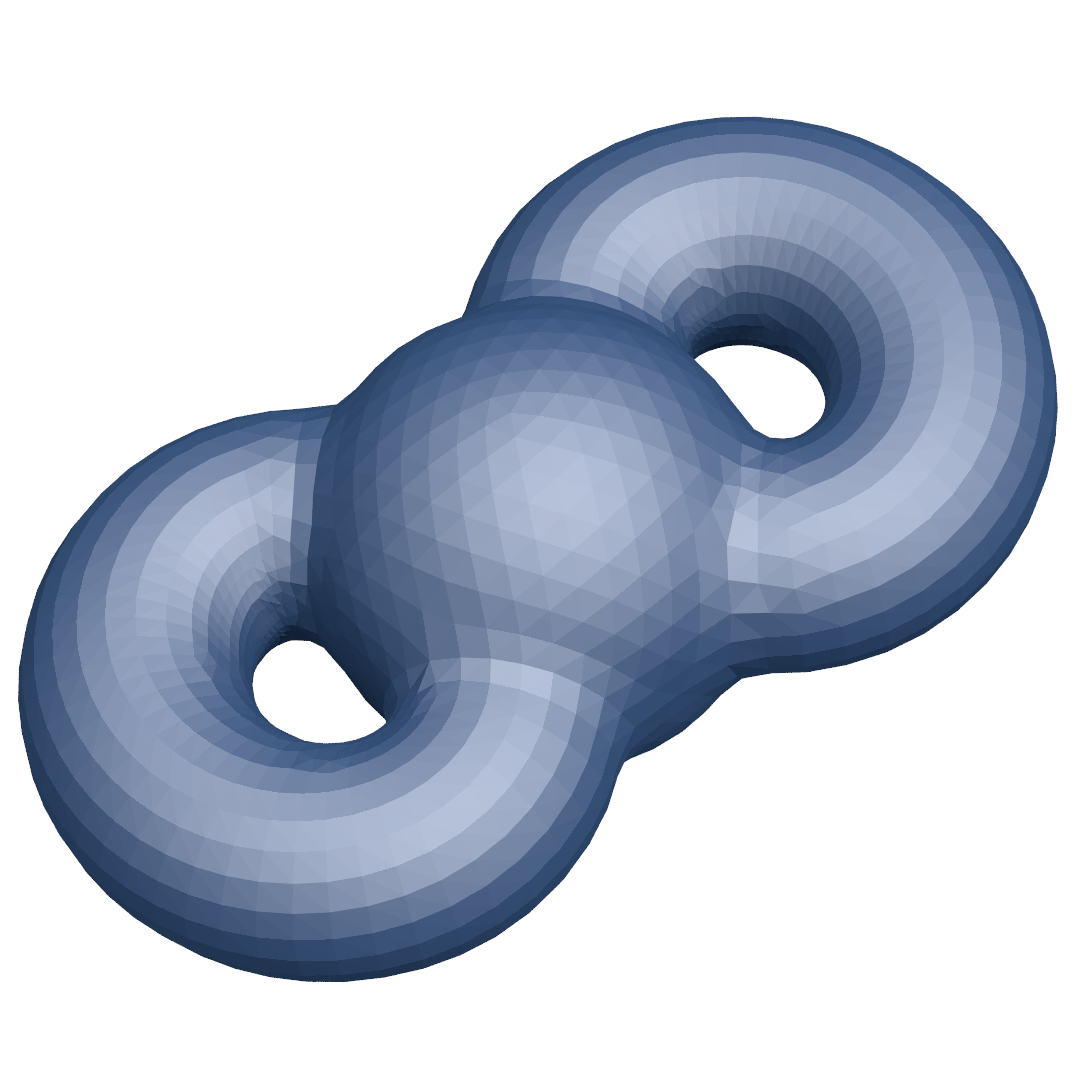}}}
		\ar[d, mapsto, "\text{Loop subdivision}"'] 
		\ar[r, mapsto, "\text{10 steps}"]
	& \vcenter{\hbox{\includegraphics{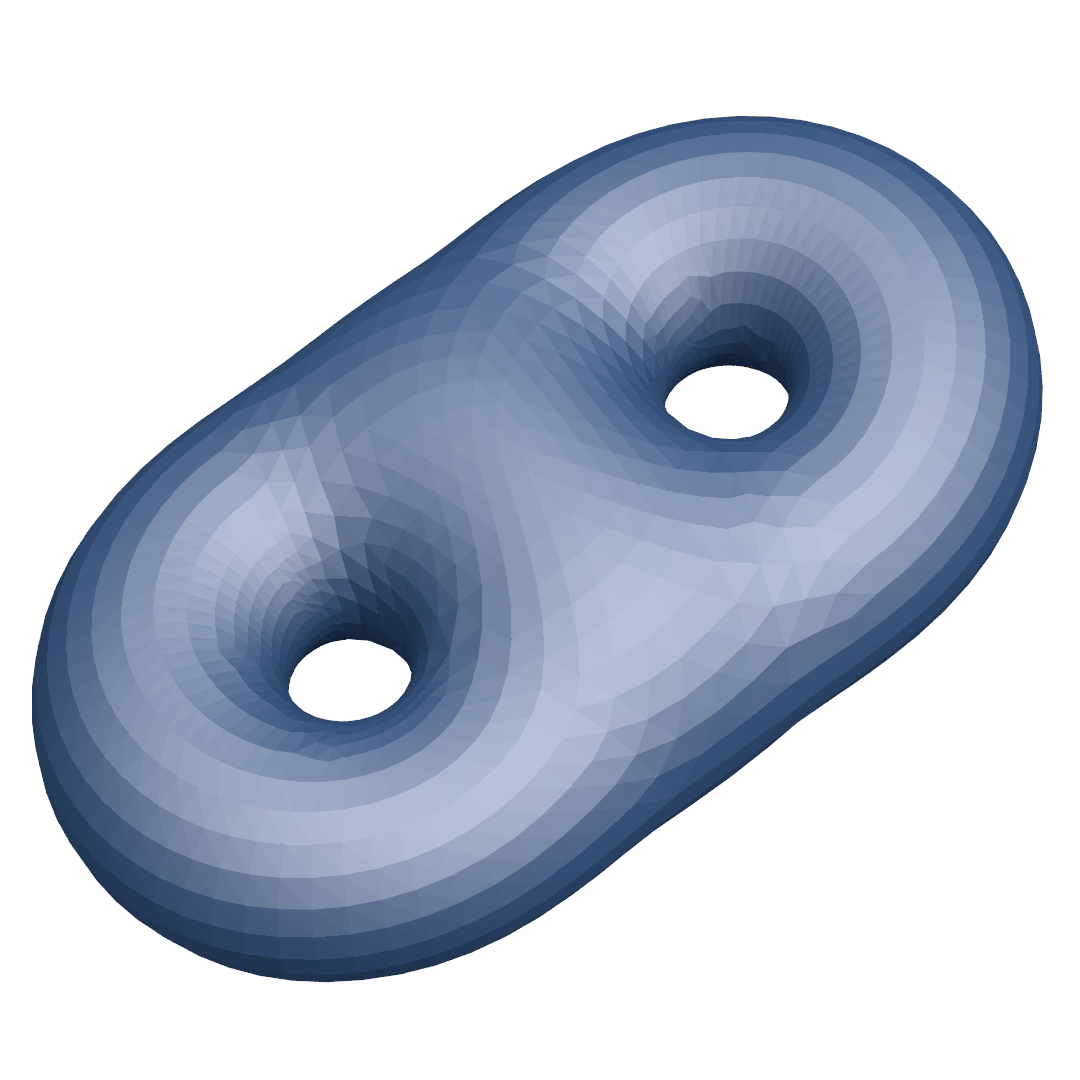}}}
		\ar[r, mapsto, "\text{50 steps}"]
	& \vcenter{\hbox{\includegraphics{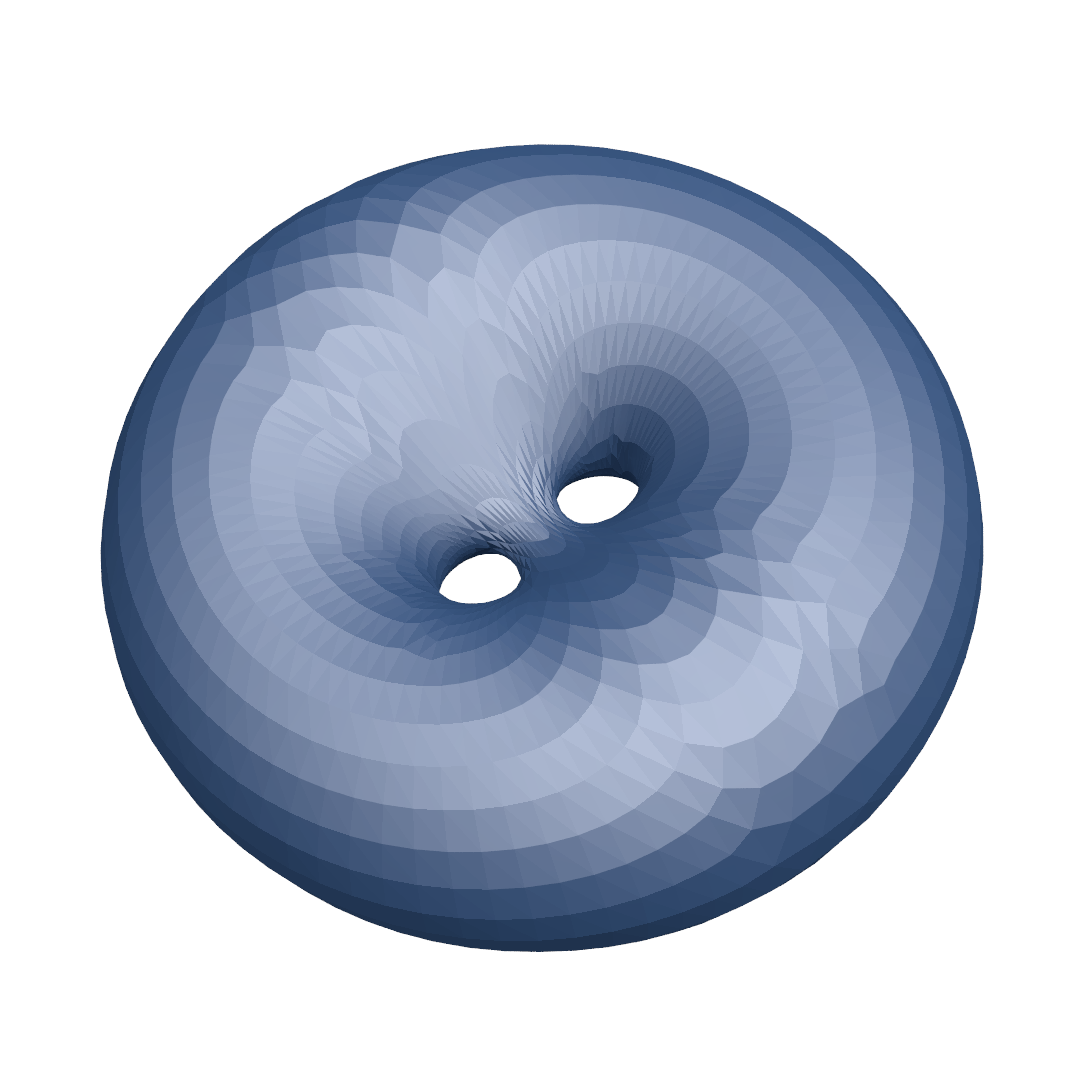}}}
		\\
	\vcenter{\hbox{\includegraphics{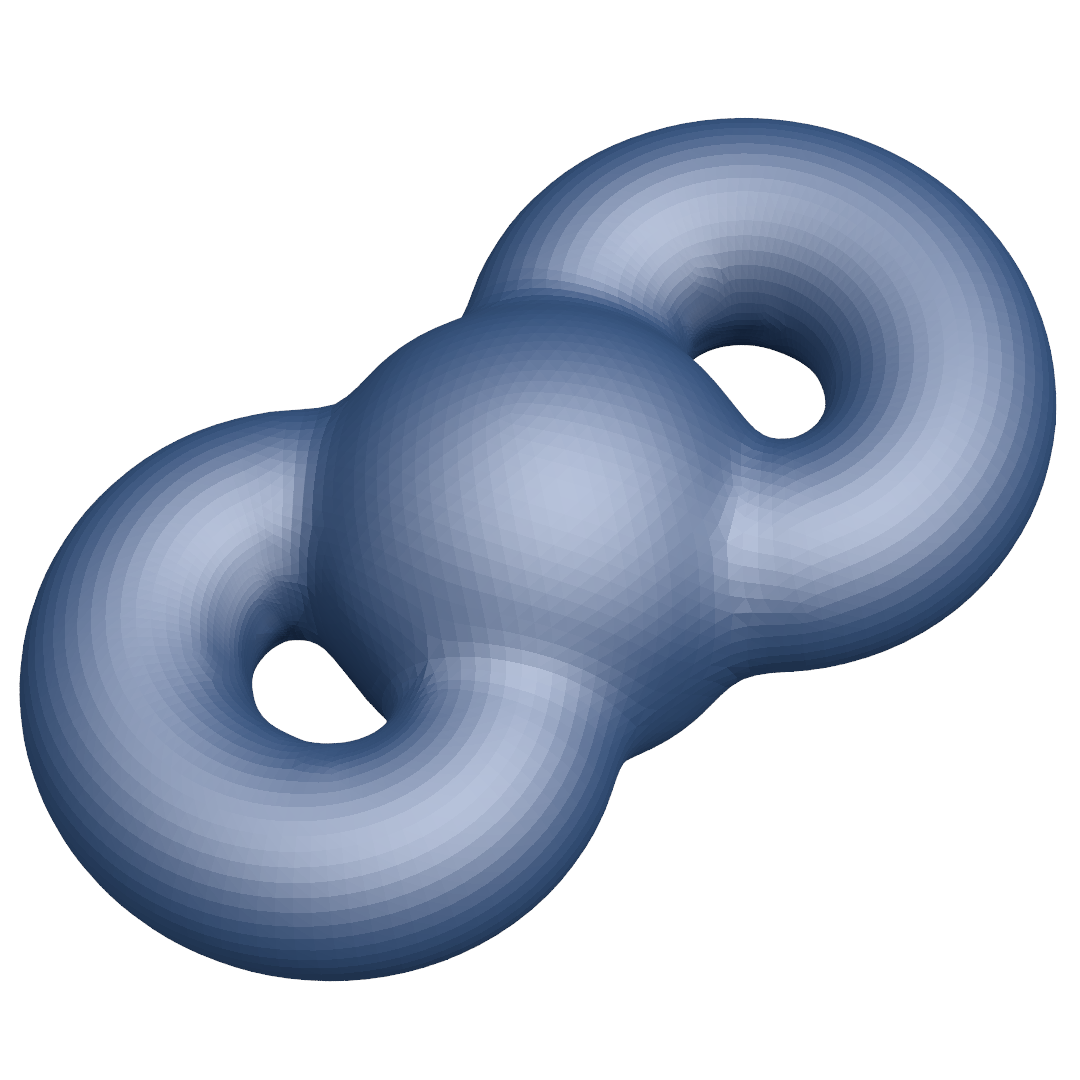}}}
		\ar[d, mapsto, "\text{Loop subdivision}"'] 
		\ar[r, mapsto, "\text{10 steps}"]
	& \vcenter{\hbox{\includegraphics{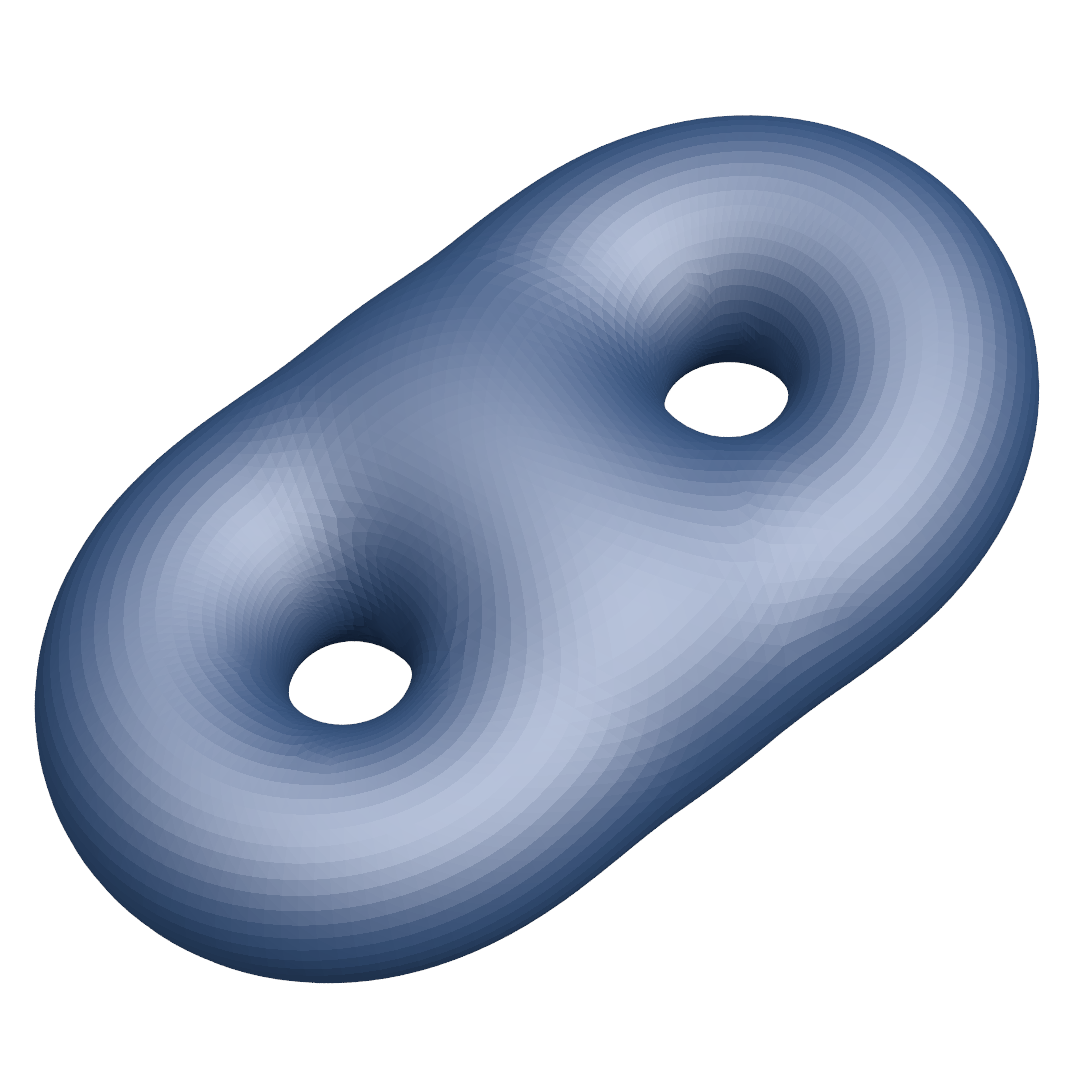}}}
		\ar[r, mapsto, "\text{50 steps}"]
	& \vcenter{\hbox{\includegraphics{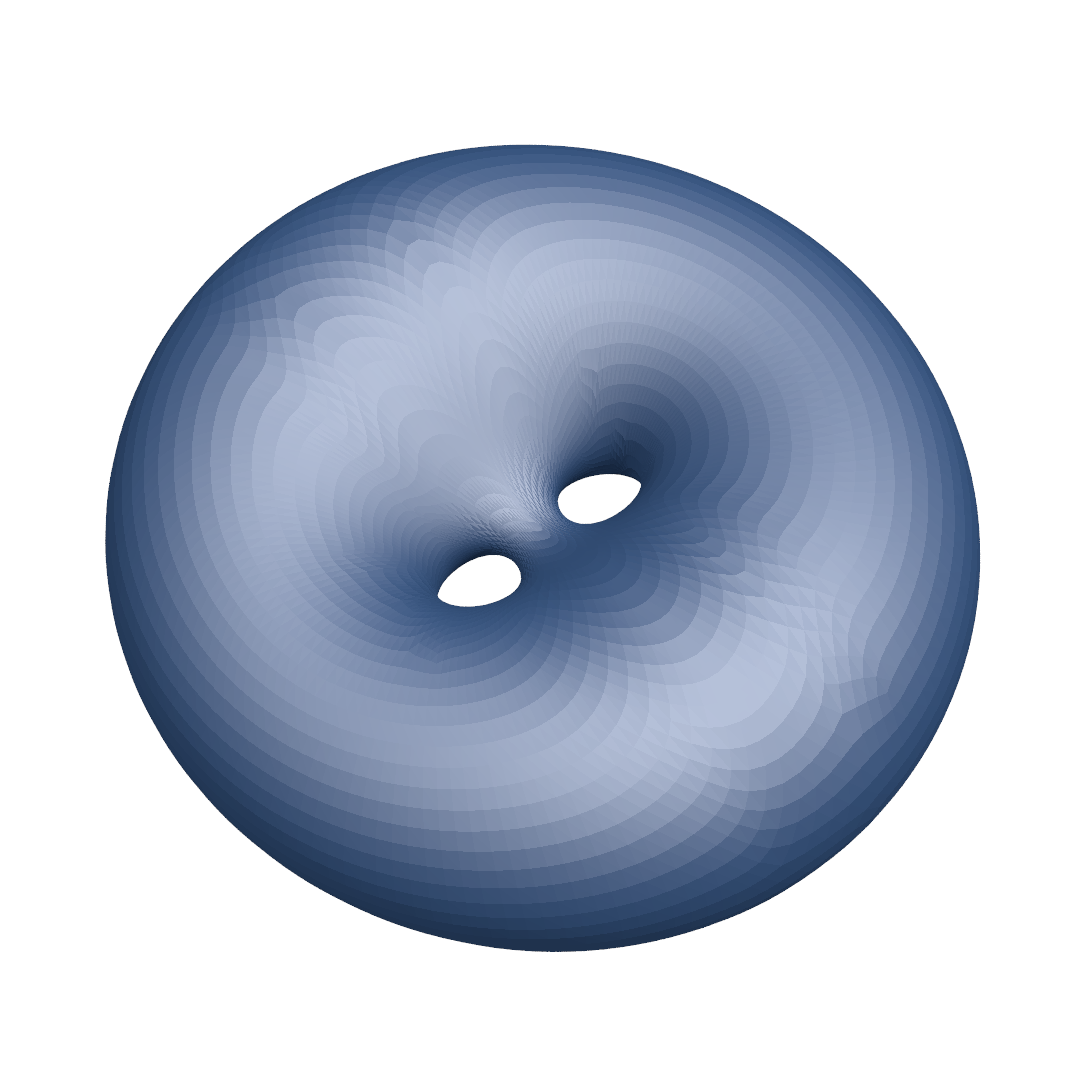}}}
	\\
	\vcenter{\hbox{\includegraphics{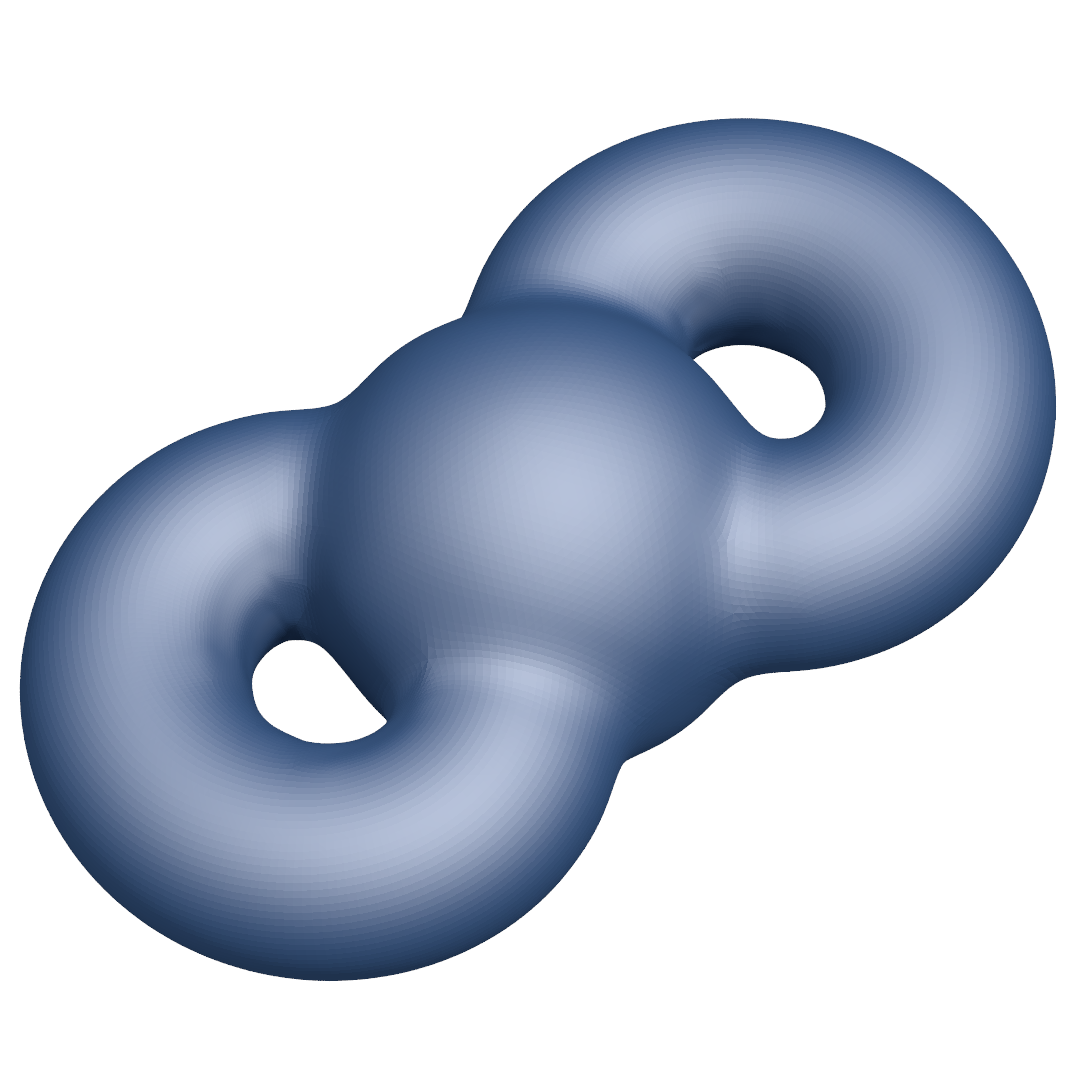}}}
		\ar[r, mapsto, "\text{10 steps}"]
	& \vcenter{\hbox{\includegraphics{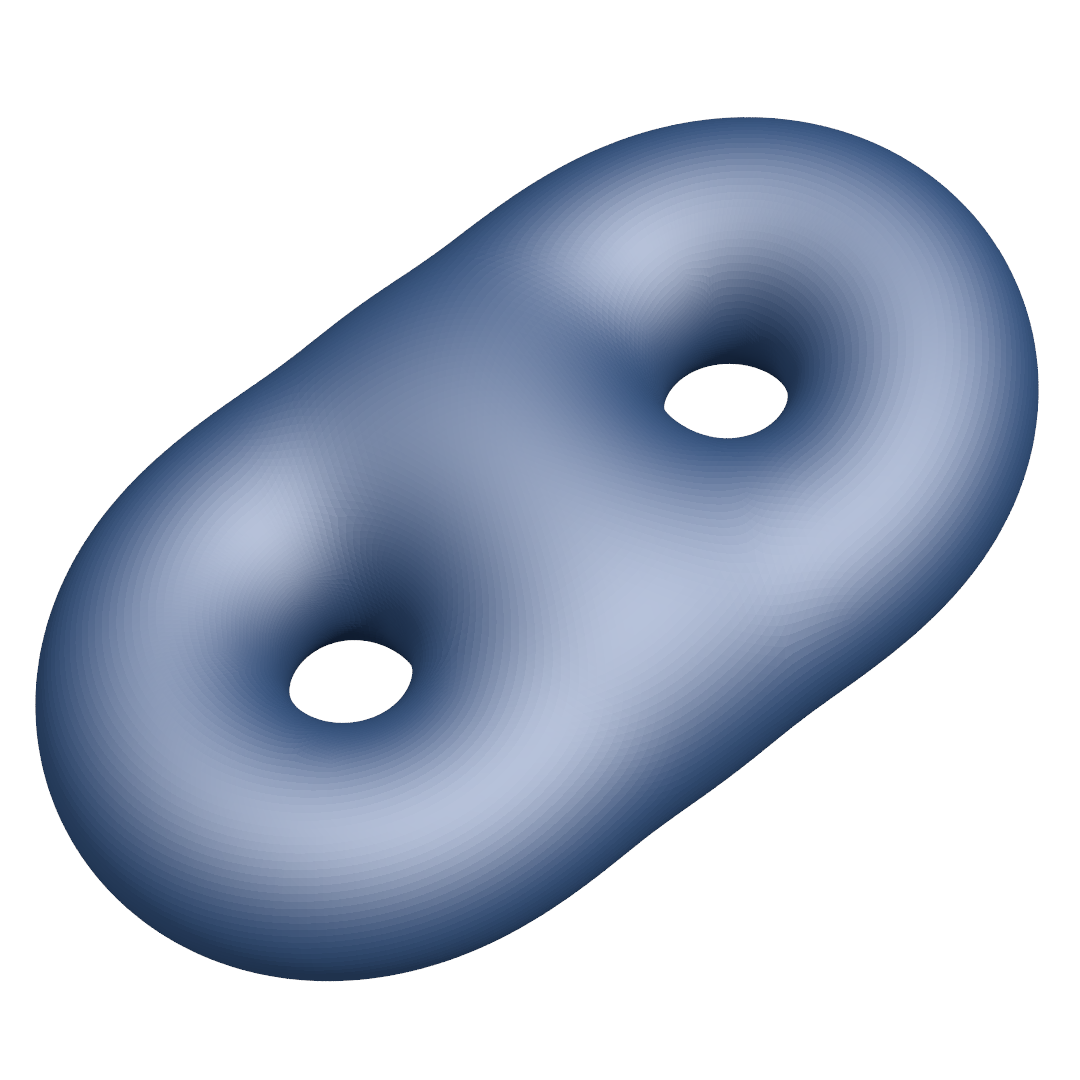}}}
		\ar[r, mapsto, "\text{50 steps}"]
	& \vcenter{\hbox{\includegraphics{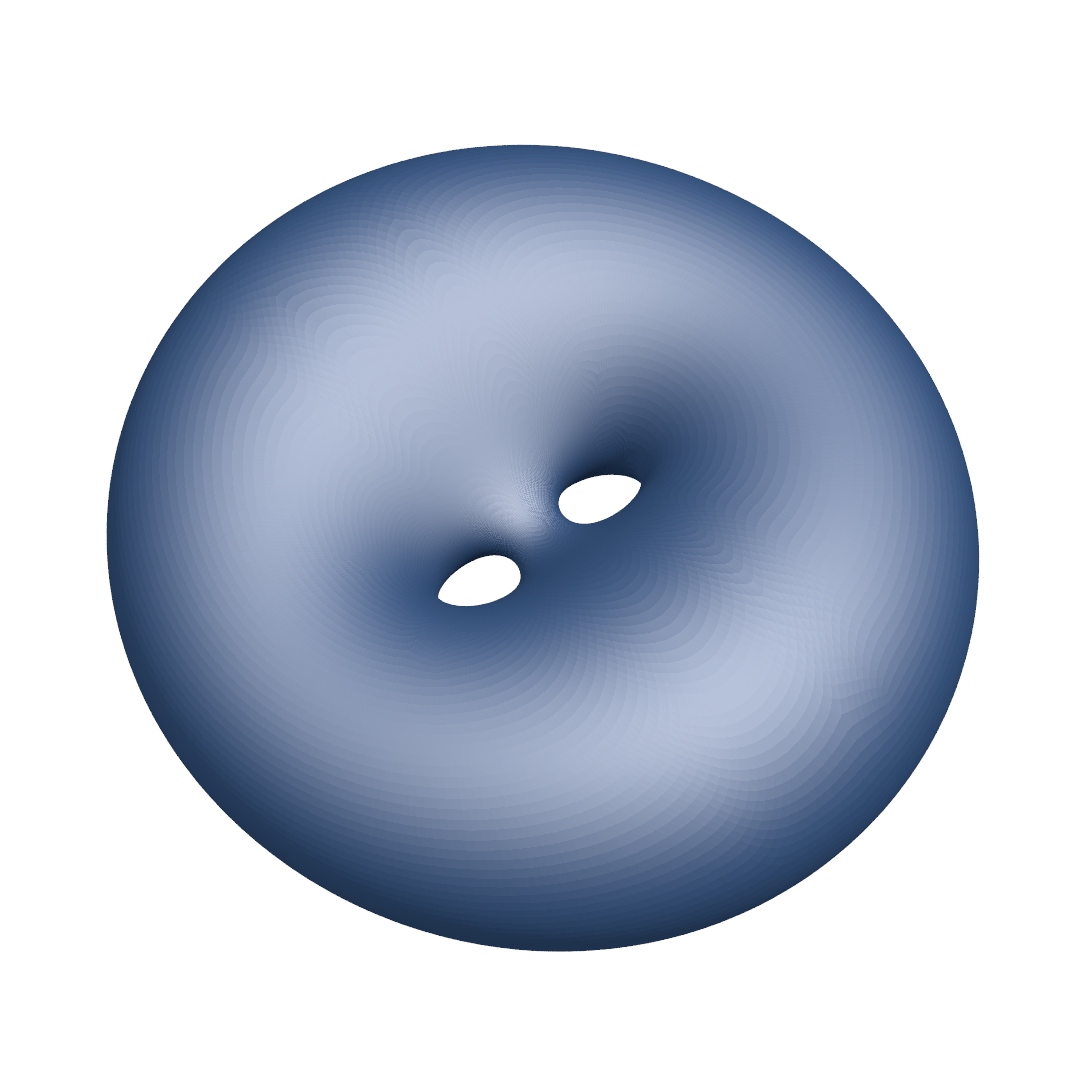}	}}
\end{tikzcd}
\end{equation*}
\caption{Rows show some iterations of the discrete $H^2$-gradient descent of Willmore energy for different mesh resolutions. Surfaces in first column originate from the top left by Loop subdivision. Note that gradient descent and subdivision almost commute.
}
\label{fig:HandleBody2}
\end{center}
\end{figure}
\begin{figure}
\capstart
\newcommand{\settrimming}{%
    \setkeys{Gin}{%
        trim = 0 0 0 0 , clip=true, 
        width=0.45\textwidth
    }
    \presetkeys{Gin}{clip}{}
}
\settrimming
\begin{center}
\includegraphics{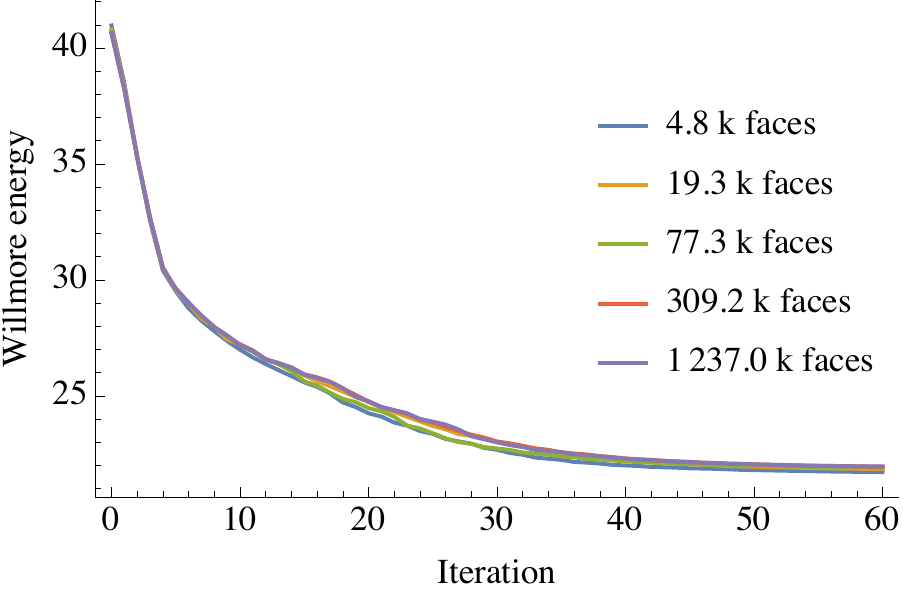}
\includegraphics{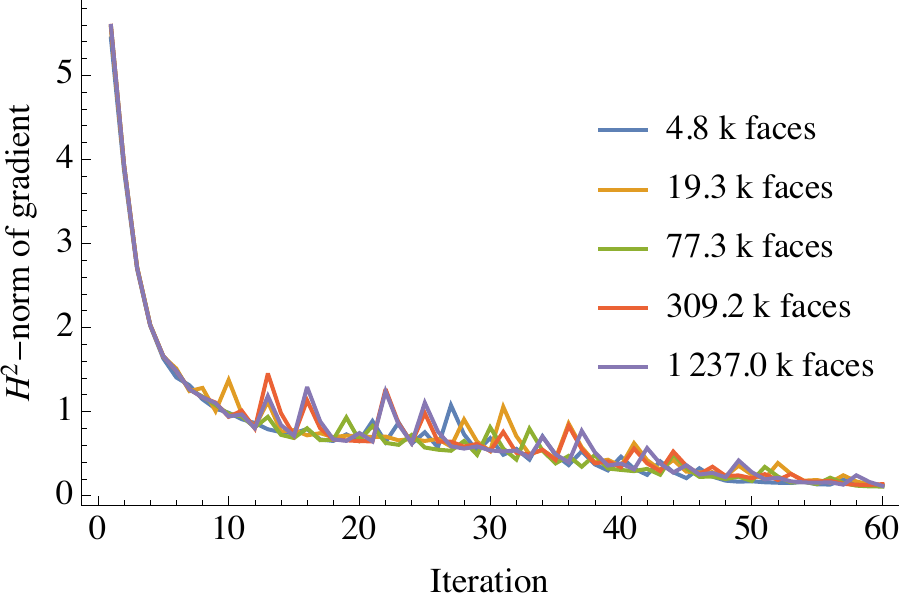}
\caption{The optimization histories of \autoref{fig:HandleBody2} and two further subdivisions show similar behavior for all mesh resolutions.
}	
\label{fig:Plots}
\end{center}
\end{figure}

Although the issue of convergence of these nonconforming Ritz-Galerkin methods has not been settled completely, the cited methods lead to very plausible results and some partial information on consistency is already  known (see \cite{MR2448203}).
If one is willing to restrict one's attention to surfaces that can be parameterized over compact domains in $\R^2$, $C^1$-finite elements such as Argyris or Bell elements for scalar functions can be adapted (as was done, e.g., in \cite{MR1825128} for graphs).
Constructing implementable and efficient finite elements with $C^1$ or $G^1$ continuity for arbitrary surfaces is harder than one would expect and is still a matter of contemporary research.

Here, we only opt for illustrating the effect of using $H^2$-gradient descent, 
and this can be done with each of the described methods.
For simplicity, we follow the nonconforming Ritz-Galerkin approach and use the following discretization: Let $\varSigma$ be a compact, two-dimensional manifold and fix a smooth triangulation $\cT$ with vertices $p_1,\dots,p_N \in \varSigma$. Such a triangulation induces barycentric coordinates on each simplex of the triangulation and thus a vector space $\on{PL}_\cT(\varSigma;\R^3) \subset \sobo{1}{\infty}[\varSigma][\R^3]$ of continuous, $\R^3$-valued, piecewise-linear functions. 
Define the discrete configuration space $\ConfSpace_\cT \ceq \set{f \in \on{PL}_\cT(\varSigma;\R^3) | \text{$f$ is a Lipschitz-immersion}}$ and observe
$\RT[\ConfSpace_\cT][f] = \on{PL}_\cT(\varSigma;\R^3)$.
Fix a basis $\varphi_1, \dotsc, \varphi_{3N}$ of $\on{PL}_\cT(\varSigma;\R^3)$ by 
\begin{align*}
	\varphi_{3(i-1)+k}(p_j) = \updelta_{ij} \, e_k
	\quad
	\text{for $i$, $j\in \set{1,\dotsc,N}$ and $k \in \{1,2,3\}$,}
\end{align*}
where $e_1$, $e_2$, $e_3$ is the standard basis in $\R^3$.
Each element $u \in \on{PL}_\cT(\varSigma;\R^3)$ is uniquely represented by the vector $\textbf{u} \in \R^{3 N}$ with entries $\mathbf{u}_{3(i-1)+k} = \ninnerprod{u(p_i),e_k}_{\R^3}$.
\begin{align*}
	\mathbf{M}_{ij}(f) \ceq 
	\int_\varSigma \ninnerprod{\varphi_i,\varphi_j} \, \vol_f
	\qand 
	\mathbf{L}_{ij}(f) \ceq 
	\int_\varSigma \ninnerprod{\dd \varphi_i,\dd \varphi_j}_f \, \vol_f
\end{align*}
In order to avoid inversion of the mass matrix $\mathbf{M}(f)$, we consider the ``inverse'' of the lumped mass matrix
\begin{align*}
	(\mathbf{\Lambda}_{\on{lump}})_{ij}(f) \ceq 
	\begin{cases}
		\updelta_{ij} \, \big(\int_\varSigma \nabs{\varphi_i}_{\R^3} \, \vol_f\big)^{-1}
			& \text{$i$ is an interior degree of freedom}
		\\
		0, & \text{else}.
	\end{cases}
\end{align*}
The entity $\frac{1}{2}\mathbf{\Lambda}_{\on{lump}}(f) \, \mathbf{L}(f) \, \mathbf{f}(f)$ is of often called \emph{discrete mean curvature vector} in the literature. 
We define the \emph{discrete $H^2$-Riesz isomorphism} and the \emph{discrete Willmore energy} by
\begin{align*}
	\mathbf{J}(f) = 
		\mathbf{L}(f)\transp \, \mathbf{\Lambda}_{\on{lump}}(f) \, \mathbf{L}(f)
	\qand
	\Willmore_\cT(f) \ceq \frac{1}{4} \,
	\mathbf{f}(f)\transp \, \mathbf{J}(f) \, \mathbf{f}(f)
	\quad
	\text{for $f \in \ConfSpace_\cT$.}
\end{align*}
Let $\varPhi_\cT \colon \ConfSpace_\cT \to \R^K$ be a discretized constraint mapping whose differential $D\varPhi_\cT(f)$  is represented by a matrix $\mathbf{A} \in \Hom(\R^{3N};\R^K)$. By \autoref{prop:ProjectedGradient}, we can determine the coordinates $\mathbf{u}$ of the discrete, projected downward gradient $u$, by solving the linear saddle point system
\begin{align}
	\begin{pmatrix}
		\mathbf{J}(f) & \mathbf{A}(f)\transp\\
		\mathbf{A}(f) & \mathbf{0}
	\end{pmatrix}
	\begin{pmatrix}
		\mathbf{u}\\
		\mathbf{\lambda}
	\end{pmatrix}	
	=
	\begin{pmatrix}
		\mathbf{b}\\
		\mathbf{0}
	\end{pmatrix}
	\qquad
	\text{where $\mathbf{b}_{i} \ceq - D \Willmore_\cT(f) \, \varphi_i$.}
	\label{eq:discretesaddlepointsystem}
\end{align}
Differentiation of this equation along the downward gradient trajectory leads to
\begin{align}
	\begin{pmatrix}
		\mathbf{J}(f) & \mathbf{A}(f)\transp\\
		\mathbf{A}(f) & \mathbf{0}
	\end{pmatrix}
	\begin{pmatrix}
		\dot{ \mathbf{u}}\\
		\dot{ \mathbf{\lambda}}
	\end{pmatrix}	
	=
	\begin{pmatrix}
		D\mathbf{b}(f) \, u\\
		\mathbf{0}
	\end{pmatrix}
	-
	\begin{pmatrix}
		D\mathbf{J}(f)\,u & D\mathbf{A}(f)\transp\,u\\
		D\mathbf{A}(f)\,u & \mathbf{0}
	\end{pmatrix}
	\begin{pmatrix}
		\mathbf{u}\\
		\mathbf{\lambda}
	\end{pmatrix}.
	\label{eq:discretesaddlepointsystem2}
\end{align}
We use this to compute a circular search path $\varTheta_f(t)$ (with speed depending affinely on step size) that fits the downward gradient trajectory up to second order.
This has two advantages: First, the increase in constraint violation is at most of third order: $\varPhi_\cT(\varTheta_f(t)) = \varPhi_\cT(f) + \LandO(\norm{t \, u}^3)$. Second, the step size $\tau_0>0$ with $\varTheta_f'(\tau_0) = 0$ is often a very good approximation of the Curry step size and provides an excellent initial guess for a backtracking line search. 
When a suitable step size $\tau$ is found, the constraint violation $\varPhi_\cT(\varTheta_f(\tau))$ can be reduced by iteratively applying
\begin{align*}
	\mathbf{f}_{k+1} = \mathbf{f}_{k} - \mathbf{A}(f)^\dagger \, \varPhi_\cT(f_k),
	\qquad
	\text{with $f_0 = \varTheta_f(\tau_0)$}
\end{align*}
until the size of $\varPhi_\cT(f_k)$ is reduced below a certain tolerance.
By \autoref{rem:saddlepointmatrixisuseful}, we may compute $\mathbf{v} \ceq \mathbf{A}(f)^\dagger \, w$ by solving the linear system
\begin{align}
	\begin{pmatrix}
		\mathbf{J}(f) & \mathbf{A}(f)\transp\\
		\mathbf{A}(f) & \mathbf{0}
	\end{pmatrix}
	\begin{pmatrix}
		\mathbf{v}\\
		\mathbf{\mu}
	\end{pmatrix}	
	=
	\begin{pmatrix}
		\mathbf{0}\\
		\mathbf{w}
	\end{pmatrix}.
	\label{eq:discretesaddlepointsystem3}
\end{align}
Usually, one has  $\tau = \tau_0$ and a single iteration suffices to keep the constraint violation at bay. Hence, the generic gradient flow iterate $f_{\on{new}}$ following $f$ has the coordinates
\begin{align*}
	\mathbf{f}_{\on{new}}
	= \varTheta_f(\tau_0) - \mathbf{A}(f)^\dagger \, \varPhi_\cT(\varTheta_f(\tau_0)).
\end{align*}
In principle, there are many ways to solve sparse saddle point systems such as \eqref{eq:discretesaddlepointsystem}--\eqref{eq:discretesaddlepointsystem3}
(see~\cite{MR2168342} for an overview).
Iterative linear solvers are not able to play out their strengths in our application: They depend heavily on good preconditioners (whose construction may be very problem depending) and good initial guesses (which our gradient method does not provide due to the closeness to the Curry step size). In general, they have been experienced to be not on par with sparse direct solvers for 2D-meshes, especially for bi-Laplacian operators (see \cite{Botsch2005}).
Thus, we decided to use a direct linear 	solver, the sparse $LDL\transp$-factorization provided by PARDISO from the Intel Math Kernel Library 2017.
As the sparsity pattern of the saddle point matrix does not change in the course of computation, its symbolic factorization has to be computed only once in the beginning.
The numerical factorization has to be recomputed in every iteration, but it is used at least three times for the systems \eqref{eq:discretesaddlepointsystem}--\eqref{eq:discretesaddlepointsystem3}.

\begin{table}
\capstart
\begin{center}
	\begin{tabular}{|r|r|r|r|r|}\hline
		No. of faces &\multicolumn{2}{c|}{Initialization} &\multicolumn{2}{c|}{Iteration} \\\cline{2-5}
&time (s) & speed (faces/s) & time (s) & speed (faces/s) \\\hline		\num{22640} & \num{    0.62} & \num{ 36633.3} & \num{    0.44} & \num{ 51762.2} \\\hline
		\num{90560} & \num{    2.10} & \num{ 43211.3} & \num{    1.17} & \num{ 77137.2} \\\hline
		\num{362240} & \num{    8.39} & \num{ 43176.8} & \num{    4.63} & \num{ 78282.9} \\\hline
		\num{1448960} & \num{   39.73} & \num{ 36471.1} & \num{   21.68} & \num{ 66828.6} \\\hline
	\end{tabular}
	\caption{Timings for the triceratops mesh from \autoref{fig:Triceratops} in different resolutions. Initialization timings include generation of sparsity patterns and numerical factorization.
Iteration timings were averaged over the first 30 iterations and consist of computation times for projected gradient, second order fit to flow lines, backtracking line search, and iterative projection onto constraint manifold.}
\label{tab:TriceratopsTimings}
\end{center}	
\end{table}

Derivatives of local vectors and matrices were symbolically computed with Mathematica at compile time and automatically turned into parallelized, runtime efficient binaries with the help of the built-in \texttt{Compile} command.
More hand-tuning was invested in the code for the assembly of first total derivatives and second directional derivatives to avoid costly assembly of higher-degree tensors such as $D\mathbf{L}$, $DD\mathbf{\Lambda}_{\on{lump}}$, and the third derivative of the area functional. Fortunately, the required quantities can be expressed by contractions of vectors with 3-tensors that are assembled from per-triangle tensors.
Thus, these can be implemented by distributing the vectors to the triangles, applying local tensor-vector operations, and assembling the resulting matrix afterwards.

In order to give the reader an impression of the performance of the projected $H^2$-gradient descent, we collected some timings for the example depicted in 	\autoref{tab:TriceratopsTimings}. The tests were run on an Intel Core i7-4980HQ with 16\,GB~RAM under macOS 10.12.3 and Mathematica 11.0.1.

\autoref{fig:HandleBody2} and \autoref{fig:Plots} show that the step sizes and tractories along the discrete $H^2$-gradient descent are rather robust under change of mesh resolution, reflecting the fact that the presented descending algorithm is also well-defined in the infinite-dimensional setting.

\clearpage

\section{Conclusion}

In order to keep the presentation as clean as possible, we refrained so far from outlining further generalizations. Still, we deem it worthwhile to draw the reader's attention to the following three observations:

First, the trajectories of the $H^2$-gradient flow starting at $f$ remain in $W^{k+2,p}$ for all finite times where the flow exists provided that $f$ is an immersion of class $W^{k+2,p}$ with $k \geq 0$ and that the boundary $\partial \varSigma$ is of class $C^{k+1,1}$. A similar statement holds true for the iterates of $H^2$-gradient descent.

Second, \autoref{theo:elasticafunctionalisRiesz} applies only to those functionals that depend at most quadratically on the second fundamental form $\II(f)$ and sufficiently smoothly on $f$ and $\dd f$. Even more general functionals and their $H^2$-gradient flows can be considered on $\cC = \Imm^{3,p}(\varSigma;\R^m)$ with the spaces
\begin{align*}
	\RX[\cC][f] = \soboo{3}{p}[\varSigma][\AmbSpace],
	\;
	\RH[\cC][f] = \soboo{2}{2}[\varSigma][\AmbSpace],	
	\;
	\RY[\cC][f] = \soboo{1}{q}[\varSigma][\AmbSpace],
\end{align*}
and the Riesz isomorphisms
\begin{align*}
	\ninnerprod{\RI[\cC]\at_f \,v_1,v_2} = \int_\varSigma \ninnerprod{\Delta_f \,v_1,\Delta_f \,v_2} \, \vol_f
	\qand
	\ninnerprod{\RJ[\cC]\at_f \,u,w} = - \int_\varSigma \ninnerprod{\dd \, \Delta_f \,u,\dd w}_f \, \vol_f.
\end{align*}	

Thirdly, there is an analogue of \autoref{prop:W2pisparaRiemannian}  for the space
\begin{align*}
		\ConfSpace \ceq \set{f \in \Imm^{2,p}(\varSigma;M) | f|_{\partial \varSigma} = \gamma},
\end{align*}		
of immersions into any other Riemannian manifold $(M,g_0)$ without boundary: 
The essential step is in using Sobolev spaces of the form  $\soboo{2}{r}[\varSigma][f^\pull TM]$ for $\RX[\cC][f]$, $\RH[\cC][f]$, and $\RY[\cC][f]$. Here, $f^\pull TM$ is the \emph{pullback vector bundle} along $f$ (see \cite{MR1335233} for the definition of the pullback of a vector bundle). This is a vector bundle of class $\sobo{2}{p} \subset C^1$ over $\varSigma$ and $\sobo{s}{r}[\varSigma][f^\pull TM]$ denotes the Banach space of \emph{sections} of regularity $\sobo{s}{r}$ in it. 
Moreover, the Riesz isomorphisms have to be defined with respect to the covariant Laplacian $\Delta_{g,g_0} \, u = \tr_g (\nabla^{g,g_0} \nabla^{g,g_0} u)$ with $g = f^\pull g_0$.

\bigskip

Finally, we point out that we have primarily numerical applications in mind so that
the spaces chosen here are \emph{not} well-suited for proving existence of minimizers via the direct method of calculus of variations; they were just not constructed for this task.
For a treatment on the direct method for the Willmore energy and on the issues associated with it, we refer to \cite{zbMATH00831712}.
In a nutshell, the Willmore energy of an immersion $f$ provides only few control on the induced Riemannian metric $g_f$ in the sense that even for a minimizing sequence $(f_n)_{n \in \N}$, uniform bounds
\begin{align*}
	\lambda \, \ReferenceMetric \leq g_{f_n} \leq \varLambda \, \ReferenceMetric \quad \text{for all $n \in N$,}
\end{align*}
with respect to a (smooth) reference metric $\ReferenceMetric$ \emph{cannot} be guaranteed.
For the surfaces $M_n \ceq f_n(\varSigma)$, this means that they might degenerate: Even if a limit surface  $M = \lim_{n\to \infty} M_n$ exists (e.g., in Hausdorff distance), it might be of different topological type.
Curiously, we never experienced such a behavior for minimizing sequences generated by our discretized $H^2$-gradient descent.

\clearpage

\appendix

\section{Elliptic Regularity}

We follow closely the exposition on $L^p$-regularity for the Dirichlet problem given in of Sections 9.5 and 9.6 in \cite{MR1814364}, generalizing the results to compact manifolds with Riemannian metrics of class $W^{1,p}$.

\begin{theorem}[Elliptic Regularity]\label{theo:Lrellipticregularity}
Let $m,n \in \N$, $k \in \N \cup\{0\}$, and $r_0$, $r$, $p \in \intervaloo{1,\infty}$ with $p \geq n$ and $r_0 \leq r \leq p$.
Let $\varSigma$ be an $n$-dimensional, compact manifold with boundary of class $C^{k+1,1}$ and
fix a smooth Riemannian metric $G$ as reference and for computing Lebesgue and Sobolev norms.
For given $0< \lambda \leq \varLambda < \infty$ and $R>0$, there is a constant $C_{k,r} = C_{k,r}(\lambda,\varLambda,R,p,G) \geq 0$ such that for each Riemannian metric $g$ of class $W^{k+1,p}$ with $\varLambda^{-1} \,G \leq g \leq \lambda^{-1} \,G$ and $\nnorm{g}_{W^{k+1,p}} \leq R$, the following statement holds true:

For each $u \in \soboo{2}{r}[\varSigma][\AmbSpace]$ 
with $\Delta_g u \in W^{k,r}(\varSigma;\AmbSpace)$ one has
\begin{align*}
	u \in \soboo{k+2}{r}[\varSigma][\AmbSpace] \qand
	\nnorm{u}_{W^{k+2,r}} \leq C_{k,r} \; (\nnorm{u}_{L^r} + \nnorm{\Delta_g u}_{W^{k,r}}).
\end{align*}
\end{theorem}
\begin{proof}
We treat only the case $k=0$ as the situation for $k>0$ follows from this case via standard arguments.
To this end, let $u \in \soboo{2}{r_0}[\varSigma][\AmbSpace]$ with $f \ceq \Delta_g u \in L^r(\varSigma)$.
Since $\varSigma$ is compact, there are finitely many charts $x_\alpha \colon U_\alpha \to  \R^{n-1}\times \intervalco{0,\infty}$ and open sets $\varOmega_\alpha' \Subset  \varOmega_\alpha \Subset U_\alpha$ such that the sets $\varOmega_\alpha' $ cover $\varSigma$.
Thus, it suffices to derive the estimates
\begin{align*}
	\nnorm{u}_{W^{2,r}(\varOmega_\alpha')} \leq C\, (\nnorm{u}_{L^r(\varOmega_\alpha)} + \nnorm{f}_{L^r(\varOmega_\alpha)})
\end{align*}
for these $\varOmega_\alpha'$. Here and for the rest of the proof $\nnorm{\cdot}_{W^{s,r}(\varOmega_\alpha')}$ and $\nnorm{\cdot}_{W^{s,r}(\varOmega_\alpha)}$ denote the corresponding norms with respect to the \emph{Euclidean} metric on $\varOmega_\alpha'$ and $\varOmega_\alpha$, respectively.

To clean up notation, we 
fix one such chart $x \colon U \to \R^{n-1}\times \intervalco{0,\infty}$, together with $\varOmega' \Subset \varOmega \Subset U$.
Within this chart, the Laplace-Beltrami operator in divergence form  reads in Einstein notation as follows:
\begin{align*}
	\Delta_g u =
	\nabs{\det(\textbf{g})}^{-\frac{1}{2}}  \pd_i \, \Big( \nabs{\det(\textbf{g})}^{\frac{1}{2}} \, \textbf{g}^{ij} \pd_j u \Big),
\end{align*}
where $\textbf{g}^{ij}$ denote the entries of the \emph{inverse} of the Gram matrix $\textbf{g}_{ij} \ceq g( \tfrac{\pd}{\pd x_i},\tfrac{\pd}{\pd x_j})$. 
By the Leibniz rule, the local nondivergence form of our equation $\Delta_g u  = f$
on $\varOmega$ reads as
\begin{align*}
	\Delta_g u = a^{ij}\, \pd_i \, \pd_j u  + \beta^j\, \pd_j u = f
\end{align*}
with $a^{ij} \ceq \textbf{g}^{ij}|_{\varOmega} \in C^0(\overline{\varOmega})$, $\beta^j \ceq \nabs{\det(\textbf{g})}^{-\frac{1}{2}}  \pd_i \, \big( \nabs{\det(\textbf{g})}^{\frac{1}{2}} \, \textbf{g}^{ij}\big)|_{\varOmega} \in L^p(\varOmega;\R)$.
We define the operator $L u = a^{ij}\, \pd_i \, \pd_j u$ and observe that $u$ is a strong solution of
\begin{align}
	Lu = f- \beta^j\, \pd_j u
	\quad \text{on $\varOmega$}.
	\label{eq:Luequals}
\end{align}
We use this equation together with Lemma 9.16 from \cite{MR1814364} for an elliptic bootstrapping argument. 
For each $\xi \in (\R^n)'$, we have the inequalities	
\begin{align*}
	c(\varOmega) \, \lambda \, \delta^{ij} \xi_i \xi_j 
	\leq \lambda \, \textbf{G}^{ij} \xi_i \xi_j 
	\leq a^{ij} \xi_i \xi_j 	
	\leq \varLambda \, \textbf{G}^{ij} \xi_i \xi_j 		
	\leq C(\varOmega)\, \varLambda \, \delta^{ij} \xi_i \xi_j,
\end{align*}
where $\textbf{G}$ is the gram matrix of $G$ with respect to the chosen local coordinates.
Hence we see that $a^{ij}$ is elliptic in the sense of Section 9.5 in \cite{MR1814364}.

\emph{Bootstrapping step:} 
Suppose that $u$ is an element of $\soboo{2}{2}[\varrho][\AmbSpace]$ and satisfies
\begin{align*}
	\nnorm{u}_{W^{2,\varrho}} \leq C_{0,\varrho} \, (\nnorm{u}_{L^\varrho} + \nnorm{f}_{L^\varrho})
	\quad
	\text{with some $\varrho \in \nintervaloc{1,p}$.}
\end{align*}
Define
\begin{align*}
	p_1(\varrho) \ceq \begin{cases}
		\frac{n\,\varrho}{n-\varrho}, &1-\frac{n}{\varrho} < 0, \\
		\frac{p (n + p)}{p-n}, &1-\frac{n}{\varrho} = 0, \\
		\infty, &1-\frac{n}{\varrho} > 0.
	\end{cases}
	\qand
	p_2(\varrho) \ceq \begin{cases}
		p, &1- \frac{n}{\varrho} >0, \\
		\frac{n+p}{2}, &1- \frac{n}{\varrho} = 0, \\
\frac{n p \varrho}{n p + n\varrho-p \varrho }, &1- \frac{n}{\varrho} <0.
	\end{cases}
\end{align*}
and observe that $p_1(\varrho)> \varrho$ and $\frac{1}{p} + \frac{1}{p_1(\varrho)} = \frac{1}{p_2(\varrho)}$ hold in any case.
The Sobolev inequality yields
\begin{align*}
	\nnorm{\dd u}_{L^{p_1(\varrho)}} 
	\leq C\, \nnorm{u}_{\sobo{2}{\varrho}} 
	\leq C \, (\nnorm{u}_{L^\varrho} + \nnorm{f}_{L^\varrho}).
\end{align*}
The H\"older inequality implies
\begin{align*}
	\nnorm{\beta^j \pd_j u}_{L^{p_2(\varrho)}(\varOmega)} 
	&\leq C\, \nnorm{\beta}_{L^p(\varOmega)} \, \nnorm{\dd u}_{L^{p_1(\varrho)}(\varOmega)} 
	\leq C \, \nnorm{\beta}_{L^p(\varOmega)} \, (\nnorm{u}_{L^\varrho} + \nnorm{f}_{L^\varrho}).
\end{align*}
This shows that (the restriction to $\varOmega$ of) the right hand side of \eqref{eq:Luequals} lies in $L^{p_3(\varrho)}(\varOmega)$ with $p_3(\varrho) \ceq \min(r, p_2(\varrho))$. 
Now Lemma 9.16 from \cite{MR1814364} provides us with the information $u|_{\varOmega'} \in \sobo{2}{p_3(\varrho)}[\varOmega'][\AmbSpace]$, 
$u|_{\varOmega' \cap \partial \varSigma} = 0$ as trace of a $W^{1,p_3(\varrho)}$-function, and
\begin{align*}
	\nnorm{u}_{\sobo{2}{p_3(\varrho)}[\varOmega']} 
	\leq 
	C_{0,p_3(\varrho)}(\varOmega,\varOmega') \, (\nnorm{u}_{L^{p_3(\varrho)}(\varOmega)} + \nnorm{f}_{L^{p_3(\varrho)}(\varOmega)}).
\end{align*}
Collecting this information from the finitely many chosen charts leads to
$u\in \soboo{2}{p_3(\varrho)}[\varSigma][\AmbSpace]$ and the global regularity estimate
\begin{align*}
	\nnorm{u}_{\sobo{2}{p_3(\varrho)}} 
	\leq 
	C_{0,p_3(\varrho)} \; (\nnorm{u}_{L^{p_3(\varrho)}} + \nnorm{f}_{L^{p_3(\varrho)}}).
\end{align*}

\emph{Bootstrapping conclusion:}
One has for all $\varrho \geq n$ that $p_2(\varrho)>n$ so that $p_3(\varrho) = \min(r,p) = r$ 
holds.
For $\varrho \in \intervaloo{1,n}$, the function $\varphi(\varrho) = \frac{n p \varrho}{n p + n\varrho-p \varrho }-\varrho$ is strictly monotonically increasing and its infimum is given by
$\varphi(1)=\frac{p-n}{n + p (n-1) p}>0$. Thus, we have $p_3(\varrho) = \min(r,\varphi(\varrho) + \varrho) \geq \min(r,\frac{p-n}{n + p (n-1) p}+\varrho)$ and we arrive at $p_3 \circ \dotsm \circ p_3(\varrho) \geq \min(r,n)$ after at most finitely many bootstrapping steps.
After a further bootstrapping step, we finally arrive at
\begin{align*}
	u \in \soboo{2}{r}[\varSigma][\AmbSpace]
	\qand
	\nnorm{u}_{\sobo{2}{r}} \leq C_{0,r} \, (\nnorm{u}_{L^{r}} + \nnorm{f}_{L^{r}}).
\end{align*}
\end{proof}

\begin{definition}\label{dfn:transversalspaces}
Let $k \in \N\cup\{0\}$ and $r \in \intervaloo{1,\infty}$. Let $\varSigma$ be a compact manifold with boundary of class $C^{k+1,1}$ and let $g$ be a Riemannian metric of class $W^{k+1,p}$. If $\varSigma$ is connected, we define the spaces
\begin{align*}
	W^{k+2,r}_{\averagedom,g}(\varSigma;\AmbSpace)
	&\ceq \begin{cases}
		\soboo{k+2}{r}[\varSigma][\AmbSpace]
	 	& \partial\varSigma \neq \emptyset,\\
	 	\Set{u \in W^{k+2,r}(\varSigma;\AmbSpace) | \int_\varSigma u \, \vol_g =0},
	 	& \partial\varSigma = \emptyset,
	 \end{cases}
	 \qand
	 \\
	W^{k,r}_{\averageima,g}(\varSigma;\AmbSpace)
	&\ceq \begin{cases}
		W^{k,r}(\varSigma;\AmbSpace),
	 	& \partial\varSigma \neq \emptyset,\\
	 	\Set{u \in W^{k,r}(\varSigma;\AmbSpace) | \int_\varSigma u \, \vol_g =0},
	 	& \partial\varSigma = \emptyset.
	 \end{cases}	 
\end{align*}
If $\varSigma$ is not connected, we define
\begin{align*}
	W^{k+2,r}_{\averagedom,g}(\varSigma;\AmbSpace)
	\ceq \bigoplus_{\varOmega \in \ConnComp(\varSigma)} W^{k+2,r}_{\averagedom,g}(\varOmega;\AmbSpace)
	\qand
	W^{k,r}_{\averageima,g}(\varSigma;\AmbSpace)
	\ceq \bigoplus_{\varOmega \in \ConnComp(\varSigma)} W^{k,r}_{\averageima,g}(\varOmega;\AmbSpace),
\end{align*}
where the direct sums run over the finitely many connected components of $\varSigma$.	
\end{definition}

\begin{lemma}[Closed Range]\label{lem:closedrange}
Let  $\varSigma$ be an $n$-dimensional, compact manifold with boundary of class $C^{k+1,1}$, let $g$ be a Riemannian metric of class $W^{k+1,p}$ with $k \in \N\cup \{0\}$ and $p>n$, and let $r \in \nintervaloc{1,p}$.
For each smooth Riemannian metric $G$ on $\varSigma$ there is a constant $C_{k,r}(G) \geq 0$ such that
\begin{align}
		\nnorm{u}_{W^{k+2,r}} \leq C_{k,r}(G) \, \nnorm{ \Delta_g u}_{W^{k,r}}
		\quad
		\text{holds for all $u \in W^{k+2,r}_{\averagedom,g}(\varSigma;\AmbSpace)$.}
		\label{eq:closednessinequality}
\end{align}
\end{lemma}
\begin{proof}[By contradiction.] 
\emph{Assume} that \eqref{eq:closednessinequality} is false. Then there is a sequence $(u_\alpha)_{\alpha\in\N}$ in $W^{k+2,r}_{\averagedom,g}(\varSigma;\AmbSpace)$ with $\nnorm{\Delta_g u_\alpha}_{L^{r}} \to 0$ while $\nnorm{u_\alpha}_{W^{k+2,r}} \not \to 0$ as $\alpha \to \infty$. 
By normalization (and by choosing a subsequence if necessary), we may assume that $\nnorm{u_\alpha}_{L^{r}} =1$.
\autoref{theo:Lrellipticregularity} implies that $(u_\alpha)_{\alpha\in\N}$ is bounded in $W^{k+2,r}(\varSigma;\AmbSpace)$.
By Rellich compactness and the Banach-Alaoglu theorem, there is $u \in W^{k+2,r}_{\averagedom,g}(\varSigma;\AmbSpace)$ such that
$u_\alpha \to u$ in $\soboo{1}{r}[\varSigma][\AmbSpace]$ and $u_\alpha \rightharpoonup u$ in $W^{k
+2,r}(\varSigma;\AmbSpace)$ with $\nnorm{u}_{L^{r}} =1$ and $\Delta_g u =0$. 
By elliptic regularity (see \autoref{theo:Lrellipticregularity}), we have at least $u \in
\soboo{2}{2}[\varSigma][\AmbSpace]$ so that we obtain
\begin{align}
	\int_\varSigma \nabs{\dd u}_g^2 \, \vol_g 
	= 
	\int_\varSigma \ninnerprod{-\Delta_g u, u}_g \, \vol_g =0.
	\label{eq:selftest}
\end{align}
This shows $\dd u = 0$ and together with $u \in \soboo{1}{2}[\varSigma][\AmbSpace]$ we obtain $u=0$ which is a \emph{contradiction} to $\nnorm{u}_{L^{r}} =1$.
\end{proof}

\begin{lemma}\label{lem:Laplaceinvertible}
Under the conditions of \autoref{lem:closedrange}, the operator
\begin{align}
	-\Delta_g \colon W^{k+2,r}_{\averagedom,g}(\varSigma;\AmbSpace) 
		\to 
	W^{k,r}_{\averageima,g}(\varSigma;\AmbSpace)
	\label{eq:Deltaginvertible}
\end{align}
is an isomorphism of Banach spaces.
\end{lemma}
\begin{proof}
For each connected component $\varOmega$ of $\varSigma$ with $\partial \varOmega = \emptyset$ we have with its characteristic function $\chi_\varOmega$:
\begin{align*}
	\int_\varOmega (-\Delta_g \, u) \, \vol_g
	= \int_\varSigma (-\Delta_g \, u) \, \chi_\varOmega \, \vol_g
	= \int_\varSigma \ninnerprod{\dd u , \dd \chi_\varOmega}\, \vol_g = 0.
\end{align*}
This shows that $-\Delta_g(W^{k+2,r}_{\averagedom,g}(\varSigma;\AmbSpace) ) \subset W^{k,r}_{\averageima,g}(\varSigma;\AmbSpace)$.
By \autoref{lem:closedrange}, the operator in \eqref{eq:Deltaginvertible} is injective with closed range. Thus, we merely have to show that it is surjective.
To this end, we fix a smooth Riemannian metric $G$ on $\varSigma$ and an element $\varphi \in W^{k,r}_{\averageima,g}(\varSigma;\AmbSpace)$.
Let $(g_\alpha)_{\alpha \in \N}$  be a sequence of smooth Riemannian metrics that  converge to $g$ in the $W^{1,p}$-norm. This way, there are $0<\lambda < \varLambda< \infty$ with
\begin{align*}
	\varLambda^{-1} \, G \leq g \leq \lambda^{-1} \,G
	\qand
	\varLambda^{-1} \, G \leq g_\alpha \leq \lambda^{-1} \,G
	\quad
	\text{for all $\alpha \in \N$.}
\end{align*}
Moreover, let $(\varphi_\alpha)_{\alpha \in \N}$ be a sequence of smooth functions that converges in the $W^{k,r}$-norm to $\varphi$. By substracting $(\int_\varSigma \vol_{g_\alpha})^{-1} \int_\varSigma \varphi_\alpha \, \vol_{g_\alpha}$ if necessary, we may assume that $\varphi_\alpha \in W^{k,r}_{\averageima,g_\alpha}(\varSigma;\AmbSpace)$.

Now we have $\varphi_\alpha \in W^{0,2}_{\averageima,g_\alpha}(\varSigma;\AmbSpace)$ and standard theory shows that the equation
$-\Delta_{g_\alpha} u_\alpha = \varphi_\alpha$ has a unique solution in $W^{2,2}_{\averagedom,g_\alpha}(\varSigma;\AmbSpace)$.
By \autoref{theo:Lrellipticregularity}, the sequence $(u_\alpha)_{\alpha \in \N}$ is bounded in $W^{k+2,r}(\varSigma;\AmbSpace)$. Rellich compactness and the Banach-Alaoglu theorem imply that there is a $u \in W^{k+2,r}(\varSigma;\AmbSpace)$ with $u_\alpha \to u$ in $W^{1,r}(\varSigma;\AmbSpace)$ and $u_\alpha \rightharpoonup u$ in $W^{k+2,r}(\varSigma;\AmbSpace)$.
The strong convergence in $W^{1,r}(\varSigma;\AmbSpace)$ implies that $u$ is actually an element of $W^{k+2,r}_{\averagedom,g}(\varSigma;\AmbSpace)$.
We have to show that $-\Delta_g u = \varphi$. Therefore, we consider the following consequence of Green's second identity:
\begin{align}
	\int_\varSigma \ninnerprod{\varphi_\alpha, \psi}\, \vol_{g_\alpha}
	= 
	\int_\varSigma \ninnerprod{-\Delta_{g_\alpha} u_\alpha, \psi}\, \vol_{g_\alpha}
	&=
	\int_\varSigma \ninnerprod{u_\alpha, -\Delta_{g_\alpha} \psi}\, \vol_{g_\alpha}
	\quad
	\text{for all $\psi \in C^\infty_0(\varSigma)$}.
	\label{eq:Fredholmsurjective}
\end{align}
On the one hand, $\varphi_\alpha \, \vol_{g_\alpha}$ converges to $\varphi\, \vol_g$ in the $L^r$-space of vector-valued densities, hence the left-hand side of \eqref{eq:Fredholmsurjective} converges towards $\int_\varSigma \ninnerprod{\varphi, \psi}\, \vol_{g}$.
On the one hand, $-\Delta_{g_\alpha} \, \psi\, \vol_{g_n}$ converges to $-\Delta_{g} \, \psi \, \vol_{g}$ in the $L^p$-space of vector-valued densities so that 
the right-hand side of \eqref{eq:Fredholmsurjective} has $\int_\varSigma \ninnerprod{u, -\Delta_{g} \psi}\, \vol_{g} = \int_\varSigma \ninnerprod{-\Delta_{g} u, \psi}\, \vol_{g}$ as its limit. Now the fundamental lemma of the calculus of variations leads to $-\Delta_g \, u = \varphi$, showing that $-\Delta_g$ is surjective.
\end{proof}

\begin{theorem}\label{theo:Fredholm}
Let  $\varSigma$ be an $n$-dimensional, compact manifold with boundary of class $C^{k+1,1}$, let $g$ be a Riemannian metric of class $W^{k+1,p}$ with $k \in \N\cup \{0\}$ and $p>n$.
Then for each $r \in \nintervaloc{1,p}$ the operator
\begin{align*}
	(-\Delta_g , \res) \colon W^{k+2,r}(\varSigma;\AmbSpace) 
		\to 
	W^{k,r}(\varSigma;\AmbSpace)
	\oplus
	W^{k+2-\frac{1}{r},r}(\partial \varSigma;\AmbSpace),
	\quad
	u \mapsto (-\Delta_g , u|_{\partial \varSigma})
\end{align*}
is a Fredholm operator of index $0$.
\end{theorem}
\begin{proof}
We may treat the finitely many connected components of $\varSigma$ independently. Thus we may suppose without loss of generality that $\varSigma$ is connected so that we have to distinguish only two cases.
\newline
\textbf{Case I:} $\partial \varSigma \neq \emptyset$.
Actually, we show that
\begin{align*}
	(-\Delta_g , \res) \colon W^{k+2,r}(\varSigma;\AmbSpace) 
		\to 
	W^{k,r}(\varSigma;\AmbSpace)
	\oplus
	W^{k+2-\frac{1}{r},r}(\partial \varSigma;\AmbSpace)
\end{align*}
is an isomorphism. Via the existence of a continuous right inverse
\begin{align*}
	\ext \colon W^{k+2-\frac{1}{r},r}(\partial \varSigma;\AmbSpace) \to W^{k+2,r}(\varSigma;\AmbSpace)
\end{align*}
of $\res$, this is equivalent to 
\begin{align*}
	-\Delta_g \colon 
	\soboo{k+2}{r}[\varSigma][\AmbSpace]
	\to 
	W^{k,r}(\varSigma; \AmbSpace)
\end{align*}
being an isomorphism. This has already been shown in \autoref{lem:Laplaceinvertible}.
\newline
\textbf{Case II:} $\partial \varSigma = \emptyset$.
Denote by $X_0 \cong \AmbSpace$ the constant mappings from $\varSigma$ to $\AmbSpace$ and observe that the splittings
\begin{align*}
	W^{k+2,r}(\varSigma;\AmbSpace) = W^{k+2,r}_{\averagedom,g}(\varSigma;\AmbSpace) \oplus X_0
	\qand
	W^{k,r}(\varSigma;\AmbSpace) = W^{k,r}_{\averageima,g}(\varSigma;\AmbSpace) \oplus X_0	
\end{align*}
are orthogonal with respect to the $L^2$-inner product induced by $g$.
For $u \in W^{k+2,r}(\varSigma;\AmbSpace)$ with $-\Delta_g u =0$, we have by elliptic regularity (see \autoref{theo:Lrellipticregularity}) that $u \in W^{2,2}(\varSigma;\AmbSpace)$ so that the same calculation as in \eqref{eq:selftest} leads to $\dd u =0$.
Hence we have $u \in X_0$ and $\ker(-\Delta_g) = X_0 \cong \AmbSpace$.
In \autoref{lem:Laplaceinvertible}, we have shown that 
$-\Delta_g ( W^{k+2,r}_{\averagedom,g}(\varSigma;\AmbSpace)) = W^{k,r}_{\averageima,g}(\varSigma;\AmbSpace)$, hence $\coker(-\Delta_g) \cong X_0 \cong \AmbSpace$.
\end{proof}

\section{Multiplication Lemma}

\begin{lemma}\label{lem:regularityofproducts}
Let $E_1$, $E_2$, and $E_3$ be smooth vector bundles over the compact, smooth manifold~$\varSigma$, let
$\mu \colon E_1 \times_\varSigma E_2 \to E_3$ be a locally Lipschitz continuous bilinear bundle map and let $p > \dim(\varSigma)$.

Then  $\mu\circ(\sigma_1,\sigma_2) \in W^{1,\min(p,r)}(\varSigma;E_3)$ holds for all sections
$\sigma_1 \in W^{1,p}(\varSigma;E_1)$ and $\sigma_2 \in W^{1,r}(\varSigma;E_2)$.
Moreover, the induced bilinear map
\begin{align*}
	A \colon W^{1,p}(\varSigma;E_1) \times W^{1,r}(\varSigma;E_2) \to W^{1,\min(p,r)}(\varSigma;E_1)
\end{align*}
is continuous.
\end{lemma}
\begin{proof}
It suffices to perform the regularity analysis locally. Thus, we may focus our attention to an open set $U \subset \varSigma$ and we may assume for each $i \in \set{1,2,3}$ that $E_i|U \cong U \times X_i$ is a trivial Banach bundle with a suitable Banach space $X_i$. Moreover, we may write $\sigma_1(x) = (x, f_1(x))$, $\sigma_2(x) = (x, f_2(x))$, and $\mu_x = B_x$ for all $x \in U$ with $f_1 \in W^{1,p}(\varSigma;X_1)$, $f_2 \in W^{1,r}(\varSigma;X_2)$, and $B \in W^{1,\infty}(U ;\Bil(X_1,X_2;X_3))$,
where $\Bil(X_1 , X_2;X_3)$ denotes the Banach space of continuous bilinear forms on $X_1 \times X_2$ with values in $X_3$.

The Sobolev embedding $W^{1,p}(\varSigma;X_1) \hookrightarrow L^{\infty}(\varSigma;X_1)$
shows that $B(f_1,f_2) \in L^r(\varSigma;X_3)$.
With $n \ceq \dim(\varSigma)$, one has the Sobolev embedding $W^{1,r}(\varSigma;X_2) \hookrightarrow L^{\bar r}(\varSigma;X_2)$ where
\begin{align*}
		\bar r \in 
	\begin{cases}
		\nintervalcc{1,\tfrac{n\,r}{n - r}}, &r <n,
		\\
		\intervalco{1,\infty}, & r = n,
		\\
		\intervalcc{1,\infty}, & r > n.
	\end{cases}
\end{align*}
For each smooth vector field $Y$ on $U$, we obtain
\begin{align*}
	\dd( B(f_1 , f_2)) \, Y
	&= (\dd B\,Y)(f_1 , f_2) + B(\dd f_1 \,Y,f_2(x)) + B(f_1 , \dd f_2 \,Y)
\end{align*}
and this together with the H\"{o}lder inequality implies $\dd (B(f_1 , f_2)) \in L^{\min(s,r)}(\varSigma;\Hom(T'\varSigma;X_3))$, hence $B(f_1 , f_2) \in W^{1,\min(s,r)}(\varSigma;X_3)$,
where $s= \left(\frac{1}{p}+\frac{1}{\bar r}\right)^{-1}$.
We analyse the following three cases:
\newline
\emph{Case 1.: $r<n$.} Because of $p>n>r$, we have
$
	s 
	= \big(\tfrac{1}{p} + \tfrac{1}{r} - \tfrac{1}{n}\big)^{-1}
	> \big(\tfrac{1}{n} + \tfrac{1}{r} - \tfrac{1}{n}\big)^{-1}
	= r,
$
so that $\min(s,r) = r = \min(p,r)$.
\newline
\emph{Case 2.: $r=n$.} One may write $r=n= (1 + \varepsilon)^{-1} p<p$ with some $\varepsilon>0$. Choosing $\bar r = \frac{p}{\varepsilon}<\infty$, we obtain
$
	s 
	= \big(\frac{1}{p} + \tfrac{\varepsilon}{p}\big)^{-1} = \tfrac{p}{1 + \varepsilon}
	=r.
$
This shows $\min(s,r) = r = \min(p,r)$.
\newline
\emph{Case 3.: $r>n$.} Then one has $\bar r = \infty$ and $s=p$, leading directly to 
$\min(s,r) = \min(p,r)$.
Finally, the continuity of $A$ follows from the already mentioned H\"{o}lder and Sobolev inequalities.	
\end{proof}

\newpage
\printbibliography

\end{document}